\pdfoutput=1
\documentclass{amsart}

\usepackage{amssymb}
\usepackage{braket}
\usepackage{mathrsfs}
\usepackage{ifthen}
\usepackage{here}
\usepackage{todonotes}
\usepackage{tikz}
\usetikzlibrary{patterns,decorations.pathreplacing,calligraphy}
\usepackage{comment} 
\usepackage{mleftright}
\usepackage[pagebackref,hypertexnames=false]{hyperref} 
\usepackage{cleveref}
 \usepackage[all]{xy}
 \usepackage{amscd}
 \usepackage[alphabetic,backrefs,msc-links]{amsrefs}
 \usepackage{color}
 \usepackage{enumitem}
\newlist{steps}{enumerate}{1}
\setlist[steps, 1]{label = Step \arabic*:}
\usepackage[abs]{overpic}
\usepackage{tikz-cd}
\usepackage{pinlabel}
\usepackage{amsmath, amsthm,verbatim,amsfonts, graphicx, enumerate}

\newcounter{casenum}
\newenvironment{caseof}{\setcounter{casenum}{1}}{\vskip.5\baselineskip}
\newcommand{\case}[2]{\vskip.5\baselineskip\par\noindent {\bfseries Case \arabic{casenum}:} #1\\#2\addtocounter{casenum}{1}}

\usepackage{geometry}
\makeatletter
\DeclareRobustCommand\widecheck[1]{{\mathpalette\@widecheck{#1}}}
\def\@widecheck#1#2{%
   \setbox\z@\hbox{\m@th$#1#2$}%
   \setbox\tw@\hbox{\m@th$#1%
      {%
         \vrule\@width\z@\@height\ht\z@
         \vrule\@height\z@\@width\wd\z@}$}%
   \dp\tw@-\ht\z@
   \@tempdima\ht\z@ \advance\@tempdima2\ht\tw@ \divide\@tempdima\thr@@
   \setbox\tw@\hbox{%
      \raise\@tempdima\hbox{\scalebox{1}[-1]{\lower\@tempdima\box\tw@}}}%
   {\ooalign{\box\tw@ \cr \box\z@}}}
\makeatother

\theoremstyle{plain}
\newtheorem{thm}{Theorem}[section]
\crefname{thm}{Theorem}{Theorems}
\Crefname{thm}{Theorem}{Theorems}
\newtheorem{prop}[thm]{Proposition}
\crefname{prop}{Proposition}{Propositions}
\Crefname{prop}{Proposition}{Propositions}
\newtheorem{lem}[thm]{Lemma}
\crefname{lem}{Lemma}{Lemmas}
\Crefname{lem}{Lemma}{Lemmas}
\newtheorem{cor}[thm]{Corollary}
\crefname{cor}{Corollary}{Corollaries}
\Crefname{cor}{Corollary}{Corollaries}

\crefname{claim}{Claim}{Claims}
\Crefname{claim}{Claim}{Claims}

\crefname{property}{Property}{Properties}
\Crefname{property}{Property}{Properties}

\crefname{problem}{Problem}{Problems}
\Crefname{problem}{Problem}{Problems}

\crefname{conjecture}{Conjecture}{Conjecture}
\Crefname{conjecture}{Conjecture}{Conjecture}

\theoremstyle{definition}

\crefname{defn}{Definition}{Definitions}
\Crefname{defn}{Definition}{Definitions}

\crefname{notation}{Notation}{Notations}
\Crefname{notation}{Notation}{Notations}

\crefname{convention}{Convention}{Conventions}
\Crefname{convention}{Convention}{Conventions}

\crefname{cond}{Condition}{Conditions}
\Crefname{cond}{Condition}{Conditions}

\crefname{assum}{Assumption}{Assumptions}
\Crefname{assum}{Assumption}{Assumptions}

\crefname{conj}{Conjecture}{Conjectures}
\Crefname{conj}{Conjecture}{Conjectures}

\crefname{claim1}{Claim}{Claims}
\Crefname{claim1}{Claim}{Claims}
\Crefname{ques}{Question}{Question}

\crefname{que}{Question}{Question}
\Crefname{que}{Question}{Question}

\theoremstyle{remark}
\newtheorem{rem}[thm]{Remark}
\crefname{rem}{Remark}{Remarks}
\Crefname{rem}{Remark}{Remarks}
\newtheorem{ex}[thm]{Example}
\crefname{ex}{Example}{Examples}
\Crefname{ex}{Example}{Examples}

\crefname{section}{Section}{Sections}
\Crefname{section}{Section}{Sections}
\crefname{subsection}{Subsection}{Subsections}
\Crefname{subsection}{Subsection}{Subsections}
\crefname{figure}{Figure}{Figures}
\Crefname{figure}{Figure}{Figures}

\newcommand{\Z}{\mathbb{Z}}

\newcommand{\Q}{\mathbb{Q}}

\newcommand{\fraks}{\mathfrak{s}}
\newcommand{\frakt}{\mathfrak{t}}

\newcommand{\id}{\mathrm{id}}

\newcommand{\C}{\mathbb{C}}
\newcommand{\s}{\mathfrak{s}}

\def\Om{\Omega}

\newcommand{\N}{\mathbb N}
\newcommand{\R}{\mathbb R}

\def\ker{\operatorname{Ker}}

\def\dim{\operatorname{dim}}

\def\id{\operatorname{Id}}

\newcommand{\mbar}[1]{{\ooalign{\hfil#1\hfil\crcr\raise.167ex\hbox{--}}}}

\def\wt{\widetilde}

\newcommand*{\QEDB}{\null\nobreak\hfill\ensuremath{\square}}%

     \RequirePackage{rotating}                   
    \def\HMt{%
       \setbox0=\hbox{$\widehat{\mathit{HM}}$}
       \setbox1=\hbox{$\mathit{HM}$}
       \dimen0=1.1\ht0
       \advance\dimen0 by 1.17\ht1
       \smash{\mskip2mu\raise\dimen0\rlap{%
          \begin{turn}{180}
              {$\widehat{\phantom{\mathit{HM}}}$}
           \end{turn}} \mskip-2mu    
                \mathit{HM}
                    }{\vphantom{\widehat{\mathit{HM}}}}{}}

\title{Cables of the figure-eight knot via real Fr\o yshov invariants}

\author{Sungkyung Kang}
\address{Mathematical Institute, University of Oxford, United Kingdom}
\email{sungkyung38@icloud.com}

\author{JungHwan Park}
\address{Department of Mathematical Sciences, KAIST, Republic of Korea}
\email{jungpark0817@kaist.ac.kr}

\author{Masaki Taniguchi} 
\address{Department of Mathematics, Kyoto University, Japan}
\email{taniguchi.masaki.7m@kyoto-u.ac.jp}

\begin{document}

\begin{abstract}
We prove that the $(2n,1)$-cable of the figure-eight knot is not smoothly slice when $n$ is odd, by using the real Seiberg--Witten Fr\o yshov invariant of Konno--Miyazawa--Taniguchi. For the computation, we develop an $O(2)$-equivariant version of the lattice homotopy type, originally introduced by Dai--Sasahira--Stoffregen. This enables us to compute the real Seiberg--Witten Floer homotopy type for a certain class of knots. Additionally, we present some computations of Miyazawa's real framed Seiberg--Witten invariant for $2$-knots.
\end{abstract}

\maketitle

\section{Introduction}

Casson and Gordon~\cite[Theorem 5.1]{Casson-Gordon:1983-1} proved that a fibered knot in a homology sphere is homotopically ribbon if and only if its closed monodromy extends over a handlebody. Utilizing this characterization, Miyazaki~\cite{Miyazaki:1994-1} constructed a large family of fibered knots and proved that each knot in this family is not ribbon. Within this family, there are two important sets of knots: the first one~\cite[Example~1]{Miyazaki:1994-1} is the set of nontrivial connected sums of iterated torus knots. The first set is related to Rudolph's conjecture~\cite{Rudolph:1976-1}, which asserts that the set of algebraic knots is linearly independent in the smooth knot concordance group (see \cites{Litherland:1984-1, Hedden-Kirk-Livingston:2012-1, Abe-Tagami:2016-1, Baker:2016-1, CKP:2023-1} for related results).

The other set~\cite[Example 2]{Miyazaki:1994-1} consists of the $(2n,1)$-cables of fibered negative-amphiciral knots with irreducible Alexander polynomial.\footnote{For the rest of the cables, it can be verified that they are not algebraically slice using Tristram-Levine Signatures~\cites{Tristram:1969-1, Levine:1969-1} and \cite{Kawauchi:1980-1} (see also \cite[Theorem 6]{CLR:2008-1}).} These knots are known to be algebraically slice and strongly rationally slice~\cites{kawauchi1, Cha:2007-1, Kim-Wu:2018-1}.\footnote{See \cite[Definition 2]{Kim-Wu:2018-1} for the precise definition of strongly rationally slice knots.} While these knots attracted considerable attention due to their relation to the slice-ribbon conjecture~\cite[Problem 25]{Fox:1962-1}, no proof of nonsliceness had been established for them until recently. In~\cite[Theorem 1.1]{DKMPS:2022-1} (see also \cite[Theorem 2.1]{ACMPS:2023-1} and \cite[Corollary 1.20]{KMT23f}), Dai, Kang, Mallick, Park, and Stoffregen proved that the simplest case--the $(2,1)$-cable of the figure-eight knot--is not smoothly slice. In fact, they show that a $(2,1)$-cable of a Floer-thin knot with nonvanishing Arf invariant has infinite order in the smooth concordance group. In this article, we consider $(2n,1)$-cables in general and obtain the following:

\begin{thm}\label{thm:2n1of41}
Let $E$ be the figure-eight knot, and let $E_{2n,1}$ denote the $(2n,1)$-cable of $E$. For each positive odd integer $n$, the knot $E_{2n,1}$ does not bound a normally immersed disk in $B^4$ with only negative double points. In particular, for each odd integer $n$, the knot $E_{2n,1}$ is not smoothly slice. 
\end{thm}



Here, we say a surface is \emph{normally immersed} if it is smoothly immersed in a manifold such that the only singularities are transverse double points in the interior of the surface. Recall that the \emph{4-dimensional clasp number} $c_4(K)$ of a knot $K$~\cite{Shibuya:1974-1} is the minimal number of double points in a normally immersed disk in $B^4$ bounded by $K$. A refinement $c_4^+(K)$, considered for example in \cites{Daemi-Scaduto:2020-1, Juhasz-Zemke:2020-1, Feller-Park:2022-1, Miller:2022-1, Livingston:2022-1}, is the minimal number of \emph{positive} double points in such a normally immersed disk. With this terminology, the main theorem can be compactly stated as $0 < c_4^+(E_{2n,1})$ for each positive odd integer $n$. Since a smoothly slice knot has vanishing $c_4^+$, the theorem is a strict improvement over previous results, even for the case $n=1$. 

Note that the figure-eight knot $E$ can be transformed into the unknot by changing a negative crossing to a positive one. This implies that $E$ bounds a normally immersed disk with only one negative double point, and $E_{2n,1}$ bounds a normally immersed disk with $2n$ positive double points and $4n^2$ negative double points in $B^4$. For the special case $E_{2,1}$, with some extra consideration, one can find two crossing changes, one from positive to negative and one from negative to positive, that turn $E_{2,1}$ into a smoothly slice knot, which in particular implies that $c_4^+(E_{2,1}) = 1$ (see Remark~\ref{rmk:c4(4121)}). Determining $c_4(E_{2n,1})$ and $c_4^+(E_{2n,1})$ in general seems to be an interesting yet challenging problem.

Our proof shows that for each odd integer $n$, the double-branched cover of $E_{2n,1}$ does not bound a 4-manifold $W$ with the following properties:
\begin{itemize}
    \item $W$ is a smooth spin 4-manifold with a spin stucture $\s$,
    \item $\tau \colon W \to W$ is a smooth involution such that $\tau|_{\partial W}$ is the deck transformation and $\tau ^* \fraks \cong \fraks$, and
    \item $b_1(W)=0$,  $b_2^+(W)-b_2^+(W/\tau)=0$, and $\sigma(W)\leq 0$.
\end{itemize}
In particular, the double-branched cover does not bound an equivariant $\Z_2$-homology ball; that is, a $\Z_2$-homology ball over which the branching involution extends as a smooth involution. From the nonexistence of such a spin 4-manifold filling of the branched cover, we can further conclude that the knot $E_{2n,1}$ does not bound a normally immersed disk with only negative double points in any $\Z_2$-homology ball.

The topological input to the theorem is the existence of a smooth concordance from the figure-eight knot to the unknot in a  twice-punctured $2\mathbb{CP}^2$, denoted by $X$, that represents $(1,3)$ in $H_2(X, \partial X; \mathbb{Z}) \cong \mathbb{Z} \oplus \mathbb{Z}$, as proved by Aceto, Castro, Miller, Park, and Stipsicz in~\cite{ACMPS:2023-1}. Our obstruction applies to all knots that permit such a concordance to a smoothly slice knot, which is the case for \cite[Theorem 2.3]{ACMPS:2023-1} as well.


\begin{thm}\label{thm:main}
Let $K$ be a knot, and let $K_{2n,1}$ denote the $(2n,1)$-cable of $K$. Suppose that $K$ can be transformed into a slice knot by applying full negative twists along two disjoint disks, where one intersects $K$ algebraically once and the other intersects it algebraically three times. Then, for each positive odd integer $n$, the knot $K_{2n,1}$ does not bound a normally immersed disk in $B^4$ with only negative double points.
\end{thm}



There are infinitely many knots that satisfy the assumptions of Theorem~\ref{thm:main}. In fact, \cite[Remark 2.6]{ACMPS:2023-1} provides an infinite family of strongly negative-amphichiral knots  meeting the assumptions.
Recall that a knot $K$ is called \emph{strongly negative-amphichiral} if there is an orientation-reversing involution $\tau \colon S^3 \to S^3$ such that $\tau(K)=K$. Since each knot in the family is strongly negative-amphichiral, the $(2n,1)$-cables of these knots are algebraically slice and strongly rationally slice~\cite{Kawauchi:2009-1}. In particular, the usual concordance invariants from knot Floer homology~\cite{ozsvath2004holomorphicknot, Rasmussen:2003-1} (cf. \cite{Hom:2017-1, HKPS:2022-1}) and the concordance invariants $s^\#$, $f_\sigma$, $\tau^\#$, $\nu^\#$, $\tau_I$ \footnote{As it is pointed out in \cite[Remark 1.6]{BS22}, the invariants $\tau^\#$ and $\nu^\#$ vanish for rationally slice knots. In particular, from \cite[Theorem 1.2]{GLW19}, $\tau_I$ also vanishes. 
}, $\wt{s}$, $\Gamma$, $r_s$ from instanton knot Floer theory \cite{KM13, DS19, GLW19, KM21,BS21,  DISST22}, and the concordance invariants $\theta^p$, $q_M$ from equivariant Seiberg--Witten theory \cite{BH21, BH22, IT24} vanish. Moreover, it can also be proved that the $s$-invariant~\cite{Ra10} from Khovanov homology~\cite{Khovanov:2000-1} vanishes (cf. \cite{MMSW:2023-1}). Additionally, we note that the $(2,1)$-cable (i.e., when $n=1$) is the only case where \cite{DKMPS:2022-1, ACMPS:2023-1, KMT23f} can be directly applied.

 Our main tools are the real Fr\o yshov inequalities involving the three concordance invariants
\[
\delta_R(K), \underline{\delta}_R(K), \text{ and  } \overline{\delta}_R (K) \in \frac{1}{16}\Z
\]
which are called {\it real Fr\o yshov invariants}, introduced by Konno, Miyazawa, and the third author in \cite{KMT:2023}. The invariants are defined as certain Fr\o yshov type invariants for the fixed point spectrum of an order $2$ subgroup $\langle I \rangle$ in $O(2)$, acting on the Manolescu's Seiberg--Witten Floer homotopy type \cite{Ma03} of the double-branched cover of a knot $K$:
\[
SWF_R(K) := (SWF(\Sigma_2(K), \fraks_0))^I,
\]
where $\fraks_0$ is the unique spin structure on the double-branched cover $\Sigma_2(K)$.  Note that $SWF_R(K)$ has a $\Z_4$-symmetry, which comes from the $j$-action in $\operatorname{Pin}(2)$. The invariant $\delta_R(K)$ is a $\Z_2$-equivariant Fr\o yshov invariant, which can be seen as an analog of the Heegaard Floer $d$-invariant~\cite{OzSz03b}. The latter two invariants, $\underline{\delta}_R(K)$ and $\overline{\delta}_R(K)$, are $\Z_4$-equivariant Fr\o yshov invariants similar to $\underline{d}$ and $\overline{d}$ in involutive Heegaard Floer theory~\cite{Hendricks-Manolescu:2017-1}. There are several variants of real Seiberg--Witten theory; for examples, see \cite{TW09, Na13, Nak15, Ka22, KMT21, Ji22, KMT:2023, Mi23, Li23, BH24}.

To prove Theorem~\ref{thm:main}, we shall show that if a knot $K$ satisfies the assumptions of the theorem, then $\underline{\delta}_R(K_{2n,1}) < 0$ for each odd $n$. To accomplish this, we make use of a smooth concordance from $K_{2n,1}$ to the torus knot $T_{2n, 1-20n}$ in a twice-punctured $2\mathbb{CP}^2$. This approach simplifies the calculation of $\underline{\delta}_R(K_{2n,1})$ to calculating $\overline{\delta}_R(T_{2n, 1-20n})$. For the computation of $\overline{\delta}_R(T_{2n, 1-20n})$, we develop a theory of the $O(2)$-homotopy type of the Seiberg--Witten Floer spectrum, which we describe below.

We introduce a method to compute both the real and the $O(2)$-equivariant Seiberg--Witten Floer homotopy type for an almost-rational plumbed homology sphere. Our main tool is based on the $\operatorname{Pin}(2)$-equivariant lattice homotopy type, developed by Dai, Stoffregen, and Sasahira~\cite{DSS2023}. Additionally, we develop an $O(2)$-equivariant version of the lattice homotopy type. For a given negative-definite plumbing graph $\Gamma$, the associated plumbed 4-manifold is denoted by $W_\Gamma$, and its boundary is denoted by $Y_\Gamma$. If the plumbing graph $\Gamma$ is \emph{almost-rational} (abbreviated as $AR$, see \cite[Definition 8.1]{Nemethi:2005-1}), then we say that $Y_\Gamma$ is an \emph{almost-rational plumbed homology sphere}. The following theorem enables us to compute the invariants $\delta_R, \underline{\delta}_R$, and $\overline{\delta}_R$ for all torus knots.

 
\begin{thm}\label{main computation}
Let $K$ be a knot in $S^3$ and $\Sigma_2(K)$ be its double-branched cover. Suppose there is an almost-rational plumbing graph $\Gamma$ with a diffeomorphism $\partial W_\Gamma\cong \Sigma_2(K)$, where $W_\Gamma$ denotes the plumbed 4-manifold given by $\Gamma$, such that the deck transformation on $\Sigma_2(K)$ extends smoothly to a smooth involution $\tau$ on $W_\Gamma$. Moreover, assume that there exists an almost $I$-invariant path\footnote{For the definition of {\it almost $I$-invariant path}, see \cref{almost I inv path}.} $\gamma$ that carries the lattice homology of $(\Gamma,\mathfrak{s}_0)$, where $\mathfrak{s}_0$ denotes the unique spin structure on $\Sigma_2(K)$. Then there is an $O(2)$-equivariant map 
    \[
    \mathcal{T}^{O(2)} \colon \mathcal{H} (\gamma,\mathfrak{s}_0) \to SWF(\Sigma_2(K), \fraks_0)
    \]
    which is an $S^1$-equivariant homotopy equivalence with respect to a certain $O(2)$-action on $\mathcal{H} (\gamma,\mathfrak{s}_0)$. Here, $\mathfrak{s}_0$ denotes the unique self-conjugate $\mathrm{spin}^c$ structure on $\Sigma_2(K)$.
\end{thm}

This can be applied to compute a 2-knot invariant from real Seiberg–Witten theory. In \cite{Mi23}, Miyazawa defined the numerical invariant
\[
|\deg (S)| \in \Z_{\geq 0}
\]
for a smoothly embedded 2-knot $S$ in $S^4$ as the absolute value of the mapping degree of the $\{\pm 1\}$-framed real Bauer–Furuta invariant. Furthermore, in \cite[Proposition 4.25, Lemma 4.27, and Proposition 4.30]{Mi23}, he provided the following formula: 
\begin{align}\label{surgery real}
|\deg \left(\tau_{(k, \alpha)} (K)\right) |= |\deg (K)|
\end{align}
for a determinant one knot $K$ in $S^3$, where $\tau_{(k, \alpha)} (K)$ is the $\alpha$-roll $k$-twisted spun $2$-knot of $K$, and $|\deg (K)|$ is the absolute value of signed counting of $\{\pm 1\}$-framed real Seiberg--Witten solutions on the double-branched cover of $K$ with respect to its unique spin structure.  Since \cref{main computation} enables us to give non-equivariant homotopy type of $SWF_R(K)$, combined with \eqref{surgery real}, we can give a general formula of $\deg \left(\tau_{k, \alpha}(K)\right)$ as follows: 
\begin{cor} \label{Miya_torus}
If $K$ is a determinant one knot in $S^3$ satisfying the same assumptions as in \cref{main computation}, then we have that
\[
|\deg (\tau_{k, \alpha}(K)) | =  |\deg (K) |= \left|\chi ( \mathcal{H} (\gamma ,\mathfrak{s}_0)^I )\right| = 1
\]
for integers $k$ and $\alpha$ such that $\frac{k}{2}+ \alpha$ is an odd integer. 
    
\end{cor}

We also consider the case when $K$ is an arborescent knot, and in particular a Montesinos knot. We refer the reader to standard textbooks in knot theory, such as \cite{Lickorish:1997}, for precise definitions.

\begin{thm} \label{thm:Montesinos}
    Let $\Gamma$ be a negative-definite almost-rational plumbing graph, and $K$ be the corresponding arborescent knot.
    If $\gamma$ is a path that carries the lattice homology of $(\Gamma,\mathfrak{s})$ for the unique $\mathrm{spin}$ structure $\mathfrak{s}$ on the double-branched covering space  $\Sigma_2(K)$, then there is an $O(2)$-equivariant map
    \[
\mathcal{T}^{O(2)}\colon\mathcal{H}(\gamma,\mathfrak{s})\rightarrow SWF(\Sigma_2(K),\mathfrak{s})
    \]
    which is an $S^1$-equivariantly homotopy equivalence, where the $I$-action on $\mathcal{H}(\Gamma,\mathfrak{s})$ is given by the complex conjugation\footnote{For the definition of the complex conjugation on the $S^1$-equivariant lattice homotopy type, see \Cref{sec:Montesino}.}.
\end{thm}



\begin{cor} \label{cor:Montesinos}
Let $k$ and $\alpha$ be integers so that $\frac{k}{2} + \alpha$ is odd. 
 Under the same assumptions as in \Cref{thm:Montesinos}, suppose that the lattice homology of $(\Gamma,\mathfrak{s})$ is expressed as a graded root $R$. Denote the sets of leaves and angles of $R$ by $L(R)$ and $A(R)$\footnote{See \cite[Section 4.4]{alfieri2020connected} for the definition of angles in a graded root.}, respectively, and shift the grading (if necessary) so that all vertices of $R$ lie on even degrees. Additionally, we assume the determinant of $K$ is one. 
    Then we have
    \[
     |\deg ( \tau_{k, \alpha}(K) )| = |\deg (K) | = \left\vert \sum_{v\in L(R)} (-1)^{\frac{\mathbf{gr}(v)}{2}} - \sum_{v\in A(R)} (-1)^{\frac{\mathbf{gr}(v)}{2}} \right\vert. 
    \]
\end{cor}
Originally, Miyazawa used the computation of Seiberg--Witten moduli spaces using the analytical result given in \cite{MOY96}. Alternatively, \cref{cor:Montesinos} gives a combinatorial computation using the lattice homotopy type \cite{DSS2023} for a certain class of twisted roll spun $2$-knots. 

We also consider an invariant of a $2$-knot or a $\mathbb{RP}^2$-knot $S$ in $S^4$. For simplicity, we assume the double-branched cover of $S$ is homology $\overline{\mathbb{CP}}^2$ in this paper, in order to consider a canonical $\mathrm{spin}^c$ structure up to sign, whose first Chern class is a generator of $H_2(\overline{\mathbb{CP}}^2; \Z)$.
For the strongest invariant in the real setting for a $2$-knot or a $\mathbb{RP}^2$-knot $S$ in $S^4$, we have the $O(2)$-equivariant Bauer--Furuta invariants \footnote{For the construction of $O(2)$-equivariant Bauer--Furuta invariants, see \cref{o(2)BF}. }
\[
BF_S : V^+ \to V^+,
\]
which were introduced in \cite{BH24}, where $V$ denotes an $O(2)$-representation space and $+$ denotes the one-point compactification. If we consider the $\langle I \rangle  \subset O(2)$ fixed point part of $BF_S$, we can recover the Miyazawa's degree invariant.  
We give some structural theorem for $O(2)$-Bauer--Furuta invariant.  
\begin{thm}\label{structual o2}
    For any $2$-knot or $\mathbb{RP}^2$-knot in $S^4$, the $O(2)$-equivariant Bauer--Furuta invariant of it is $O(2)$-stably homotopic to $\pm$ identity up to the coordinate changes of the domain \footnote{For the definition of the coordinate changes, see \cref{coordinate change}.} if Miyazawa's degree invariant is one. 
\end{thm}

A similar structural theorem for $S^1\times \Z_p$-equivariant Bauer--Furuta invariants for $2$-knots introduced in \cite{BH21} is also proved in~\cite[Theorem 1.18]{IT24}, based on a similar technique.  

\begin{rem}
    As a refinement of \cref{structual o2}, one can also observe the following:
    For a given pair of $2$-knots or $\mathbb{RP}^2$-knots, suppose that Miyazawa's degree invariants of them are the same, then the corresponding $O(2)$-Bauer--Furuta invariants are $O(2)$-equivariantly stably homotopic up to sign and coordinate change. 
    Note that we can also define the $O(2)$-stable homotopy class of real Bauer--Furuta invariants even for orientable surfaces in $S^4$ by considering their double-branched covers with invariant spin structures with respect to the covering involutions. However, a similar technique proves that if the genus is positive, then the $O(2)$-stable homotopy class of the Bauer--Furuta invariant does not depend on the embeddings. 
\end{rem}

\subsection*{Acknowledgements} 
We express our gratitude to Irving Dai, Matthew Stoffregen, and Hirofumi Sasahira for their invaluable assistance regarding their publication \cite{DSS2023}. We also sincerely thank Jin Miyazawa for his insightful comments on the proof of \cref{deg comp}, and Hokuto Konno and Kouki Sato for the stimulating discussions. Finally, we are grateful to the referees for their careful reading and valuable suggestions.

The second author is partially supported by Samsung Science and Technology Foundation (SSTF-BA2102-02) and the POSCO TJ Park Science Fellowship.
The third author was partially supported
by JSPS KAKENHI Grant Number 20K22319, 22K13921, and RIKEN iTHEMS Program.

\section{Some topological facts}

\subsection{Concordance to torus knots}

In~\cite{ACMPS:2023-1} (see also~\cite{Ballinger:2022-1}), it was observed that the $0$-framed figure-eight knot can be transformed into a $-10$-framed unknot by performing two full  negative  twists, as described in Figure~\ref{fig:figure8}. This observation provided a new proof that the $(2,1)$-cable of the figure-eight knot is not smoothly slice in $B^4$. Furthermore, it will be crucially used in this article.

The $1$-framed red circles in Figure~\ref{fig:figure8} link the $0$-framed figure-eight knot, one linking algebraically once and the other algebraically three times, respectively. This implies that there is a concordance $S$ in
$$X := 2{\mathbb{CP}}^2 \smallsetminus \left(\mathring{B}^4 \sqcup \mathring{B}^4\right) \cong 2{\mathbb{CP}}^2 \# \left(S^3 \times I\right),$$
from the figure-eight knot to the unknot, such that $S$ represents the homology class $(1,3)$ in $H_2(X, \partial X; \mathbb{Z}) \cong H_2(2\mathbb{CP}^2; \mathbb{Z}) = \mathbb{Z} \oplus \mathbb{Z}$. Due to the framing change, applying a cabling operation along the annulus results in a new concordance $S_{2n}$ in $X$ from the $(2n,1)$-cable of the figure-eight knot to the $(2n,1-20n)$-cable of the unknot, namely the $T_{2n,1-20n}$ torus knot. Moreover, $S_{2n}$ represents the homology class $(2n,6n)$ in $H_2(X, \partial X; \mathbb{Z})$. For this, we only needed the fact that the figure-eight knot can be converted into a slice knot by introducing full negative twists along two disjoint disks, one intersecting $K$ algebraically once and the other intersecting it algebraically three times. We record this as a proposition:

\begin{prop}\label{prop:cableof41totorusknot} 
Let $K$ be a knot, such as the figure-eight knot, which can be transformed into a slice knot by applying full negative twists along two disjoint disks--one that intersects $K$ algebraically once and another that intersects it algebraically three times. Then, for each positive integer $n$, there is a smooth concordance $S_{2n}$ in the twice-punctured $2\mathbb{CP}^2$, denoted by $X$, from $K_{2n,1}$ to $T_{2n, 1-20n}$. Moreover, $S_{2n}$ represents the homology class $(2n, 6n)$ in $H_2(X, \partial X; \mathbb{Z})$. \QEDB \end{prop}


\begin{figure}[h]
\centering
\labellist
\pinlabel $0$ at 85 100
\pinlabel $0$ at 227 100
\pinlabel $-10$ at 347 100
\pinlabel \textcolor{red}{$1$} at 205 68
\pinlabel \textcolor{red}{$1$} at 197 33
\endlabellist
\includegraphics[width=.9\linewidth]{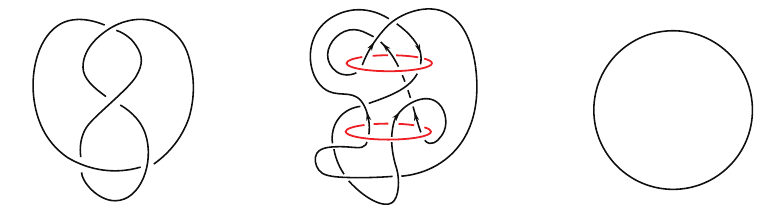}
\caption{The $0$-framed figure-eight knot becomes the $-10$-framed unknot after two full negative twists.}
\label{fig:figure8}
\end{figure}

\subsection{Topological invariants for torus knots}\label{subsec:topinvs}
The signature of a positive torus knot $T_{p,q}$ is computed via the following recursive formulae~\cite[Theorem 5.2]{Gordon-Litherland-Murasugi:1981-1}; note that we are using the convention where positive torus knots have negative signature. When $2q<n$, we have
\[
\sigma(T_{n,q}) =\begin{cases}
 \sigma(T_{n-2q,q})-q^2 +1 & \text{if $q$ is odd}\\
\sigma(T_{n-2q,q})-q^2 &  \text{if $q$ is even.}
\end{cases}
\]
When $q \le n < 2q$, we have
\[
\sigma(T_{n,q}) =\begin{cases}
  - \sigma(T_{2q-n,q}) - q^2 +1  & \text{if $q$ is odd}\\
- \sigma(T_{2q-n,q}) - q^2 +2 &  \text{if $q$ is even.}
\end{cases}
\]
Using this formula, we compute the signature of $T_{2n,1-20n}$ as follows.
\[
\begin{split}
    \sigma(T_{2n,1-20n}) &= -\sigma(T_{20n-1,2n}) \\
    &= -\sigma(T_{4n-1,2n}) + 16n^2 \\
    &= -2 + 20n^2 + \sigma(T_{1,2n}) = -2+20n^2.
\end{split}
\]

Now we compute the Neumann-Siebenmann invariant $\bar{\mu}$~\cites{Neumann:1980-1, Siebenmann:1980-1} of the double-branched cover of $S^3$ along $T_{2n,1-20n}$, which is the rational Brieskorn sphere $\Sigma(2,2n,1-20n)$ with respect to its unique spin structure. Since $\bar{\mu}$ satisfies $\bar{\mu}(-Y) = -\bar{\mu}(Y)$, we will instead compute $\bar{\mu}(\Sigma(2,2n,20n-1))$.

To compute it, we follow \cite{Neumann-Raymond:1978-1}. We first represent $\Sigma(2,2n,20n-1)$ as the boundary of a plumbed 4-manifold. One can do this by first representing it as a Seifert manifold and then translating each singular fiber as a leg in a star-shaped plumbing graph. To do so, we first write down the circle action on $\Sigma(2,2n,20n-1)$, which is given as:
\[
t\cdot (z_1,z_2,z_3) = \left(t^{n(20n-1)}z_1,t^{20n-1}z_2,t^{2n}z_3\right).
\]
It is then clear that, when $n > 1$, the above action has three singular orbits, with Seifert coefficients given by $(20n-1,-20n+11)$, $(20n-1,-20n+11)$, and $(n,-1)$, respectively. Note that the $(n,-1)$ orbit becomes nonsingular when $n = 1$, resulting in only two singular orbits.

Now, we can draw a plumbing graph $\Gamma$ such that the boundary of the 4-manifold obtained by plumbing disk bundles corresponding to $\Gamma$ is $\Sigma(2,2n,20n-1)$. Recall that a singular orbit of Seifert coefficient $(p,q)$ contributes to a leg of the form $[a_1,\ldots,a_n]$ in the resulting star-shaped plumbing graph, where $a_1,\ldots,a_n$ satisfy $a_i \leq -2$ and are uniquely determined by the continued fraction expansion
$$\frac{p}{q} = a_1 - \cfrac{1}{a_2 - \cfrac{1}{\ddots - \cfrac{1}{a_n}}}$$
of $p/q$. When $n>1$, we obtain a graph with three legs, given by 
\[
[\underbrace{-2,\ldots,-2}_{2n-2},-3,\underbrace{-2,\ldots,-2}_{8}],\qquad [\underbrace{-2,\ldots,-2}_{2n-2},-3,\underbrace{-2,\ldots,-2}_{8}],\qquad \text{ and }\qquad [-n].
\]
Moreover, the central vertex has a coefficient of $-2$. Hence the plumbing graph is given as follows.
\[
\begin{tikzpicture}[xscale=1.5, yscale=1, baseline={(0,-0.1)}]
    \node at (-0.1, 0.3) {$-2$};
    \node at (-1, 0.3) {$-n$};
    
    \node at (1, 1.3) {$-2$};
    \node at (2.5, 1.3) {$-2$};
    \node at (3.5, 1.3) {$-3$};
    \node at (4.5, 1.3) {$-2$};
    \node at (6, 1.3) {$-2$};

    \node at (1, -0.7) {$-2$};
    \node at (2.5, -0.7) {$-2$};
    \node at (3.5, -0.7) {$-3$};
    \node at (4.5, -0.7) {$-2$};
    \node at (6, -0.7) {$-2$};
    
    \node at (0, 0) (A0) {$\bullet$};
    \node at (1, 1) (A1) {$\bullet$};
    \node at (1, 0.8) (A1b) {};
    \node at (2.5, 1) (A2) {$\bullet$};
    \node at (2.5, 0.8) (A2b) {};
    \node at (3.5, 1) (A3) {$\bullet$};
    \node at (4.5, 1) (A4) {$\bullet$};
    \node at (4.5, 0.8) (A4b) {};
    \node at (6, 1) (A5) {$\bullet$};
    \node at (6, 0.8) (A5b) {};

    \node at (-1, 0) (B1) {$\bullet$};
    
    \node at (1, -1) (C1) {$\bullet$};
    \node at (1, -1.2) (C1b) {};
    \node at (2.5, -1) (C2) {$\bullet$};
    \node at (2.5, -1.2) (C2b) {};
    \node at (3.5, -1) (C3) {$\bullet$};
    \node at (4.5, -1) (C4) {$\bullet$};
    \node at (4.5, -1.2) (C4b) {};
    \node at (6, -1) (C5) {$\bullet$};
    \node at (6, -1.2) (C5b) {};
    
    \draw (A0) -- (A1);
    \draw[dotted] (A1) -- (A2);
    \draw (A2) -- (A3);
    \draw (A3) -- (A4);
    \draw[dotted] (A4) -- (A5);
    \draw (A0) -- (B1);
    \draw (A0) -- (C1);
    \draw[dotted] (C1) -- (C2);
    \draw (C2) -- (C3);
    \draw (C3) -- (C4);
    \draw[dotted] (C4) -- (C5);

    \draw [line width=0.4mm, decorate, decoration = {calligraphic brace,mirror}] (A1b) --  (A2b) node[midway,yshift=-0.8em]{$2n-2$};
    \draw [line width=0.4mm, decorate, decoration = {calligraphic brace,mirror}] (C1b) --  (C2b) node[midway,yshift=-0.8em]{$2n-2$};
    \draw [line width=0.4mm, decorate, decoration = {calligraphic brace,mirror}] (A4b) --  (A5b) node[midway,yshift=-0.8em]{$8$};
    \draw [line width=0.4mm, decorate, decoration = {calligraphic brace,mirror}] (C4b) --  (C5b) node[midway,yshift=-0.8em]{$8$};
    
\end{tikzpicture}
\]
On the other hand, when $n=1$, we only have two identical legs, given by $$[-3,\underbrace{-2,\ldots,-2}_{8}],$$ and the central vertex has coefficient $-1$. Hence, in this case, the plumbing graph is given as follows.

\[
\begin{tikzpicture}[xscale=1.5, yscale=1, baseline={(0,-0.1)}]
    \node at (0, 0.3) {$-1$};
    \node at (1, 0.3) {$-3$};
    \node at (2, 0.3) {$-2$};
    \node at (3.5, 0.3) {$-2$};
    \node at (-1, 0.3) {$-3$};
    \node at (-2, 0.3) {$-2$};
    \node at (-3.5, 0.3) {$-2$};
    
    \node at (0, 0) (C) {$\bullet$};
    \node at (1, 0) (A0) {$\bullet$};
    \node at (2, 0) (A1) {$\bullet$};
    \node at (2, -0.2) (A1b) {};
    \node at (3.5, 0) (A2) {$\bullet$};
    \node at (3.5, -0.2) (A2b) {};
    \node at (-1, 0) (B0) {$\bullet$};
    \node at (-2, 0) (B1) {$\bullet$};
    \node at (-2, -0.2) (B1b) {};
    \node at (-3.5, 0) (B2) {$\bullet$};
    \node at (-3.5, -0.2) (B2b) {};
    
    \draw (C) -- (A0);
    \draw (A0) -- (A1);
    \draw[dotted] (A1) -- (A2);
    \draw (C) -- (B0);
    \draw (B0) -- (B1);
    \draw[dotted] (B1) -- (B2);

    \draw [line width=0.4mm, decorate, decoration = {calligraphic brace,mirror}] (A1b) --  (A2b) node[midway,yshift=-0.8em]{$8$};
    \draw [line width=0.4mm, decorate, decoration = {calligraphic brace,mirror}] (B2b) --  (B1b) node[midway,yshift=-0.8em]{$8$};
    
\end{tikzpicture}
\]

Given these plumbing graphs, it is now easy to compute $\bar{\mu}$ using the formula
\[
\bar{\mu}(Y)= \frac{\sigma(\Gamma) - w^2}{8}
\]
where $\Gamma$ is a plumbing graph for $Y$, $\sigma(\Gamma)$ is the signature of $\Gamma$, and $w$ is the spherical Wu class. When $n$ is odd and $n>1$, using the plumbing graph from above, we find that $\sigma(\Gamma)=-4n-16$, and $w$ satisfies $w^2 = -4n+2$. For $n=1$, we have $\sigma(\Gamma)=-18$ and $w=0$. Therefore, we find:
\[
\bar{\mu}(\Sigma(2,2n,20n-1)) = -\frac{18}{8},
\]
for each positive odd integer $n$.



\begin{rem}
We do a brief sanity check here in the case $n=1$. Since $T_{2,19}$ is a 2-bridge knot, $\Sigma(2,2,19)$ is a lens space, so we should have 
\[
\bar{\mu}(\Sigma(2,2,19)) = \frac{\sigma(T_{2,19})}{8}.
\]
Since $T_{2,19}$ has signature $-18$, we see that our computation is correct for $n=1$.
\end{rem}

\section{Review of the $\delta_R$ invariant and the case $n=1$}

\subsection{Category of spectrums for $SWF_R(K)$}
In this section, we introduce a category $\mathfrak{C}_{G}$ that contains the real stable equivariant Floer homotopy type $SWF_R(K)$ for a knot $K$ in $S^3$. For a finite-dimensional vector space $V$, let $V^+$ be the one-point compactification of $V$. We define the group $G$ to be the cyclic group of order $4$ generated by $j \in \operatorname{Pin}(2)$, i.e., 
\[
G = \Z_4 = \{1, j, -1, -j\} \text{ with a subgroup } H = \Z_2 =  \{1,-1\} \subset G. 
\]
We will use the following representations of $G$:
\begin{itemize}
    \item  $\tilde{\R} :$  the 1-dimensional real representation space of $G$ defined by the surjection $G \to \Z_2=\{1,-1\}$ with $j \mapsto -1$, 
    \item $\C :$ the complex $1$-dimensional representation defined by assigning $j \in G$ to $i$ in $\C$. 
\end{itemize}
As representations for the suspensions, we shall only use subspaces of 
$\mathcal{V} = \oplus_{\N} \tilde{\R}$ and $\mathcal{W} = \oplus_{\N} \C$. A pointed finite $G$-CW complex $X$ is called a {\it space of type $(G, H)$-SWF}, if 
 $X^{H}$ is $G$-homotopy equivalent to $V^+$, where $V$ is a finite dimensional subspace of $\mathcal{V}$, and  $H$ acts freely on $X \smallsetminus X^{H}$. 
The dimension $\dim V$ is called the {\it level} of $X$.  

Now we introduce the category $\mathfrak{C}_{G}$ whose object is the equivalence classes of $(X,m,n)$ up to $G$-stably equivalence, where  $X$ is a space of type $(G, H)$-SWF, $m \in \Z$, and $n \in \Q$. We say that $(X,m,n)$ and $ (X',m',n')$ are {\it $G$-stably equivalent} if $n-n' \in \Z$ and there exist finite dimensional subspaces $V, V' \subset \mathcal{V}$  and $W, W' \subset \mathcal{W}$ and a pointed $G$-homotopy equivalence
\[
\Sigma^{V} \Sigma^{W}X
\to \Sigma^{V'} \Sigma^{W'}X',
\]
where $\dim_\R V -\dim_\R V' = m'-m$ and $\dim_{\C} W -\dim_{\C} W' = n'-n$.

Informally, we may think of the triple $(X,m,n)$ as the formal desuspension of $X$ by $V$ and $W$, where $V \subset \mathcal{V}$ and $W \subset \mathcal{W}$ with $\dim V=m$ and $\dim W=n$. So, symbolically one may write
\[
(X,m,n) =
\Sigma^{-m\tilde{\R}}\Sigma^{-n\C}X.
\]

Let $(X,m,n)$ and $(X',m',n')$ be triples as above. 
A $G$-stable map $(X,m,n) \to (X',m',n')$ is called a {\it $G$-local map}, if $\dim_{\R}V - \dim_{\R}V'=m'-m$ and it induces a $G$-homotopy equivalence on the $H$-fixed-point sets.
We say that $(X,m,n)$ and $(X',m',n')$ are {\it $G$-locally equivalent} if there exist $G$-local maps $(X,m,n) \to (X',m',n')$ and $(X',m',n') \to (X,m,n)$. 
The invariants $\delta_R$  $\overline{\delta}_R$ and $\underline{\delta}_R$ are invariant under  $G=\mathbb{Z}_4$ local equivalence. 
By considering the action comes from the inclusion $\Z_2 \to \Z_4$, we have the corresponding representations: 
\begin{itemize}
    \item the trivial representation $\R$
    \item 
    the non-trivial real representation $\tilde{\R}$. 
\end{itemize} 
With these representations, we also define $\mathfrak{C}_{\Z_2}$.

\subsection{The real Fr\o yshov invariants}\label{realFroyshov}

In this subsection, we review the construction of the real Fr\o yshov invariants. In~\cite{KMT:2023}, the three invariants 
\[
\delta_R (K), \bar{\delta}_R(K),  \text{ and  }\underline{\delta}_R(K)\in  \frac{1}{16}\Z 
\]
are introduced for a knot $K$ in $S^3$. In fact, it is defined for any oriented link in $S^3$ with non-zero determinant. In the case of the knot, the invariants are independent of the choice of orientations. These invariants are derived from $\Z_4$-equivariant stable homotopy type 
\[
SWF_R (K).  
\] Let $\mathfrak{s}_0$ be the unique spin structure on the double-branched cover $\Sigma_2(K)$, $\tau \colon \Sigma_2(K) \to \Sigma_2(K)$ be the deck transformation, and $P$ be the principal $\operatorname{Spin}(3)$ bundle for $\mathfrak{s}_0$. Since the fixed point set is codimension $2$, we can take an order $4$ lift $\wt{\tau} \colon P \to P$ of the induced map $$\tau_* \colon SO(T\Sigma_2(K)) \to SO(T\Sigma_2(K)),$$ where $SO(T\Sigma_2(K))$ is the orthonormal framed bundle of $\Sigma_2(K)$ with respect to a fixed invariant metric $g$ on $\Sigma_2(K)$. Then, we have the infinite-dimensional functional 
\[
CSD \colon \mathcal{C}_K := \left(i\ker d^ *\subset  i\Om^1_{\Sigma_2(K)} \right)  \oplus \Gamma (\mathbb{S}) \to \R
\]
called the {\it Chern--Simons Dirac functional}, where $\mathbb{S}$ is the spinor bundle with respect to $\fraks_0$ and $\Gamma (\mathbb{S})$ denotes the set of sections of $\mathbb
{S}$. The Seiberg--Witten Floer homotopy type is defined as the Conley index of the finite-dimensional approximation of the formal gradient flow of $CSD$. For that purpose, we describe the formal gradient of $CSD$ as the sum $l+c$, where $l$ is a self-adjoint elliptic part and $c$ is a compact map. Then, we decompose $\mathcal{C}_K$ into eigenspaces of $l$. 
Define $V^{\lambda}_{-\lambda}(K)\oplus W^{\lambda}_{-\lambda}(K)$ to be the direct sums of the eigenspaces of $l$ whose eigenvalues are in $(-\lambda, \lambda]$ and restrict the formal gradient flow $l + c$ to $V^{\lambda}_{-\lambda}(K) \oplus W^{\lambda}_{-\lambda}(K)$, where $V^{\lambda}_{-\lambda}(K)$ is the eigenspace corresponding to the space of 1-forms and $W^{\lambda}_{-\lambda}(K)$ is the eigenspace corresponding to spinors. Then by considering the Conley index $(N, L)$ for $\left(V^{\lambda}_{-\lambda}(K) \oplus W^{\lambda}_{-\lambda}(K) , l+p^{\lambda}_{-\lambda}c\right)$ with a  certain cutting off, we get the Manolescu's Seiberg--Witten Floer homotopy type
\begin{align}\label{usual SWF}
SWF(\Sigma_2(K), \frak{s}_0) := \Sigma^{-V^0_{-\lambda} \oplus W^0_{-\lambda}-n(\Sigma_2(K), \fraks_0, g))\C } N / L,  
\end{align}
where $n(Y, \fraks_0, g)$ is the quantity given in \cite{Ma03} and $g$ is a Riemannian metric on $\Sigma_2(K)$. For the meaning of desuspensions and how to formulate a well-defined homotopy type in a certain category, see \cite{Ma03}. 
For the latter purpose, we take $g$ as $\Z_2$-invariant metric.
Since we are working with the spin structure $\frak{s}_0$, we have an additional $\operatorname{Pin}(2)$-action on the configuration space $\mathcal{C}_K$ which preserves the values of $CSD$.
Now, we define an involution on $\mathcal{C}_K$
\[
I := j \circ \wt{\tau},  
\]
where $j$ is the quaternionic element in $\operatorname{Pin}(2) = S^1 \cup j\cdot S^1$. \footnote{The map $-I$ also induces another real involution on the configuration space. One can easily check the invariants $\delta_R, \underline{\delta}_R$ and $\overline{\delta}_R$ we will focus on in this paper do not depend on such choices. }

Since $I$ also acts on $S$ anti-complex linearly, the lift $I$ is called a \emph{real structure} on $\frak{s}_0$.
Combined with $S^1$-action, we 
can take Conley index so that we have an $O(2)$-action on $SWF(\Sigma_2(K), \frak{s}_0)$.

 Now, we define 
\begin{align*}
SWF_R(K) :&= \Sigma^{-(V^0_{-\lambda} \oplus W^0_{-\lambda} )^I - \frac{1}{2} n (\Sigma_2(K), \fraks_0, g) \C } N^I / L^I \\
& =  \left[\left(N^I/L^I,
\dim_{\R}\left(V^0_{-\lambda}\right)^I,
\dim_{\C}\left(W^0_{-\lambda}\right)^I
+n(Y, \frakt,g)/2\right)\right] \in \mathfrak{C}_{G}
\end{align*}
which we call the \emph{real Seiberg--Witten Floer homotopy type} for $K$. Here, we take an $O(2)$-invariant index pair $(N, L)$ for the flow $\left(V^{\lambda}_{-\lambda}(K) \oplus W^{\lambda}_{-\lambda}(K), l+p^{\lambda}_{-\lambda}c\right)$ with a certain cutting off. 
Since the action of $j$ commutes with $I$, we have a $\Z_4$-action on the stable homotopy types $SWF_R(K)$. Therefore, we have the following two equivariant cohomologies: 
\begin{align*}
&\wt{H}^*_{G} (SWF_R(K); \mathbb{Z}_2):= \wt{H}^{* + \dim (V^0_{-\lambda})^I +2 \dim_{\C}(W^0_{-\lambda})^I  +  n (\Sigma_2(K), \fraks_0, g) }_{G} (N^I/L^I; \mathbb{Z}_2) 
\end{align*}
for $G= \Z_2$ or $\Z_4$, $\Z_2$ denotes the group of the $\{\pm 1\}$-constant gauge transformations on $N^I/L^I$. 
If we write $H^*(B\Z_2) \cong \mathbb{Z}_2 [W]$, 
we define 
\[
\delta_R(K) :=\frac{1}{2}\left( \min \left\{m \in \Z \mid x\in H^m_{\Z_2} (N^I/ L^I; \mathbb{Z}_2), W^k x \neq 0, \forall k \right\}- \dim (V^0_{-\lambda})^I- 2 \dim_{\C}(W^0_{-\lambda})^I - n(\Sigma_2(K), \fraks_0, g)\right). 
\]
 
Similarly, if we put 
\[
\wt{H}_{\Z_4}^\ast (S^0) \cong \Z_2 [U, Q]/(Q^2 = 0),
\]
we can write the definitions of $\underline{\delta}_R$ and  $\overline{\delta}_R$ as 
\begin{align*}
\underline{\delta}_R(K) := & \frac{1}{2} (\min \left\{m \in \Z \mid x\in H^m_{\Z_4} (N^I/L^I; \mathbb{Z}_2), U^k x \neq 0, \forall k, m \equiv \dim (N^I/L^I)^{\Z_2} \operatorname{mod} 2 \right\} 
\\
&-\dim \left(V^0_{-\lambda}\right)^I-  2 \dim_{\C}(W^0_{-\lambda})^I - n(\Sigma_2(K), \fraks_0, g) ) \\
\overline{\delta}_R(K) := & \frac{1}{2} (\min \left\{m \in \Z \mid x\in H^m_{\Z_4} (N^I/L^I; \mathbb{Z}_2), U^k x \neq 0, \forall k , m \equiv \dim (N^I/L^I)^{\Z_2} +1 \operatorname{mod} 2  \right\}
\\
&- \dim \left(V^0_{-\lambda}\right)^I-  2 \dim_{\C}(W^0_{-\lambda})^I - n(\Sigma_2(K), \fraks_0, g)) - \frac{1}{2}.
\end{align*}

\subsection{$O(2)$-equivariant Floer homotopy type}

In this section, we define the $O(2)$-equivariant Floer homotopy type for knots.

We have actions of $I$ and $S^1$ on the Floer homotopy type $SWF(\Sigma_2(K) , \mathfrak{s}_0)$. This gives a well-defined $O(2)$-action on $SWF(\Sigma_2(K) , \mathfrak{s}_0)$ which lies in certain $O(2)$-equivariant stable homotopy category. 
From \cref{O2rep}, and up to base change, the $O(2)$-representations that appear in this setting are the following:
\begin{itemize}
    \item the trivial $1$-dimensional real representation $\R$,
    \item the non-trivial $1$-dimensional real representation $\wt{\R}$ obtained via the surjection
    \[
    O(2) \to O(1) = \Z_2,
    \]
    \item the irreducible $2$-dimensional representation $\C$, equipped with the natural action of $O(2) \cong S^1 \rtimes \Z_2$, where $S^1$ acts by complex multiplication and $\Z_2$ acts by complex conjugation.
\end{itemize}
Accordingly, the universe in this setting is given by
\[
\mathcal{U} = \R^\infty \oplus \wt{\R}^\infty \oplus \C^\infty.
\]
The category $\mathfrak{C}_{O(2)}$ containing the $O(2)$-Floer homotopy type is described as follows:
\begin{itemize}
    \item The objects are tuples $(W, l, m, n)$, where $W$ is a pointed $O(2)$-space, $l, m \in \mathbb{Z}$, and $n \in \mathbb{Q}$.
    \item Given two objects $(W_0, l_0, m_0, n_0)$ and $(W_1, l_1, m_1, n_1)$, the set of morphisms is given by
    \[
    \lim_{p_0, p_1, q \to \infty} \left[\Sigma^{p_0 \R \oplus p_1 \wt{\R} \oplus q \C }W_0,\, \Sigma^{(p_0 + l_0 - l_1) \R \oplus (p_1 + m_0 - m_1) \wt{\R} \oplus (q + n_0 - n_1) \C }W_1 \right]^0_{O(2)}
    \]
    provided that $n_0 - n_1 \in \Z$, where $[X, Y]^0_{O(2)}$ denotes the set of based $O(2)$-equivariant maps up to $O(2)$-equivariant based homotopy.
\end{itemize}

As in the case of $S^1$, we define the (de)suspension by
\begin{align}\label{aaa}
\Sigma^V (W, l, m, n) := \left(\Sigma^{V} W,\, l + 2a,\, m + 2b,\, n +2 c\right)
\end{align}
when $V \cong V_1 \oplus V_2 \oplus V_3$ has a trivialization of the form $\R^a \oplus \wt{\R}^b \oplus \C^c$. The reason why we consider the suspension of the form \eqref{aaa} is the set homotopy classes of identifications $V_1 \cong \R^a, V_2 \cong \wt{\R}^b$ and $V_3 \cong \C^c$ as real representations are identified with $\pi_0 ( GL_a(\R))$, $\pi_0 ( GL_b(\R))$ and $\pi_0 ( GL_c(\R))$ from \cref{auto}. 

Under the setting of \cref{realFroyshov}, we consider Manolescu's Floer homotopy type 
$SWF(\Sigma_2(K), \frak{s}_0)$ with an $O(2)$-action:
\[
\left[
\left(\Sigma^{V^0_{-\lambda} \oplus W^0_{-\lambda} }N/L,\,
0,\,
0,\,
\frac{n(\Sigma_2(K), \frak{s}_0, g)}{2}\right)\right]
\]
We need to confirm the dependence of $SWF(\Sigma_2(K), \frak{s}_0)$
on the $\Z_2$-invariant Riemannian metrics. Since we are considering rational homology 3-spheres, we only need to take into account the spectral flows coming from Dirac operators. One can check the $O(2)$-equivariant version of the spectral flow coincides with the usual $S^1$-equivariant spectral flow. Thus, by employing the same formulation as in Manolescu's work \cite[Section 6]{Ma03}, we can use $n(\Sigma_2(K), \frak{s}_0, g)$ to determine the data for desuspensions and obtain a well-defined $O(2)$-equivariant Seiberg--Witten Floer homotopy type. That is, we simply replace the standard $S^1$-representation $\C$ with the standard $O(2)$-representation $\C$.

There is a functor
\[
\mathfrak{C}_{O(2)} \to \mathfrak{C}_{\Z_2}
\]
defined by taking the $\langle I \rangle$-fixed point part of the spectrum:
\[
[(W, l, m, n)]^I := [(W^I, m, 2n)],
\]
together with the natural restriction maps of equivariant maps. With respect to this functor, the constructions yield the following:

\begin{prop}
For any knot $K$ in $S^3$, we have
\[
SWF_R(K) \cong SWF(\Sigma_2(K), \frak{s}_0)^I. \eqno\QEDB
\]
\end{prop}

\subsection{$O(2)$-equivariant cobordism map}\label{o(2)BF}

In order to calculate the real Seiberg--Witten Floer homotopy type, $SWF_R(K)$, we will construct an $O(2)$-equivariant map. This map is obtained as the $O(2)$-equivariant Bauer--Furuta invariant for the branched covers and the homotopies between them. We review the construction of the $O(2)$-equivariant Bauer--Furuta invariant in this section.  

Let $(Y_0, \frakt_0)$ and $(Y_1, \frakt_1)$ be spin$^c$ rational homology 3-spheres with \emph{odd involutions} $\tau_i \colon Y_i \to Y_i$, i.e., an involution $\tau_i$ such that
\[
\tau_i^* \frakt_i \cong \overline{\frakt}_i 
\]
for each $i =0,1$.
A typical situation involves $Y_0$ and $Y_1$ as the double-branched covers of  knots $K$ and $K'$, each with unique spin structures $\frakt_0$ and $\frakt_1$, respectively. Let $(W,\fraks)$ be a smooth spin 4-dimensional oriented cobordism from $(Y_0, \frakt_0)$ to $(Y_1, \frakt_1)$ with $b_1(W)=0$. We assume that there is an \emph{odd involution} $\tau$ on $W$ such that $\tau|_{Y_i}=\tau_i$ for each $i$, i.e., an involution $\tau$ such that 
\[
\tau ^* \fraks \cong \overline{\fraks} 
\]
and 
the fixed point set of $\tau$ is of codimension $2$. Again, a typical situation is when $W$ is obtained as the double-branched cover along a smoothly embedded surface in a $4$-manifold. 
Let $\mathbb{S}^{\pm}$ be positive and negative spinor bundles on $W$, and let $\mathbb{S}_i$ be the spinor bundles on $Y_i$ for each $i$. In \cite[Section 2]{KMT:2023}, an antilinear lift $I$ on the spinor bundles $\mathbb{S}^{\pm},\  \mathbb{S}_i$, and the configuration spaces are constructed. Note that such a choice (of $I$) corresponds to a choice of splittings of 
\[
1 \to S^1 \to G_\fraks \to \Z_2 \to 1 
\]
as it is pointed out in \cite[Subsection 2.1]{BH24}, where $G_\fraks$ denotes a group of certain bundle maps of the spinor bundle $\mathbb{S}$ on $W$ which covers $\tau$. We fix a splitting when we consider $O(2)$-equivariant Bauer--Furuta invariant. \footnote{It should be possible to write down this dependence on the choices of splittings explicitly. However, we do not need to do it in this paper, so we omit it. }

In this setting, Konno, Miyazawa, and the third author \cite[Section 3.7]{KMT:2023} (see also \cite[Section 2.1]{BH24}) constructed an $I$-equivariant Bauer--Furuta map, which is formally written as 
\begin{align}\label{O2BF}
BF_{W, \fraks} \colon \left(\C^{ \frac{1}{8} (c_1(\fraks)^2 - \sigma(W)) }\right)^+ \wedge  SWF(Y_0, \frakt_0)  \to \left( \R^{b_2^+(W)}\right)^+   \wedge SWF(Y_1, \frakt_1) 
\end{align}
for the 4-manifold $W$ up to stable homotopy, with a certain $I$-action on $\C^{ \frac{1}{8} (c_1(\fraks)^2 - \sigma(W)) }$ and $\R^{b_2^+(W)}$ . 

In this paper, we mainly focus on the case of $b_2^+(W)=0$ in the construction of a map between $O(2)$-lattice homotopy type and the Seiberg--Witten Floer homotopy type. Note that if we forget the $I$-action, $BF_{W, \fraks}$ recovers the usual $S^1$-equivariant Bauer--Furuta invariant. 
Combined it with the $S^1$-action, one can see the map $BF_{W, \fraks}$ is $O(2)$-equivariant since $I$ and $i \in S^1$ anticommute. 
 Consequently, we can define $O(2)$-equivariant Bauer--Furuta invariants of $2$-knots or $\mathbb{RP}^2$-knots as introduced earlier.

Suppose $Y_0$ and $Y_1$ are the double-branched covers along knots $K$ and $K'$, each with unique spin structures $\frakt_0$ and $\frakt_1$, respectively. Assume $W$ is obtained as the double-branched cover along a surface cobordism $S$ from $K$ to $K'$ properly and smoothly embedded in a 4-dimensional cobordism from $S^3$ to $S^3$.
If we consider the $I$-invariant part of $BF_{W, \fraks}$, we obtain a cobordism map in real Seiberg--Witten theory, still denoted by $BF_{W, \fraks}$ if there is no confusion. This cobordism map is used to prove Fr\o yshov type inequalities in \cite{KMT:2023}.

\subsection{The case $n=1$: a toy model}
We now offer an alternative proof of the main theorem for the case of $n=1$, previously established using Heegaard Floer theory in~\cite{DKMPS:2022-1} and minimal genus functions in~\cite{ACMPS:2023-1}. This proof serves as a useful toy model for the case of general odd $n>1$.


Let $K$ be a knot in $S^3$, and let $\Sigma_2(K)$ denote its double-branched cover. Recall from~\cite[Proposition 1.10]{KMT:2023} that when $\Sigma_2(K)$ is a lens space, we have
\[
\delta_R(K)=\underline{\delta}_R(K)=\overline{\delta}_R(K)=-\frac{\sigma(K)}{16},
\]
where $\sigma(K)$ is the signature of $K$. Given that the torus knot $T_{2,-19}$ is a two-bridge knot, and thus its double-branched cover is the lens space $\Sigma(2,2,-19) = L(19,1)$, we deduce
\begin{equation}\label{eq:delta219}
\overline{\delta}_R(T_{2,-19})= -\frac{\sigma(T_{2,-19})}{16} =-\frac{9}{8}.    
\end{equation}

Now, we invoke the following theorem.


\begin{thm}[{\cite[Theorem~1.6]{KMT:2023}}]\label{Theorem B for links}
Let $K$ and $K'$ be knots in $S^3$, let $X$ be an oriented, smooth, compact, connected 4-manifold cobordism from $S^3$ to $S^3$ with $H_1(X; \Z_2)=0$, and let $S$ be a connected surface cobordism that is smoothly embedded in $X$ from $K$ to $K'$, such that the homology class $[S]/2$ in $H_2(X, \partial X; \mathbb{Z})$ reduces to $w_2(X)$. Let $\Sigma_2(S)$ be the double-branched cover of $X$ branched along $S$ and $\sigma(\Sigma_2(S))$ be its signature.


If $b_2^+(\Sigma_2(S)) - b_2^+(X) = 1$, then we have
\begin{align*}
\underline{\delta}_R(K)-\frac{1}{16}\sigma(\Sigma_2(S)) \leq \bar{\delta}_R(K').
\end{align*}
If $b_2^+(\Sigma_2(S)) - b_2^+(X) = 0$, then the following stronger inequality holds:
$$
\underline{\delta}_R(K)-\frac{1}{16}\sigma(\Sigma_2(S)) \leq \underline{\delta}_R(K'). \eqno\QEDB$$
\end{thm}

\noindent The latter part is a stronger conclusion since we have $\underline{\delta}_R(K) \leq \delta_R(K) \leq \bar{\delta}_R(K)$ for each knot $K$.

\begin{rem}\label{rem:b2sig} The following can be computed using the Mayer--Vietoris sequence and the $G$-signature theorem (see \cite[Lemma 4.5]{KMT:2023}). Suppose that $S$ is an annulus; then, we have
    \[
    \begin{split}
    b_2^+(\Sigma_2(S)) - b_2^+(X) &= b_2^+(X) -\frac{1}{4} [S]^2  - \frac{1}{2} \sigma (K)+ \frac{1}{2} \sigma (K'), \\
    \sigma(\Sigma_2(S)) &=  2 \sigma (X) - \frac{1}{2} [S]^2- \sigma (K)+\sigma (K').
    \end{split}
    \]
    We will use these to compute the quantities $b_2^+(\Sigma_2(S))$ and $\sigma(\Sigma_2(S))$.
\end{rem}

We have the following immediate corollary of Theorem~\ref{Theorem B for links}.
\begin{cor}\label{cor:negativeimmersedpoints}
    Let $K$ be a knot with vanishing signature. Suppose $K$ bounds a normally immersed disk in $B^4$ with only negative double points. Then, we have $0 \leq \underline{\delta}_R(K)$.
\end{cor}
\begin{proof}
    If $K$ bounds a normally immersed disk in $B^4$ with $m$ negative double points, then there is a smooth concordance $S$ in twice-punctured $m\mathbb{CP}^2$, denoted by $X$, from the unknot to $K$. Moreover, $S$ represents $[S] = (2,2, \ldots, 2)$ in $H_2(X, \partial X; \mathbb{Z})$. Moreover, by Remark~\ref{rem:b2sig}, we have that $b_2^+(\Sigma_2(S)) - b_2^+(X)=0$ and $\sigma(\Sigma_2(S))=0$. Then the conclusion follows from Theorem~\ref{Theorem B for links}.
\end{proof}


Let $E$ be the figure-eight knot. Consider the smooth concordance $S$, as described in \Cref{prop:cableof41totorusknot}, from $E_{2,1}$ to $T_{2,-19}$ in a twice-punctured $2\mathbb{CP}^2$, which is denoted by $X$. This concordance has the homology class $(2n,6n)$. To check that the assumptions of \Cref{Theorem B for links} are satisfied for $S$, we calculate:
\[
\begin{split}
    b_2^+(\Sigma_2(S)) - b_2^+(X) 
    &= b_2^+(X) -\frac{1}{4} [S]^2 + \frac{1}{2} \sigma (T_{2,-19})\\
    &= 2-\frac{1}{4}\left(2^2+6^2\right)+\frac{1}{2} (18) \\
    &= 2-10+9 \\&= 1.
\end{split}
\]
Hence the assumptions are satisfied, and thus we get
\[
 \underline{\delta}_R(E_{2,1}) -\frac{1}{16}\left(  2 \sigma (X) - \frac{1}{2} [S]^2 + \sigma (T_{2,-19})\right) \leq \bar{\delta}_R(T_{2,-19}).
\]
Since we have
\[
\begin{split}
    -\frac{1}{16}\left(  2 \sigma (X) - \frac{1}{2} [S]^2 + \sigma (T_{2,-19})\right) &= -\frac{1}{16} \left( 2\cdot 2-\frac{1}{2}(2^2+6^2) + 18  \right) \\
    &= -\frac{1}{8},
\end{split}
\]
use \eqref{eq:delta219} to conclude that 

$$ \underline{\delta}_R(E_{2,1}) \leq -1.$$ Thus, by applying Corollary~\ref{cor:negativeimmersedpoints}, we conclude that $E_{2,1}$ does not bound a normally immersed disk in $B^4$ with only negative double points. In particular, it is not smoothly slice.

\begin{rem}\label{rmk:c4(4121)}
    Consider the unique minimal genus Seifert surface $S$ for the figure-eight knot $E$. It consists of two bands, one with a full positive twist and the other with a full negative twist. Take two parallel copies of $S$ and denote them by $S^+$ and $S^-$. Connecting them with a half-twisted band yields a Seifert surface $S'$ for $E_{2,1}$. Perform a crossing change on $E_{2,1}$ that corresponds to undoing the full positive twist on $S^+$ and a crossing change that corresponds to undoing the full negative twist on $S^-$. These crossing changes produce a new knot $R$ and a Seifert surface $S''$ derived from $S'$ for $R$. Moreover, on $S''$, we have a two-component unlink $U_1 \cup U_2$ such that the Seifert form restricted to the homology classes of the unlink vanishes (i.e., it forms a \emph{derivative link} for $R$), which in particular implies that $R$ is a ribbon knot. In fact, one can check that $R$ is the ribbon knot $\mathrm{12n268}$. Therefore, we conclude that $c_4^+(E_{2,1}) = 1$.
\end{rem}

\subsection{Two technical lemmas}

Before ending this section, we shall show the following lemmas, which will be used later. We say that a spectrum $X$ is a \emph{$\Z_2$-homology sphere} if $\tilde{H}^\ast(X;\Z_2)\cong \pi_\ast(X\wedge H\mathbb{Z}_2)$ is 1-dimensional over $\mathbb{Z}_2$.
\begin{lem}\label{lem:keylemma2}
    If $SWF_R(K)$ is a $\Z_2$-homology sphere, then we have 
    \[
    \delta_R(K) = \underline{\delta}_R (K) = \overline{\delta}_R (K).
    \]
\end{lem}
\begin{proof}
    Since Seiberg--Witten spectra are finite, we may assume for simplicity that it is actually a finite CW-complex by stabilizing it many times. Then, since $SWF^I(K)$ is a $\Z_2$-homology sphere, then it is also a $\Z_2$-cohomology sphere. Consider the Serre spectral sequence
    \[
    E_2=\tilde{H}^\ast\left(SWF_R(K);\Z_2\right)\otimes_{\Z_2} H^\ast(B\Z_4;\Z_2)\Rightarrow \tilde{H}^\ast_{\Z_4}\left(SWF_R(K);\Z_2\right)=E_\infty.
    \]
    We already know that the $E_2$ page is free of rank 1 over $H^\ast(B\Z_4;\Z_2)\cong \Z_2[U,Q]/(Q^2)$. On the other hand, it follows from discussions in \cite[Section 3]{KMT:2023} that the $E_\infty$ page, after localizing by formally inverting $U$, is free of rank 1 over $\Z_2[U,U^{-1},Q]$. Therefore we see that the spectral sequence collapses at the $E_2$ page, and hence we get
    \[
    \delta_R(K)=\underline{\delta}_R(K)=\overline{\delta}_R(K)
    \]
    as desired.
\end{proof}

\begin{lem}\label{lem:keylemma1}
    Let $G$ be a finite 2-group and $X,Y$ be finite $G$-CW-complexes. Suppose that there exists a homotopy equivalence $f:X\rightarrow Y$ which is $G$-equivariant; note that $f$ might not be a $G$-equivariant homotopy equivalence. Then the restriction of $f$ to $G$-fixed point loci, i.e.,
    \[
    f^G:X^G\rightarrow Y^G,
    \]
    induces an isomorphism between $\Z_2$-coefficient singular homology.
\end{lem}
\begin{proof}
    Consider the commutative diagram
    \[
    \xymatrix{
    X^G \ar[d]^{f^G}\ar[r] & \left(X^G\right)^\wedge_2 \ar[d]^{(f^G)^\wedge_2}\ar[r] & \left(X^\wedge_2\right)^{hG}\ar[d]^{(f^\wedge_2)^{hG}} \\
    Y^G \ar[r] & \left(Y^G\right)^\wedge_2 \ar[r]& \left(Y^\wedge_2\right)^{hG} 
    }
    \]
    where $(-)^{hG}$ denotes the homotopy fixed point, i.e.,
    \[
    Z^{hG} = [EG,Z]^G,
    \]
    and $(-)_2^\wedge$ denotes the Bousfield-Kan 2-adic completion. Note that for any $G$-space $Z$, we have a canonically defined 2-adic completion map
    \[
    Z\rightarrow Z^\wedge _2
    \]
    and the (2-completed) comparison map
    \[
    \left(Z^G\right)^\wedge_2 \rightarrow \left(Z^\wedge_2\right)^{hG}.
    \]
    
    But 2-adic completion maps are mod 2 homotopy equivalences. Furthermore, for finite $G$-complexes, the 2-completed comparison map is a weak homotopy equivalence, due to the Sullivan conjecture \cite{dwyer1989fibrewise,carlsson1991equivariant,lannes1992espaces}. By the mod $p$ Whitehead theorem \cite{schiffman1981mod}, this is equivalent to saying that all horizontal maps in the diagram above, as well as $(f^\wedge_2)^{hG}$, induce isomorphisms between $\Z_2$-coefficient homology. Therefore $f^G$ also induces an isomorphism between $\Z_2$-coefficient homology.
\end{proof}

\begin{rem}
    Since completion is a stable operation, by replacing fixed points with geometric fixed points, we can easily see that \Cref{lem:keylemma1} also applies to the case when $X$ and $Y$ are finite $G$-spectra.
\end{rem}

\section{Lattice homotopy type and the proof of \Cref{thm:main}}

Our strategy utilizes the work of Dai, Sasahira, and Stoffregen~\cite{DSS2023} on the lattice homotopy computation of the Floer homotopy type. For background materials on lattice homology, we refer the reader to~\cite{Oz-Sz:2003-1, Nemethi:2005-1, Nemethi:2008-1} for the general theory, and to~\cite{DSS2023} for the modernized constructions. We will mainly follow the notations used in~\cite{DSS2023} for lattice homology. In this section, we construct $O(2)$-equivariant maps between the $O(2)$-lattice homotopy type and the $O(2)$-equivariant Seiberg--Witten Floer homotopy type in two different situations, which can be regarded as morphisms in $\mathfrak{C}_{O(2)}$. Note that these stable equivariant morphisms are assumed to be based maps. However, by~\cite[Chapter II, Lemma (4.15)]{DT87}, there is no distinction between based $O(2)$-equivariant maps and unbased ones between $O(2)$-CW complexes $X$ and $Y$ when $\pi_1(X^{O(2)}) = \pi_1(Y^{O(2)}) = 0$. In our setting, both the $O(2)$-lattice homotopy type and the $O(2)$-Seiberg--Witten Floer homotopy type satisfy this condition, so we will ignore basepoints in the construction.

\subsection{Computation sequences in lattice homology} \label{subsec:computations sequences}
In this subsection, we will review the construction of \emph{computation sequences} in lattice homology, following~\cite{Nemethi:2008-1}, as a detailed understanding of it is crucial in understanding the construction of $j$-action on the Dai--Sasahira--Stoffregen lattice homotopy type. We will then modify it a little bit to constuct $O(2)$-lattice homotopy type in the later subsection.

Given an almost rational, negative-definite plumbing graph $\Gamma$, let $W_\Gamma$ denote the corresponding 4-manifold. Since $H^2(W_\Gamma;\mathbb{Z})$ has no 2-torsion, the first Chern class map
\[
c_1\colon \mathrm{Spin}^c(W_\Gamma)\rightarrow H^2(W_\Gamma;\mathbb{Z})
\]
is injective. Furthermore, its image is the set of \emph{characteristic elements}, i.e., cohomology classes $\alpha\in H^2(W_\Gamma;\mathbb{Z})$ satisfying $\alpha\cup\beta = \beta^2 \pmod 2$ for all $\beta\in H^2(W_\Gamma;\mathbb{Z})$. We will henceforth identify $\mathrm{Spin}^c$ structures on $W_\Gamma$ with characteristic elements of $H^2(W_\Gamma;\mathbb{Z})$. More precisely, given a characteristic element $\alpha$, we will denote the unique $\mathrm{Spin}^c$ structure on $W_\Gamma$ whose $c_1$ is $\alpha$ as $[\alpha]$.

Recall that $W_\Gamma$ is defined via gluing various disk bundles. For each node $v$ of weight $w_v$ in $\Gamma$, we have a disk bundle $p_v\colon E_v \rightarrow S^2$ of Euler number $w_v$, whose total space is embedded in $W_\Gamma$. We denote the zero-section of $p_v$ by $S_v$; clearly, $H_2(W_\Gamma;\mathbb{Z})$ is freely generated by the classes $[S_v]$ where $v$ runs over all nodes of $\Gamma$. Then the weights of nodes in $\Gamma$ give rise to the \emph{canonical class} $K$, which is the unique element of $H^2(W_\Gamma;\mathbb{Z})$ satisfying
\[
K \cap [S_v] = -w_v - 2 \quad \text{for all nodes }  v  \text{ of }\Gamma.
\]
Clearly $K$ is a characteristic element, and any characteristic element can be written uniquely as $K+2\alpha$ for some $\alpha\in H^2(W_\Gamma;\mathbb{Z})$.

Note that $H^2(W_\Gamma;\mathbb{Z})$ can be embedded as a subgroup of $H_2(W_\Gamma;\mathbb{Q})$ via the intersection form of $W_\Gamma$. Consider the cone
\[
S_\mathbb{Q} = \left\{ x\in H_\ast(W_\Gamma;\mathbb{Q}) \mid (x,[S_v])\leq 0 \text{ for all node } v \text{ of }\Gamma \right\}.
\]
It follows from the negative definiteness of $\Gamma$ that every element $x\in S_\mathbb{Q}$ satisfy $x\ge 0$, where $\ge$ denotes the partial ordering on $H_\ast(W_\Gamma;\mathbb{Q})$ defined by inequalities on weights of vertices of $\Gamma$.

We can find a \emph{minimal} representative of $[k]\in \mathrm{Spin}^c(Y_\Gamma)$ as follows: Consider the intersection 
\[
\left(l^\prime + H_2(W_\Gamma;\mathbb{Z})\right)\cap S_\mathbb{Q}.
\]
With respect to the partial ordering on $H_2(W_\Gamma;\mathbb{Z})$, given by 
\[
x\leq y \quad \text{if}\quad (x,[S_v])\le (y,[S_v]) \quad \text{for all node}\quad v\quad \text{of}\quad \Gamma,
\]
this subset admits a unique minimal element $l^\prime_{[k]}$ \cite[Lemma 5.4]{Nemethi:2005-1}. Thus we take the corresponding distinguished representative $k_r$ of $[k]$ as follows:
\[
k_r = K+2l^\prime_{[k]}.
\]

Now fix a vertex $b_o$ among the vertices of $\Gamma$. Then, for each integer $i\ge 0$, we construct a sequence of cycles $x(i)\in H_2(W_\Gamma;\mathbb{Z})$ as the minimal element satisfying the following conditions:
\begin{itemize}
    \item the coefficient of $x(i)$ for the vertex $b_o$ is $i$;
    \item $\left(x(i)+l^\prime_{[k]},b_j\right)\le 0$ for any vertex $b_j \ne b_o$.
\end{itemize}
It follows from \cite[Lemma 7.6]{Nemethi:2005-1} that $x(i)$ is uniquely defined and satisfies $x(i)\ge 0$. Furthermore, every leaf of the graded root $R_\Gamma$ induced by $\Gamma$ contains at least one $x(i)$ \cite[Lemma 9.2]{Nemethi:2005-1}.

We then construct a computation sequence between $x(i)$ and $x(i+1)$ as follows. Set $x_0 = x(i)$ and $x_1 = x(i)+b_o$. Assuming that $x_1,\dots,x_l$ are already constructed, we inductively define $x_{l+1}$ as follows. If $(x_l + l^\prime_{[k]},b_j)\le 0$ for all vertices $b_j \ne b_o$, then we stop, as $x_l=x(i+1)$ is satisfied by \cite[Lemma 7.7]{Nemethi:2005-1}. Otherwise, we take $x_{l+1}=x_l + b_{j(l)}$, where $b_{j(l)}$ is a vertex of $\Gamma$ which is not $b_o$ and satisfies $(x_l + l^\prime_{[k]},b_{j(l)})>0$.

Now we amalgamate computation sequence between $x(i)$ and $x(i+1)$ for each $i\ge 0$ to obtain an infinite sequence of cycles. We can truncate this sequence after sufficiently many terms to get a finite sequence. This sequence is the \emph{computation sequence} for the lattice homology of $Y_\Gamma=\partial W_\Gamma$; more precisely, this sequence \emph{carries} the lattice homology of $Y_\Gamma$ in the sense of~\cite[Theorem 4.9]{DSS2023}.

\begin{rem}
    It is possible to make sense of computation sequences between $x(i)$ and $x(i+s)$ for positive integers $s$, by going from $x(i)$ to $x(i)+sb_o$ by adding one $b_o$ at a time and then applying the same algorithm to go from $x(i)+sb_o$ to $x(i+s)$. If there exists an increasing sequence $0\le i_1 < i_2 < \cdots < i_m$ such that each leaf of the graded root $R_\Gamma$ contains at least one of the cycles $x(i_1),\dots,x(i_m)$, one can generate computations sequences between $x(i_s)$ and $x(i_{s+1})$ and then merge them to obtain a sequence which also carries the lattice homology of $Y_\Gamma$. This observation will be used to construct ``almost $I$-equivariant paths'' in \Cref{subsec: finding almost I-eqv path}.
\end{rem}

\subsection{Review of $S^1$ and $\mathrm{Pin}(2)$-lattice homotopy type} \label{subsec: s1 and pin(2) construction}
In this subsection, we review the construction of the $\operatorname{Pin}(2)$-lattice homotopy type. We will closely follow the arguments of \cite{DSS2023}.

We start by defining the weight function $w$ as follows. Given a $\mathrm{spin}^c$-structure $[k]\in\mathrm{Spin}^c(\partial W_\Gamma)$ and its element $k\in [k]$, we define its weight as 
\[
w(k) = \frac{1}{4}(c_1(k)^2 + n ).
\]
Also, given a pair of elements $k,k^\prime \in [k]$ which differ by $b_j$ for some $j$, we consider the pair as an ``edge'' $e_{k,k^\prime}$ and define its weight as
\[
w(e_{k,k^\prime}) = \mathrm{min}(w(k),w(k^\prime)).
\]
Then, given a sequence $\gamma = (x_1,\dots,x_m) \in (H^2(W_\Gamma; \Z))^m  $ such that for each $i$, $x_i$ and $x_{i+1}$ differ by $b_j$ for some $j$, we consider a CW-complex $\mathcal{F}(\gamma,h)$ for very big positive  even integers $h$ as
\begin{align}\label{cell str of path ht}
\mathcal{F}(\gamma,h) = \left( \left( \bigsqcup_{i=1,\dots,m} \left(\mathbb{C}^{\frac{w(x_i)+h}{2}} \right)^+ \right) \cup \left( \bigsqcup_{i=1,\dots,m-1} \left(\mathbb{C}^{\frac{w(e_{x_i,x_{i+1}})+h}{2}} \right)^+ \wedge [0,1] \right) \right) / \sim,
\end{align}
where we identify all basepoints, and furthermore, the points $x\sim (x,0)$ for $x\in \left(\mathbb{C}^{\frac{w(e_{x_i,x_{i+1}})+h}{2}} \right)^+$, considered as a point in $\left(\mathbb{C}^{\frac{w(x_i)+h}{2}} \right)^+$. We also similarly identify $y\sim (y,1)$ for $y\in \left(\mathbb{C}^{\frac{w(e_{x_i,x_{i+1}})+h}{2}} \right)^+$, considered as a point in $\left(\mathbb{C}^{\frac{w(x_{i+1})+h}{2}} \right)^+$. Then we define the \emph{path homotopy type} of $\gamma$ as the formal de-suspension
\[
\mathcal{H}(\gamma,[k]) = \Sigma^{-\frac{h}{2}\C }\mathcal{F}(\gamma,h),
\]
where this formal desuspension $\Sigma^{-\frac{h}{2}}$ is taken in a certain $S^1$-equivariant stable homotopy category. In our situation, we take it in $\mathfrak{C}_{O(2)}$. 
As the convenient notations, we abbreviate 
\[
\mathbb{S}(x_i) =   \left(\mathbb{C}^{\frac{w(x_i)+h}{2}} \right)^+ \qquad \text{ and }\qquad \mathbb{E}(e_{x_i, x_{i+1}} ) =  \left(\mathbb{C}^{\frac{w(e_{x_i,x_{i+1}})+h}{2}} \right)^+\wedge [0,1] \subset \Sigma^{\frac{h}{2}\C } \mathcal{H}(\gamma,[k]) .  
\]

This spectrum is naturally endowed with an $S^1$-action as follows: $S^1$ acts by complex multiplication on $\mathbb{C}$ and trivially on $[0,1]$. If $\gamma$ is a sequence which carries the lattice homology of $Y_\Gamma$, the homotopy type of $\mathcal{H}(\gamma,[k])$ depends only on the plumbing graph $\Gamma$ and the boundary $\mathrm{spin}^c$ structure $[k]$, and is defined as the \emph{$S^1$-lattice homotopy type} of $(\Gamma,[k])$.

To upgrade the symmetry group from $S^1$ to $\operatorname{Pin}(2)$, under the assumption that $[k]$ is self-conjugate, we have to choose the computation sequence carefully. We say that a computation sequence $\gamma$ is \emph{almost $J$-invariant} if it can be written as an amalgamation of three interior-disjoint paths
\[
\gamma = \gamma_0 \cup \gamma_\Theta \cup J\gamma_0,
\]
where $J$ acts by the negation map, i.e., $k\mapsto -k$ (see \cite[Section 6.1]{DSS2023} for more details). We make a remark that  $\gamma_0$ is a sub-path of $\gamma$ such that $J$-acts on $\gamma_0 \cup J \gamma_0 $ freely.) This negation map has a unique invariant \emph{lattice cube} $\square_J$; the condition here is that $\gamma_\Theta$ should be entirely contained in $\square_J$.


To construct an almost $J$-invariant computation sequence which carries the lattice homology of $Y_\Gamma$, we proceed as follows. We know from \cite[Theorem 1.1]{dai2019involutive} that $J$ acts on the leaves of the graded root $R_\Gamma$ by reflection; it has at most one invariant leaf. If an invariant leaf exists, it is the component containing the $J$-invariant cube $\square_J$. This cube has the following property: for any $\mathrm{spin}^c$ structure $\mathfrak{s}$ which is a vertex of $\square_J$, we have 
\[
c_1(\mathfrak{s}) = \mathfrak{s} - \bar{\mathfrak{s}} = (\text{spherical Wu class of }\Gamma).
\]

Choose a set $S$ of leaves of $R_\Gamma$ so that $S\cap JS=\emptyset$ and $S\cup JS$ is the set of all non-invariant leaves of $R_\Gamma$. Since the spherical Wu class of $W_\Gamma$ is a linear combination (with coefficients 0 or 1) of a subset of nodes of $\Gamma$ which do not contain any pairs of adjacent nodes, we can choose a base node $b_o$ so that every vertex of the cube $\square_J$ has zero coefficient for $b_o$, which implies that $x(0)$, and no other $x(i)$, is contained in $\square_J$. For each leaf $C\in S$, choose an integer $i_C\ge 0$ such that $x(i_C)\in R_\Gamma$, following \cite[Lemma 9.2]{Nemethi:2005-1}. Consider the set 
\[
I = \{ 0 \} \cup \{ i_C \mid C\in S \}
\]
and write it as $I=\{i_1,\dots,i_s\}$, $0=i_1 < \cdots < i_s$. Then one can take computation sequences between $x(i_t)$ and $x(i_{t+1})$ for each $t=1,\dots,s-1$ and amalgamate them to form a path $\gamma_0$. Then, by construction, $\gamma \cap J\gamma = \emptyset$. Then we can choose a path $\gamma_\Theta$ inside $\square_J$ which connects $x(0)$ and $Jx(0)$ and take the amalgamation
\[
\gamma = \gamma_0 \cup \gamma_\Theta \cup J\gamma_0.
\]

Here, $\gamma_\Theta$ is not really a ``path''. It consists of two points, which are a pair of opposite vertices in the invariant lattice cube $\square_J$. Then $\gamma_0$ is a path which starts from $\mathfrak{s}$. Its orbit under the $J$ action, which is conjugation, is $J\gamma_0$, and this path ends at $\mathfrak{s}' = \bar{\mathfrak{s}}$. Such paths are called \emph{almost $J$-invariant paths}, and they are central in the construction of $\operatorname{Pin}(2)$-lattice homotopy type.

We will slightly modify the construction of $\operatorname{Pin}(2)$-lattice homotopy type so that we can represent the action of $I$ on the $O(2)$-equivariant stable homotopy type $SWF(\Sigma_2(K), \mathfrak{s}_0)$. 

\subsection{Involutions on plumbed 4-manifolds and almost $I$-invariant paths}\label{almost I inv path}

Given an almost rational negative-definite plumbing graph $\Gamma$, the associated plumbed 4-manifold $W_\Gamma$, an orientation-preserving involution $\tau$ on $W_\Gamma$ with codimension two fixed point set, and a $\mathrm{spin}$ structure $\mathfrak{s}$ on $\partial W_\Gamma$ satisfying $\tau^\ast \mathfrak{s} = \bar{\mathfrak{s}}=\mathfrak{s} $, an \emph{almost $I$-invariant path} is a sequence of spin$^c$ structures on $W_\Gamma$
\[
\gamma = \{ \mathfrak{s}_{-n},\dots,\mathfrak{s}_{-1},\mathfrak{s}_1,\dots,\mathfrak{s}_n \}
\]
such that the following conditions are satisfied:
\begin{itemize}
    \item $\mathfrak{s}_i \vert_{\partial W_\Gamma} = \mathfrak{s}$ for all $i=1,\ldots,n,-1,\ldots,-n$;
    \item $\mathfrak{s}_{-i}=\tau^\ast\bar{\mathfrak{s}}_i$, for each 
 $i\in \{1,\dots,n$\};
    \item $\mathfrak{s}_{i+1}-\mathfrak{s}_i = PD[S]$, for each 
 $i\in \{1,\dots,n$\} and a sphere $S$ which represents a vertex of $\Gamma$;
 \item $\mathfrak{s}_{-i}-\mathfrak{s}_{-i-1}=PD[S]$, for each 
 $i\in \{1,\dots,n$\} and a sphere $S$ which represents a vertex of $\Gamma$;
    \item $\mathfrak{s}_1-\mathfrak{s}_{-1}=\sum_{S\in \mathcal{S}}PD[S]$ for a finite collection $\mathcal{S}$ of pairwise disjoint smoothly embedded spheres $S\subset W_\Gamma$ with $[S]^2<0$, where $\tau$ fixes $S$ setwise and acts on $S$ by either an orientation-preserving involution (which fixes two points) or identity.
\end{itemize}

Recall that, given a $\mathrm{spin}^c$-structure $\mathfrak{s}$ on $\partial W_\Gamma$, a path of $\mathrm{spin}^c$-structures on $W_\Gamma$ is said to \emph{carry the lattice homology of $(\Gamma,\mathfrak{s})$} if the obvious inclusion map
\[
\mathcal{H}(\gamma,\mathfrak{s})\hookrightarrow \mathcal{H}(\Gamma,\mathfrak{s})
\]
is a chain homotopy equivalence on $S^1$-equivariant Borel chain complexes. We will mainly consider almost $I$-invariant paths which carry the lattice homology of $(\Gamma,\mathfrak{s})$ in the proof of \cref{main computation}. 

\subsection{Construction of $O(2)$-action}\label{subsection:o2equivmap}

Under the above data and assumptions, and the existence of an almost $I$-invariant path $\gamma$, we will construct an $O(2)$-equivariant map 
\[
\mathcal{T}^{O(2)} \colon \mathcal{H}(\gamma, \fraks) \to SWF(Y_\Gamma, \fraks)
\]
which is an $S^1$-equivariantly homotopy equivalence for a given almost $I$-invariant path that carries the lattice homology. 

Due to \cref{O2rep}, we are allowed to choose the following universe for $O(2)$-equivariant Seiberg--Witten theory:
\[
\mathcal{U}=\mathbb{R}^\infty \oplus \tilde{\mathbb{R}}^\infty \oplus \mathbb{C}^\infty.
\]
Note that this universe induces the following universe when we restrict to $S^1$-equivariance:
\[
\mathcal{U}_{S^1}=\mathbb{R}^\infty \oplus \mathbb{C}^\infty.
\]


We suppose our AR-graph 4-manifold $W_\Gamma$ has an orientation-preserving involution $\tau$ with codimension two fixed point set. 
Let us have an almost $I$-equivariant path. 
In this setting, we define a class of $O(2)$-actions on the path homotopy type: 
\[
\Sigma^{\frac{h}{2}\C } \mathcal{H}(\gamma,\mathfrak{s}) =  \left(\mathbb{S}(\fraks_{-n}) \cup \cdots \cup \mathbb{E}(e_{\fraks_{-2} , \fraks_{-1}})\right)  \cup \left(\mathbb{S}(\fraks_{-1})  \cup   \mathbb{E}(e_{\fraks_{-1} , \fraks_{1}})\cup \mathbb{S}(\fraks_1) \right)  \cup \left( \mathbb{E}(e_{\fraks_{1} , \fraks_{2}}) \cup \cdots \cup  \mathbb{S}(\fraks_n) \right)  . 
\]
For the subgroup $S^1 \subset O(2)$, we define the $S^1$-actions as the usual complex multiplication on  
\[
\mathbb{S}(\mathfrak{s}_i) = \left(\mathbb{C}^{\frac{w(\mathfrak{s}_i)+h}{2}} \right)^+ \qquad \text{ and } \qquad \mathbb{E}(e_{\fraks_i, \fraks_{i-1}}) = \left(\mathbb{C}^{\frac{w(e_{\fraks_i,\fraks_{i+1}})+h}{2}} \right)^+ \wedge [0,1].
\]
Next, we define the action of $I \subset O(2)$ on $\mathcal{H}(\gamma,\mathfrak{s})$.  
For the vertices, when $i > 1$, we define anti-complex linear maps:
\begin{align*}
&I \colon \mathbb{S}(\fraks_i) \to \mathbb{S}(\fraks_{-i}), \\
&I \colon \mathbb{S}(\fraks_{-i}) \to \mathbb{S}(\fraks_i),
\end{align*}
given by complex conjugation. For the edges $\mathbb{E}(e_{\fraks_i,\fraks_{i+1}})$ with $i > 0$, we define similar actions:
\begin{align*}
&I \colon \mathbb{E}(e_{\fraks_i, \fraks_{i+1}}) \to \mathbb{E}(e_{\fraks_{-i}, \fraks_{-i-1}}), \\
&I \colon \mathbb{E}(e_{\fraks_{-i}, \fraks_{-i-1}}) \to \mathbb{E}(e_{\fraks_i, \fraks_{i+1}}),
\end{align*}
where the actions on $[0,1]$ are trivial, except for the central edge $\mathbb{E}(e_{\fraks_{-1}, \fraks_1})$.  
These actions anti-commute with the action of $i \in S^1$. For the central edge $\mathbb{E}(e_{\fraks_{-1}, \fraks_1})$, we define
\[
I \colon \left(\mathbb{C}^{(w(\fraks_{-1})+ h)/2} \right)^+ \wedge [0,1] \to \left(\mathbb{C}^{(w(\fraks_{-1})+ h)/2} \right)^+ \wedge [0,1]
\]
such that $I$ acts on $[0,1]$ by reflection, and on $\left(\mathbb{C}^{w(\fraks_{-1})/2} \right)^+$ by complex conjugation. Since all $I$-actions are compatible, we obtain a well-defined $O(2)$-action on $\mathcal{H}(\gamma,\mathfrak{s})$.  
With this action, we may regard $\Sigma^{\frac{h}{2}\mathbb{C}} \mathcal{H}(\gamma,\mathfrak{s})$, and therefore $\mathcal{H}(\gamma,\mathfrak{s})$, as objects in $\mathfrak{C}_{O(2)}$.

 \subsection{Proof of \cref{main computation}}
 
Now, we provide the construction of $\mathcal{T}^{O(2)} $ here, which gives the proof of \cref{main computation}. 
Note that we have a decomposition of the lattice homotopy type 
\[
\mathcal{H}(\gamma,\mathfrak{s}) = \Gamma_0 \cup \Gamma_{\theta}, 
\]
where 
\[
\Sigma^{\frac{h}{2}\C } \Gamma_{\theta} = \mathbb{S}(\fraks_{-1}) \cup \mathbb{E}(e_{\fraks_{-1},\fraks_{1}}) \cup \mathbb{S}(\fraks_{1})   
\]
and $\Gamma_0$ is the other part which has a free $I$-action. 
For each vertex $\fraks_i$ of an almost $I$-equivariant path $\gamma$, we associate the corresponding $U(1)$-equivariant Bauer--Furuta invariant 
\[
BF_{W_\Gamma, \fraks_i} \colon \mathbb{S}(\fraks_i) \to \Sigma^{\frac{h}{2} \mathbb{C} } SWF(Y, \fraks)
\]
with stabilizations by $\R$ and ${\C}$ for $i> 0 $. 
Let $\fraks_i$ and $\fraks_{i+1}$ be two successive vertices in $\gamma$, so that $\fraks_{i+1}= \fraks_i + 2v^*$ for some vertex $v$ of $\Gamma$, there $v^*$ denotes the homology class of the 2-handle core of $v$. For the edges with $i>1$, we use the following adjunction relation described in \cref{adjunction} and obtain a $U(1)$-equivariant homotopy 
\[
BF_{W_\Gamma, e_{\fraks_{i}, \fraks_{i-1}}} \colon \mathbb{E}(e_{\fraks_{i}, \fraks_{i-1}}) \to \Sigma^{\frac{h}{2} \C} SWF(Y). 
\]

Here we state the $U(1)$-adjunction relation proven in \cite[Proposition 3.15]{DSS2023}:

\begin{prop}\label{adjunction}
Let $(X, \fraks)$ be a smooth 4-dimensional $\mathrm{spin}^c$ cobordism from $(Y_0, \fraks_0)$ to $(Y_1, \fraks_1)$ with $b_1(X) = b_1(Y_i) = 0$.  
Suppose that $X$ contains an embedded sphere $S$ with $S \cdot S < 0$, and let $L$ be the complex line bundle on $X$ with $c_1(L) = \operatorname{PD}(S)$. Define
\[
\fraks' = \fraks \otimes L \qquad \text{ and } \qquad n := \frac{\langle c_1(\fraks), [S] \rangle + [S]^2}{2}.
\]
We write the $S^1$-equivariant Bauer--Furuta invariants of $\fraks$ and $\fraks'$ as maps:
\begin{align*}
&BF_{X, \fraks} \colon \left( \mathbb{C}^{\frac{c_1(\fraks)^2 - \sigma(X)}{8}} \right)^+ \wedge SWF(Y_0) \to SWF(Y_1), \\
&BF_{X, \fraks'} \colon \left( \mathbb{C}^{\frac{c_1(\fraks')^2 - \sigma(X)}{8}} \right)^+ \wedge SWF(Y_0) \to SWF(Y_1).
\end{align*}
Then $BF_{X, \fraks}$ and $U^n BF_{X, \fraks'}$ are $S^1$-stably homotopic if $n > 0$, and $U^{-n} BF_{X, \fraks}$ and $BF_{X, \fraks'}$ are $S^1$-stably homotopic if $n < 0$.  
Here, $U$ denotes the stable homotopy class of the map
\[
X \to \Sigma^{\mathbb{C}} X,
\]
defined by $x \mapsto (0, x)$.
\end{prop}

For $i < 0$, we obtain maps
\begin{align*}
&I_Y \circ BF_{W_\Gamma, \fraks_{-i}} \circ I \colon \mathbb{S}(\fraks_i) \to \Sigma^{\frac{h}{2} \mathbb{C}} SWF(Y, \fraks), \\
&I_Y \circ BF_{W_\Gamma, e_{\fraks_{-i}, \fraks_{-i+1}}} \circ I \colon \mathbb{E}(e_{\fraks_{i}, \fraks_{i-1}}) \to \Sigma^{\frac{h}{2} \mathbb{C}} SWF(Y, \fraks),
\end{align*}
where $I_Y$ denotes the real involution on $SWF(Y, \fraks)$. This yields a well-defined $O(2)$-equivariant map
\[
\mathcal{T}_0 \colon \Gamma_0 \to SWF(Y, \fraks).
\]
For the edge $\mathbb{E}(e_{\fraks_{-1},\fraks_{1}})$, we construct $U(1)$-equivariant homotopies between the maps
\begin{align*}
&I_Y \circ BF_{W_\Gamma, \fraks_{1}} \circ I \colon \mathbb{S}(\fraks_{-1}) \to \Sigma^{\frac{h}{2} \mathbb{C}} SWF(Y, \fraks) \text{ and } \\
&BF_{W_\Gamma, \fraks_{1}} \colon \mathbb{S}(\fraks_1) \to \Sigma^{\frac{h}{2} \mathbb{C}} SWF(Y, \fraks).
\end{align*}
In order to connect these maps, we use the following:

\begin{lem}\label{adjunctionwe}
After identifying $\mathbb{S}(\fraks_1) \cong \mathbb{S}(\fraks_{-1})$, the two maps
\begin{align*}
&I_Y \circ BF_{W_\Gamma, \fraks_{1}} \circ I \colon \mathbb{S}(\fraks_{-1}) \to \Sigma^{\frac{h}{2} \mathbb{C}} SWF(Y, \fraks), \\
&BF_{W_\Gamma, \fraks_{1}} \colon \mathbb{S}(\fraks_1) \to \Sigma^{\frac{h}{2} \mathbb{C}} SWF(Y, \fraks)
\end{align*}
are $S^1$-equivariantly homotopic.
\end{lem}




\begin{proof}
Note that we have $\tau^* \fraks_{1} \cong \overline{\fraks}_{-1}$ and $\tau^* \fraks_{-1} \cong \fraks_{1}$, and 
\[
\fraks_{1} - \fraks_{-1} = \operatorname{PD}([S])
\]
for some $\tau$-invariant, negatively embedded surface $S$ in $W_\Gamma$ which is described as the disjoint union of setwise $\tau$-fixed embedded spheres $S= S_1 \cup \cdots \cup S_a \subset W_\Gamma $ for $a \geq 1$. Let us denote by $\nu(S)$ a closed, $\tau$-invariant tubular neighborhood of $S$, which can be identified with the union of the total space of the disk bundle over $S^2$ with Euler number $p_j < 0$. Then the boundary of $\nu(S)$ is identified with the lens space $\displaystyle \bigcup_{1\leq j \leq a} L(p_j,1)$. Here, we use the orientation convention that $L(p, q)$ is obtained by $p/q$-surgery on the unknot.  We claim the following:
\begin{align}\label{uu}
\frac{c_1(\fraks_{-1}|_{\nu(S)})^2 - \sigma(\nu(S))}{4} 
= \frac{c_1(\fraks_{1}|_{\nu(S)})^2 - \sigma(\nu(S))}{4} 
= \sum_{1 \leq j \leq a} d\big(L(p_j,1), \fraks_{1}|_{L(p,1)}\big).
\end{align}
For proving \eqref{uu}, it is enough to see 
\[
\frac{c_1(\fraks_{-1}|_{\nu(S_j)})^2 - \sigma(\nu(S_j))}{4} 
= \frac{c_1(\fraks_{1}|_{\nu(S_j)})^2 - \sigma(\nu(S_j))}{4} 
=  d\big(L(p_j,1), \fraks_{1}|_{L(p_j,1)}\big)
\]
for each $j\in \{1, \dots, a\}$ where $\nu(S_j)$ is a $\tau$-invariant closed neighborhood of $S_j$. 

The inequality version of \eqref{uu} is nothing but a Fr{\o}yshov-type inequality in Heegaard Floer theory. 
The following formula for the Heegaard Floer $d$-invariant is known \cite[Section~4]{OzSz03b}:
\[
d(L(p,1), [i]) = -\frac{1}{4} + \frac{(2i - p)^2}{4p}
\]
for $i = 0, \ldots, p-1$, where $[i]$ denotes the unique $\mathrm{spin}^c$ structure on $L(p,1)$ that extends over the $p$-trace $O(p)$ of the unknot and satisfies $c_1 = 2i - p$. Notice that, when restricted to $\nu(S_j)$, we have
\[
c_1(\mathfrak{s}_1)\vert_{\nu(S_j)} = (\mathfrak{s}_1 - \bar{\mathfrak{s}}_1)\vert_{\nu(S_j)} = (\text{spherical Wu class of } W_{\Gamma_{p,q}})\vert_{\nu(S_j)} = PD_{\nu(S_j)}[S_j],
\]
where the Poincar\'{e} dual is taken in $\nu(S_j)$. Since the boundary of $\nu(S_j)$ is $L(p_j,1)$, we have $PD_{\nu(S_j)}[S_j] = p_j$, and thus
\[
c_1(\mathfrak{s}_1\vert_{\nu(S_j)}) = p_j \qquad \text{ and} \qquad c_1(\mathfrak{s}_{-1}\vert_{\nu(S_j)}) = -p_j,
\]
i.e.,
\[
c_1(\mathfrak{s}_1\vert_{\nu(S_j)})^2 = c_1(\mathfrak{s}_{-1}\vert_{\nu(S_j)})^2 = p_j \qquad \text{ and } \qquad \mathfrak{s}_1\vert_{\partial \nu(S_j)} = [0].
\]
Hence, we have
\[
\frac{c_1(\mathfrak{s}_{-1}|_{\nu(S_j)})^2 - \sigma(\nu(S_j))}{4} 
= \frac{c_1(\mathfrak{s}_1|_{\nu(S_j)})^2 - \sigma(\nu(S_j))}{4} = \frac{p_j-1}{4} = d(L(p_j,1), \mathfrak{s}_1\vert_{L(p_j,1)}),
\]
which proves the claim.

From \eqref{uu}, the $U(1)$-equivariant Bauer--Furuta invariants for $\nu(S)$ with $\fraks_1|_{\nu(S)}$ or $\fraks_{-1}|_{\nu(S)}$ are regarded as stable homotopy classes of $U(1)$-equivariant maps
\begin{align*}
   & BF_{\nu(S), \fraks_1|_{\nu(S)}} \colon \mathbb{C}^m \to \mathbb{C}^m, \\
   & BF_{\nu(S), \fraks_{-1}|_{\nu(S)}} \colon \mathbb{C}^m \to \mathbb{C}^m
\end{align*}
for some $m$. By the observation given in \cite[Proof of Proposition 3.15]{DSS2023}, one can see that $BF_{\nu(S), \fraks_1|_{\nu(S)}}$ and $BF_{\nu(S), \fraks_{-1}|_{\nu(S)}}$ are $U(1)$-equivariantly stably homotopic to the identity. Therefore, by the gluing result for $U(1)$-equivariant Bauer--Furuta invariants, we see that
$BF_{W_\Gamma, \fraks_1}$ and $BF_{W_\Gamma \smallsetminus \overset{\circ}{\nu}(S), \fraks_1}$ (resp. $BF_{W_\Gamma, \fraks_{-1}}$ and $BF_{W_\Gamma \smallsetminus \overset{\circ}{\nu}(S), \fraks_{-1}}$) are $U(1)$-equivariantly stably homotopic, where $\overset{\circ}{\nu}(S)$ denotes the interior of $\nu(S)$. It also follows that we can identify the domains of the maps $BF_{W_\Gamma \smallsetminus \overset{\circ}{\nu}(S), \fraks_1}$ and $BF_{W_\Gamma \smallsetminus \overset{\circ}{\nu}(S), \fraks_{-1}}$, since these spin$^c$ structures coincide over $W_\Gamma \smallsetminus \overset{\circ}{\nu}(S)$. We fix such a geometric identification $\mathbb{S}(\mathfrak{s}_1) = \mathbb{S}(\mathfrak{s}_{-1})$.

It is then sufficient to show that
\begin{align*}
&I_Y \circ BF_{W_\Gamma \smallsetminus \overset{\circ}{\nu}(S), \fraks_1} \circ I \colon \mathbb{S}(\fraks_{-1}) \to \Sigma^{\frac{h}{2} \mathbb{C}} SWF(Y, \fraks), \\
&BF_{W_\Gamma \smallsetminus \overset{\circ}{\nu}(S), \fraks_1} \colon \mathbb{S}(\fraks_1) \to \Sigma^{\frac{h}{2} \mathbb{C}} SWF(Y, \fraks)
\end{align*}
are $U(1)$-equivariantly stably homotopic. Since
\[
\fraks_1 \big|_{W_\Gamma \smallsetminus \overset{\circ}{\nu}(S)} = \fraks_{-1} \big|_{W_\Gamma \smallsetminus \overset{\circ}{\nu}(S)} =: \fraks \qquad \text{ and } \qquad \tau^* \fraks_1 \cong \overline{\fraks}_{-1},
\]
we have a real structure on $\fraks$ over $W_\Gamma \smallsetminus \overset{\circ}{\nu}(S)$. 
We take a norm-preserving, anti-complex linear involution $\widetilde{\tau} \colon S_{\fraks} \to S_{\fraks}$ lifting $\tau$ which is compatible with the Clifford multiplication for $\fraks$. This shows that we have a representative of $BF_{W_\Gamma \smallsetminus \overset{\circ}{\nu}(S), \fraks_1}$ such that 
$BF_{W_\Gamma \smallsetminus \overset{\circ}{\nu}(S), \fraks_1}$ is $I$-equivariant, i.e.,
\[
I_Y \circ BF_{W_\Gamma \smallsetminus \overset{\circ}{\nu}(S), \fraks_1} \circ I' = BF_{W_\Gamma \smallsetminus \overset{\circ}{\nu}(S), \fraks_1}
\]
for a certain anti-complex linear map
\[
I' \colon \mathbb{S}(\fraks_1) \to \mathbb{S}(\fraks_{-1}) = \mathbb{S}(\fraks_1)
\]
induced from $\widetilde{\tau}$.
By a certain complex base change that is isotopic to the identity (see \cref{O2rep} and \cref{auto}), the actions $I$ and $I'$ can be identified. This completes the proof.
\end{proof}

For \cref{adjunctionwe}, we  have an $S^1$-equivariant map 
\[
\mathcal{T}\colon \Gamma_{\theta}= \mathbb{E}({e_{\fraks_{-1},\fraks_{1}}}) \to SWF(Y)
\]
and 
it defines an $S^1$-equivariant map 
\[
\mathcal{T}\colon \mathcal{H}(\gamma,\mathfrak{s}) \to SWF(Y)
\]
which is $O(2)$-equivariant except for the central edge $\mathbb{E}({e_{\fraks_{-1},\fraks_{1}}})$.  In order to give an $O(2)$-equivariant map
\[
\mathcal{T}^{O(2)}\colon \mathcal{H}(\gamma,\mathfrak{s}) \to SWF(Y), 
\]
we need to modify this construction.  The main strategy is almost the same as the $\operatorname{Pin}(2)$ case in \cite{DSS2023}. Since the argument is almost the same as that given in \cite{DSS2023}, we just write a flow of the proof and which part is different. 
First, we define 
\[
\Theta := \operatorname{Cone} \left( \mathcal{T}_0 \colon \Gamma_0 \to SWF(Y)\right) 
\]
and regard it as an $O(2)$-space.
Note that we have the following diagram of equivariant cofibration sequences: 
\begin{align}\label{cof seq}
  \begin{CD}
 \cdots @>>>   \Gamma_0 @>>>   \mathcal{H}(\gamma,\mathfrak{s}) @>>> \Sigma^{\wt{\R}} S^0 @>>>  \cdots  \\
   @VVV    @VVV @V{\mathcal{T}}VV  @VVV  @VVV  \\
 \cdots @>>>   \Gamma_0    @>{\mathcal{T}^{O(2)}|_{\Gamma_0}}>> SWF(Y)  @>>>  \Theta  @>>>  \cdots
  \end{CD}
\end{align}

Since $\mathcal{T}$ is $S^1$-equivariantly homotopic, the $O(2)$-space $\Theta$ is $S^1$-equivariantly homotopic to $\Sigma^{\wt{\R}} S^0$ up to stabilization by $\C$. Let us stabilize $\Theta$ so that $\Theta$ is $S^1$-equivariantly homotopy equivalent to  $\Sigma^{\wt{\R}} S^0$. 
We will prove $\mathcal{T}$ admits an $O(2)$-equivariant lift. 
 If there is an $O(2)$-equivariant homotopy equivalence  
\[
\Theta \to \Sigma^{\wt{\R}} S^0
\]
which satisfies the commutativity 
\[
  \begin{CD}
     \Sigma^{\wt{\R}}S^0 @>>> \Sigma^\R \Gamma_0  \\
  @VVV    @VVV \\
     \Theta  @>>>  \Sigma^\R \Gamma_0
  \end{CD}
\]
up to $O(2)$-homotopy, then from the identifications
\[
    \Sigma^\R \mathcal{H}(\gamma,\mathfrak{s}) = \operatorname{Cone} \left(\Sigma^{\wt{\R}}S^0 \to \Sigma^\R \Gamma_0\right) \qquad\text{ and }\qquad
    \Sigma^\R SWF(Y) = \operatorname{Cone} \left(\Theta \to \Sigma^\R \Gamma_0\right),
    \]
    we see $\mathcal{T}$ has an $O(2)$-equivariant lift. 
More precisely, we consider the following steps: 
\begin{itemize}
    \item[\bf Step 1]: First we prove there is $\Z_2 = \langle I \rangle$ equivariant homotopy equivalence: 
    \begin{align}\label{S1fix}
         \Theta^{S^1} \cong_{\Z_2} \Sigma^{\wt{\R}} S^0. 
    \end{align}
 This statement is corresponding to \cite[Lemma 6.5]{DSS2023}. Since several techniques \cite[(Ho-3), (6.6) in the $O(2)$-setting]{DSS2023} to see \eqref{S1fix} in the $\operatorname{Pin}(2)$-setting can also work for $O(2)$, we have the desired result. 
    \item[\bf Step 2]: Next, we prove that, for sufficiently large $p$ and $q$, there exists an $O(2)$-map 
    \[
    M \colon \Theta \to \Sigma^{q\wt{\R} \oplus p\C} SWF(Y)
    \]
    which induces homotopy equivalence in $O(2)$-fixed point spectra. This statement is an analog of \cite[Lemma 6.6]{DSS2023}. 
Here we use the following facts: 
\begin{itemize}
    \item The vanishing result 
    \[
    \left[\Sigma^{r \R  } \Gamma_0, \Sigma^{r \R \oplus q\wt{\R}\oplus p \mathbb{C}} S^0\right]_{O(2)}=0
    \]
    for sufficiently large $p$, $q$ and $r$.  
    It follows from \cite[Proposition 4.2]{Ada84}. Since we are working in $\mathfrak{C}_{O(2)}$, we omit $\Sigma^{r \R  }$. 

    \item For sufficiently large $p$ and $q$, there is an $O(2)$-equivariant map 
    \[
   N\colon  SWF(Y) \to \Sigma^{r \R \oplus q\wt{\R}\oplus  p \mathbb{C}} S^0 
    \]
    for some $r$. 
    This comes from the $O(2)$ Bauer--Furuta invariant for the double-branched cover of an oriented surface $S$ (with high genus) in $D^4$ bounded by $K$ with respect to its unique spin structure.  
   Note that $r$ is zero in this situation since $r = b_2^+(\Sigma_2 (S)/ \Z_2 ) = b^+_2 ( D^4 )  =0 $.  
    Since it is spin structure, the Bauer--Furuta invariant has a $\operatorname{Pin}(2)\times_{\Z_2 } \Z_4$-symmetry (as the maximal symmetry, see \cite{Mon22}), but we just forget by the homomorphism $S^1\rtimes \Z_2\cong O(2) \to \operatorname{Pin}(2)\times_{\Z_2 } \Z_4$, defined by 
    \[
   (u,0) \mapsto (u,0) \qquad \text{ and }\qquad  \   (u, 1) \mapsto (j u, j), 
    \]
    where $j$ (resp. $1$) denotes the generator of $\Z_4$ (resp. $\Z_2$).
\end{itemize} 

    \item[\bf Step 3]:
 We reduce the numbers $p$ and $q$ so that we have 
\[
M' :  \Theta \to \Sigma^{\wt{\R}}S^0 
\]
and $M'$ induces homotopy equivalence for $S^1$- and $O(2)$- fixed point parts and take an O(2)-equivariant map 
\[
f : \Sigma^{\wt{\R}}S^0  \to \Sigma^\R \Gamma_0 
\]
which satisfies the commutativity 
\[
  \begin{CD}
\Sigma^{\wt{\R}}S^0 @>{f}>> \Sigma^\R \Gamma_0  \\
  @VVV    @VVV \\
     \Theta  @>>>  \Sigma^\R \Gamma_0
  \end{CD}
\]
up to $O(2)$-homotopy and $f \circ M' = g$, where $f$ and $g$ are attaching maps of the cofibrations  in \eqref{cof seq}. 
This is again an analog of \cite[Lemma 6.7]{DSS2023}. 
 Since $O(2)$-analogs of \cite[(6.8), (ho-03), Theorem 6.4]{DSS2023} are still true, their argument still works in our setting. 
\end{itemize}

\subsection{Almost $I$-equivariant path for even torus knots} \label{subsec: finding almost I-eqv path}

Given a torus knot $K = T_{p,q}$ where $p, q > 0$ and $p$ is even, applying Seifert's algorithm to the Brieskorn sphere $\Sigma_2(K) = \Sigma(2, p, q)$ shows that the canonical Seifert action on $\Sigma(2, p, q)$ has three singular fibers, two of which are identical. This data can be translated into an almost rational, negative-definite plumbing graph $\Gamma_{p,q}$ with three legs, where two of the legs are identical. We may depict $\Gamma_{p,q}$ as follows. For simplicity, we refer to the node of weight $-a_i$ as the $a_i$-node, the node of weight $-b_j$ in the upper right leg as the upper $b_j$-node, the node of weight $-b_j$ in the lower right leg as the lower $b_j$-node, and the node of weight $-c$ as the central node.

\[
\begin{tikzpicture}[xscale=1.5, yscale=1, baseline={(0,-0.1)}]
    \node at (-0.1, 0.3) {$-c$};
    \node at (-1, 0.3) {$-a_1$};
    \node at (-2.5, 0.3) {$-a_m$};
    
    \node at (1, 1.3) {$-b_1$};
    \node at (2.5, 1.3) {$-b_n$};

    \node at (1, -0.7) {$-b_1$};
    \node at (2.5, -0.7) {$-b_n$};
    
    \node at (0, 0) (A0) {$\bullet$};
    \node at (1, 1) (A1) {$\bullet$};
    \node at (2.5, 1) (A2) {$\bullet$};

    \node at (-1, 0) (B1) {$\bullet$};
    \node at (-2.5, 0) (B2) {$\bullet$};
    
    \node at (1, -1) (C1) {$\bullet$};
    \node at (2.5, -1) (C2) {$\bullet$};
    
    \draw (A0) -- (A1);
    \draw[dotted] (A1) -- (A2);
    \draw (A0) -- (B1);
    \draw[dotted] (B1) -- (B2);
    \draw (A0) -- (C1);
    \draw[dotted] (C1) -- (C2);
    \end{tikzpicture}
\]
The integers $c, a_1, \ldots, a_m, b_1, \ldots, b_n$ satisfy the following relations: there exist integers $\beta_2$ and $\beta_3$ such that $1 \leq \beta_2 < \frac{p}{2}$ and $1 \leq \beta_3 < q$, and
\[
\frac{p}{2\beta_2} = a_1 - \cfrac{1}{a_2 - \cfrac{1}{\ddots - \cfrac{1}{a_m}}}, \qquad
\frac{q}{\beta_3} = b_1 - \cfrac{1}{b_2 - \cfrac{1}{\ddots - \cfrac{1}{b_n}}},
\]
together with the Seifert condition
\[
\frac{cpq}{2} + q\beta_2 + p\beta_3 = 1.
\]

We now describe an involution $\tau$ on the corresponding plumbed 4-manifold $W_{\Gamma_{p,q}}$. We start with the central vertex: choose a 2-sphere $S_c$, equipped with a rotation involution fixing two points, namely the north pole $P^+_c$ and the south pole $P^-_c$. Choose a disk $D^{+,c}$ centered at $P^+_c$, such that it is setwise fixed under the rotation involution. Furthermore, choose disjoint disks $D^{-,c}_u, D^{-,c}_l \subset S_c \smallsetminus D^{+,c}$ that are swapped under the rotation involution.

Similarly, for each $a_i$-node, choose a 2-sphere $S_{a_i}$, again equipped with a rotation involution fixing the north pole $P^{+,a_i}$ and the south pole $P^{-,a_i}$. Choose disjoint disks $D^{+,a_i}$ and $D^{-,a_i}$, centered at $P^{+,a_i}$ and $P^{-,a_i}$ respectively, so that they are setwise fixed under the rotation involution.

Next, for each $j = 1, \dots, n$, we choose a disjoint union of two 2-spheres, i.e., $S_{b_j} = S^2 \sqcup S^2$, equipped with an involution that swaps the two components. Denote by $D^{+,b_j}_u$ and $D^{-,b_j}_l$ the (disjoint) disks centered at the north and south poles of the first and second components of $S_{b_j}$, respectively.

Now we build a symmetric model of $W_{\Gamma_{p,q}}$ as follows.
\begin{itemize}
    \item Consider the disk bundle $p_c \colon E_c \rightarrow S_c$ of Euler number $-c$. We choose a lift of the rotation involution on $S_c$ to the bundle $E_c$ so that its restriction to the trivial bundle $p_c^{-1}(D^{+,c}) \cong D^2 \times D^{+,c}$ is given by $(x, y) \mapsto (x, -y)$.
    
    \item For each $i = 1, \dots, n$, consider the disk bundles $p_{a_i} \colon E_{a_i} \rightarrow S_{a_i}$ of Euler number $-a_i$, and $p_{b_j} \colon E_{b_j} \rightarrow S_{b_j}$ of Euler number $-b_j$ (on both components of $S_{b_j}$). Note that for each such node, the given involution on the corresponding sphere (or pair of spheres) admits several possible lifts to its associated disk bundle.
    
    \item For each $i = 1, \dots, n - 1$, glue the total space of $E_{a_i}$ to $E_{a_{i+1}}$ by identifying $p_{a_i}^{-1}(D^{+,a_i})$ and $p_{a_{i+1}}^{-1}(D^{+,a_{i+1}})$ via the coordinate-swapping diffeomorphism:
    \[
    \mathrm{swap} \colon D^2 \times D^2 \xrightarrow{(x, y) \mapsto (y, x)} D^2 \times D^2.
    \]
    Similarly, glue the total space of $E_{b_j}$ to $E_{b_{j+1}}$ for $j = 1, \dots, m - 1$.
    
    \item Next, glue the total spaces of $E_{a_1}$ and $E_{b_1}$ as follows. First, identify $p_{a_1}^{-1}(D^{-,a_1})$ with $p_c^{-1}(D^{+,c})$ via the coordinate-swapping diffeomorphism. Then identify $p_{b_1}^{-1}(D^{+,b_1})$ with $p_c^{-1}(D^{-,c}_u \sqcup D^{-,c}_l)$ via the following diffeomorphism:
    \[
    (D^2 \sqcup D^2) \times D^2 = (D^2 \times D^2) \sqcup (D^2 \times D^2) 
    \xrightarrow{\mathrm{swap} \sqcup \mathrm{swap}} 
    (D^2 \times D^2) \sqcup (D^2 \times D^2) = D^2 \times (D^2 \sqcup D^2).
    \]
    
    \item The involution on $E_c$ determines a unique choice of linear involution on $E_{a_1}$, which in turn inductively determines linear involutions on each $E_{a_i}$. Similarly, we can also determine compatible linear involutions on each $E_{b_j}$.
\end{itemize}

The resulting 4-manifold is clearly $W_{\Gamma_{p,q}}$, and the involutions on each disk bundle induce a smooth involution $\tau$ on $W_{\Gamma_{p,q}}$. Observe that
\[
\mathrm{Fix}(\tau) = p_c^{-1}(P^{-,c}) \sqcup \text{disjoint embedded spheres}.
\]
Here, the sphere components of $\mathrm{Fix}(\tau)$ are contained in the interior and thus do not intersect $\partial W_{\Gamma_{p,q}}$. Hence, the fixed point set of the action of $\tau$ on $\partial W_{\Gamma_{p,q}}$ is given by the boundary of the disk $p_c^{-1}(P^{-,c})$.  By the following lemma, we may identify $\partial W_{\Gamma_{p,q}}$ with $\Sigma_2(K)$, where the involution $\tau$ (restricted from its action on $W_{\Gamma_{p,q}}$) corresponds to the deck transformation on $\Sigma_2(K)$. Thus, we will henceforth simply denote the deck transformation on $\Sigma_2(K)$ by $\tau$.

\begin{lem} \label{lem:branching-set-torus-knot}
The pair $(\partial W_{\Gamma_{p,q}}, \tau)$ is equivariantly diffeomorphic to $(\Sigma_2(K), \text{deck transformation})$.
\end{lem}

\begin{proof}
    We start by drawing the action of $\tau$ on $\partial W_{\Gamma_{p,q}}$ in terms of surgery diagrams, as shown in \Cref{fig:kirby1}. Note that we have chosen the weight of the central node to be $-2$, i.e., $c=2$, at the expense of modifying the rational surgery slope of the invariant component. The action of $\tau$ can be seen on the given surgery diagram as the $180^\circ$ rotation about the vertical surgery curve of slope $\frac{p}{2\beta_2}$. Hence, the quotient manifold $\partial W_{\Gamma_{p,q}}/\tau$ can be drawn as in \Cref{fig:kirby2}; note that the branching set $K$ is the meridian of the central surgery curve of slope $-1$. To prove the lemma, it suffices to show that $(\partial W_{\Gamma_{p,q}}, K)$ is diffeomorphic to $(S^3, T_{p,q})$.

\begin{figure}[h]
\centering
\includegraphics[height=.28\linewidth]{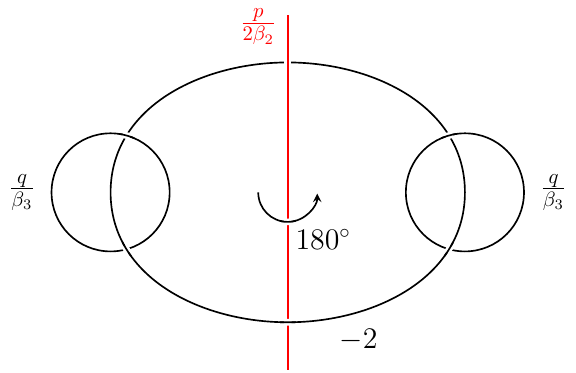}
\caption{A surgery diagram for $\partial W_{\Gamma_{p,q}}$. The action of $\tau$ can be seen as the $180^\circ$ rotation about the red vertical surgery curve. Here, $\beta_2$ and $\beta_3$ are negative integers satisfying the equation
$pq + q\beta_2 + p\beta_3 = 1.$
}
\label{fig:kirby1}
\end{figure}

    It can be readily checked that $\partial W_{\Gamma_{p,q}}/\tau \cong S^3$. To see this, one can reverse the slam-dunk moves to transform the $\frac{p}{\beta_2}$-framed unknotted curve into a chain of integral-framed unknotted curves, and then perform slam-dunk moves starting from the leftmost $\frac{q}{\beta_3}$-framed unknotted curve. This process yields an unknotted curve with framing $r$, where the numerator satisfies
\[
pq + q\beta_2 + p\beta_3 = 1,
\]
which implies that $\partial W_{\Gamma_{p,q}}/\tau \cong S^3$.



    It remains to show that, under the identification $\partial W_{\Gamma_{p,q}}/\tau \cong S^3$, the branching set $K$ corresponds to the torus knot $T_{p,q}$. While there is a direct way to see this by carefully following the slam-dunk moves (or by invoking the classification of involutions on $\Sigma(2, p, q)$ that commute with the Seifert $S^1$-action), we provide an indirect proof for simplicity. We begin by computing the difference between the blackboard framing of $K$ as shown in \Cref{fig:kirby2} and the Seifert framing of $K$. Observe that performing surgery along $K$ with respect to its blackboard framing yields $L(p, \beta_2) \# L(q, \beta_3)$. By \cite[Theorem 1.5]{greene2015space}, we know that the surgery slope must be $pq$ (with respect to the Seifert framing). Hence, the surgery diagram in \Cref{fig:kirby3} describes $S^3_{pq+N}(K)$ for any integer $N$.

\begin{figure}[h]
\centering
\includegraphics[height=.26\linewidth]{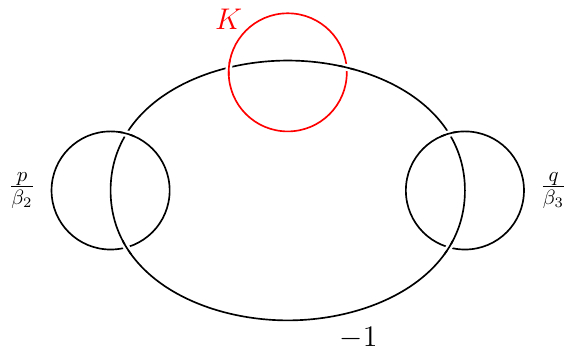}
\caption{A surgery diagram of $\partial W_{\Gamma_{p,q}}/\tau$. The branching set $K$, namely the image of $\mathrm{Fix}(\tau)$ under the projection $\partial W_{\Gamma_{p,q}} \to \partial W_{\Gamma_{p,q}}/\tau$, is drawn in red.}
\label{fig:kirby2}
\end{figure}

\begin{figure}[h]
\centering
\includegraphics[height=.26\linewidth]{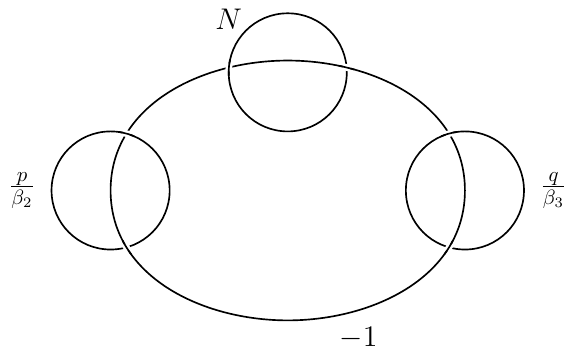}
\caption{A surgery diagram for the $(pq+N)$-surgery along $K$.
}
\label{fig:kirby3}
\end{figure}

    Now, whenever $N > 1$, the manifold $S^3_N(T_{p,q})$ is Seifert fibered, and its Seifert invariants can be computed by following the proof of \cite[Proposition 3.1]{moser1971elementary}. A straightforward computation shows that the resulting Seifert invariants agree with those that can be read off from \Cref{fig:kirby3}. Hence, we have shown that
\[
S^3_N(K) \cong S^3_N(T_{p,q}) \qquad \text{ for all integers } \qquad N > 1.
\]
Since any slope greater than $\frac{30(p^2-1)(q^2-1)}{67}$ is a characterizing slope for $T_{p,q}$ by \cite[Theorem 1.3]{ni2014characterizing}, we conclude that $K = T_{p,q}$, as desired. \end{proof}


Observe that the action of $\tau$ fixes all spheres $S_c$, $S_{a_i}$, and $S_{b_j}$, embedded in $W_{\Gamma_{p,q}}$ as the zero-sections of the corresponding disk bundles, setwise. Furthermore, the Poincaré dual of the spherical Wu class of $W_{\Gamma_{p,q}}$ is the sum of spheres corresponding to a collection of pairwise non-adjacent nodes of $\Gamma_{p,q}$. Hence, we may write
\[
\mathrm{Fix}(\tau) = S^a_1 \sqcup \cdots \sqcup S^a_k \sqcup S^b_1 \sqcup \cdots \sqcup S^b_s \sqcup \tau S^b_1 \sqcup \cdots \sqcup \tau S^b_s,
\]
where $S^a_i$ are the zero-sections of disk bundles corresponding to either the central node or nodes in the invariant leg of $\Gamma_{p,q}$ that support the spherical Wu class, $S^b_j$ are the zero-sections of disk bundles corresponding to nodes in the upper right leg of $\Gamma_{p,q}$ that support the spherical Wu class, and $\tau S^b_j$ is the image of $S^b_j$ under $\tau$.  Note that, since the spherical Wu class is symmetric with respect to the action of $\tau$, the spheres $\tau S^b_j$ are also the zero-sections of disk bundles corresponding to nodes in the lower right leg of $\Gamma_{p,q}$ that support the spherical Wu class.

We may assume that, among the spheres $S^a_1, \ldots, S^a_k$, $S^a_1$ corresponds to the rightmost node, i.e., all other nodes are located to its left. Similarly, among the spheres $S^b_1, \ldots, S^b_s$, we may assume that $S^b_1$ corresponds to the leftmost node. We divide into several cases.

\begin{caseof}
\case{$k > 0$ and $S^a_1$ is supported on the central node of $\Gamma_{p,q}$.}{We choose a smooth path $\gamma_1$ from a non-$\tau$-invariant point of $S^a_1$ to the north pole of $S^b_1$; we may then perturb $\gamma_1$ to ensure that $\gamma_1 \cap \tau \gamma_1 = \emptyset$. For each $j = 2, \ldots, s$, we choose a smooth path $\gamma_j$ from the south pole of $S^b_{j-1}$ to the north pole of $S^b_j$. Then $\tau \gamma_j$ is a smooth path from the south pole of $\tau S^b_{j-1}$ to the north pole of $\tau S^b_j$, and clearly $\gamma_j \cap \tau \gamma_j = \emptyset$.  Hence, we may connect the spheres $S^a_1, S^b_1, \ldots, S^b_s, \tau S^b_1, \ldots, \tau S^b_s$ by tubing along the curves $\gamma_1, \ldots, \gamma_s, \tau \gamma_1, \ldots, \tau \gamma_s$ to obtain a smoothly embedded, setwise $\tau$-invariant sphere $\hat{S}$ whose homology class satisfies
\[
[\hat{S}] = [S^a_1] + [S^b_1] + \cdots + [S^b_s] + [\tau S^b_1] + \cdots + [\tau S^b_s].
\]
Observe that the action of $\tau$ on $\hat{S}$ is a rotation involution with two fixed points, and that $\hat{S}$ is disjoint from the spheres $S^a_2, \ldots, S^a_k$. Clearly, the $\tau$-action on these spheres is either a rotation involution or trivial, and the spherical Wu class is given by the Poincar\'{e} dual of the sum of the homology classes of the spheres $\hat{S}, S^a_2, \ldots, S^a_k$.}
\case{$k > 0$ and $S^a_1$ is not supported on the central node of $\Gamma_{p,q}$.}{We choose a smooth path $\gamma_1$ from the south pole of the zero-section $S_c$ of the central node of $\Gamma_{p,q}$ to the north pole of $S^b_1$. By perturbing $\gamma_1$ if necessary, we may assume that $\gamma_1 \cap \tau \gamma_1 = \emptyset$. Then $\gamma_1 \cup \tau \gamma_1$ is a piecewise smooth, setwise $\tau$-invariant path between the north pole of $S^b_1$ and the north pole of $\tau S^b_1$. We can then smooth it to obtain a smooth, setwise $\tau$-fixed path $\hat{\gamma}$. For each $j = 2, \ldots, s$, we choose a smooth path $\gamma_j$ from the south pole of $S^b_{j-1}$ to the north pole of $S^b_j$. Then $\tau \gamma_j$ is a smooth path from the south pole of $\tau S^b_{j-1}$ to the north pole of $\tau S^b_j$, and clearly $\gamma_j \cap \tau \gamma_j = \emptyset$.
We can then connect the spheres $S^b_1, \ldots, S^b_s, \tau S^b_1, \ldots, \tau S^b_s$ along the paths $\hat{\gamma}, \gamma_2, \ldots, \gamma_s, \tau \gamma_2, \ldots, \tau \gamma_s$ to obtain a smoothly embedded, setwise $\tau$-fixed sphere $\hat{S}$ whose homology class satisfies
\[
[\hat{S}] = [S^b_1] + \cdots + [S^b_s] + [\tau S^b_1] + \cdots + [\tau S^b_s].
\]
Observe that the action of $\tau$ on $\hat{S}$ is a rotation involution with two fixed points, and that $\hat{S}$ is disjoint from the spheres $S^a_1, \ldots, S^a_k$. Clearly, the $\tau$-action on these spheres is either a rotation involution or trivial, and the spherical Wu class is given by the Poincar\'{e} dual of the sum of the homology classes of the spheres $\hat{S}, S^a_1, \ldots, S^a_k$.}
\case{$k = 0$.}{This case is essentially the same as Case 2, which was already discussed above.}
\end{caseof}

Therefore, in any case, we can find a finite collection $\mathcal{S}$ of pairwise disjoint, smoothly embedded spheres $S$ in $W_{\Gamma_{p,q}}$ such that $[S]^2 < 0$, $\tau$ acts on $S$ by either the rotation involution or the identity, and
\[
\text{the spherical Wu class of } W_{\Gamma_{p,q}} = \sum_{S \in \mathcal{S}} [S].
\]

\begin{lem}
Let $p, q$ be coprime positive integers with $p$ even. Consider the plumbing graph $\Gamma_{p,q}$ and the action of $\tau$ on the associated 4-manifold $W_{\Gamma_{p,q}}$. Denote by $\mathfrak{s}$ the unique self-conjugate $\mathrm{spin}^c$ structure on $\partial W_{\Gamma_{p,q}} \cong \Sigma(2, p, q)$. Then, with respect to these data, there exists an almost $I$-invariant path which carries the lattice homology of $(\Gamma_{p,q}, \mathfrak{s})$.
\end{lem}

\begin{proof}
We simply follow the procedure described in \Cref{subsec: s1 and pin(2) construction}, with a minimal modification, to construct an almost $J$-invariant computation sequence $\gamma_J$ which carries the lattice homology of $(\Gamma_{p,q}, \mathfrak{s})$, where $\mathfrak{s}$ denotes the unique self-conjugate $\mathrm{spin}^c$ structure on $\partial W_{\Gamma_{p,q}} \cong \Sigma(2, p, q)$. Recall that such a sequence is obtained by connecting the cycles $x(i_t)$, each defined as the minimal cycle whose coefficient at the (arbitrarily chosen; the minimality condition (discussed in \Cref{subsec:computations sequences}) depends on this choice) base vertex $b_o$ is $i_t$. Since the spherical Wu class consists of zero-sections of disk bundles corresponding to nodes that are pairwise non-adjacent, we can choose $b_o$ to be either the central node or a node contained in the invariant leg of $\Gamma$, in which case it is clear that the minimality condition is symmetric with respect to the $\tau$-action, and thus each $x(i_t)$ is $\tau$-invariant.
Therefore, the induced action of $\tau$ on the graded root $R_\Gamma$ is trivial. Given a decomposition
\[
\gamma_J = \gamma_0 \cup \gamma_\Theta \cup J \gamma_0,
\]
the modified path
\[
\gamma_I = \gamma_0 \cup \gamma_\Theta \cup J \tau \gamma_0
\]
also carries the lattice homology of $(\Gamma_{p,q}, \mathfrak{s})$.

To ensure that $\gamma_I$ is the desired almost $I$-equivariant path carrying the lattice homology of $(\Gamma_{p,q}, \mathfrak{s})$, it remains to verify one final condition: if we denote by $\mathfrak{s}_1$ the first $\mathrm{spin}^c$ structure that appears in the path $\gamma_0$, then
\[
\mathfrak{s}_1 - \overline{\mathfrak{s}}_1 = \sum_{S \in \mathcal{S}} PD[S]
\]
for a finite collection $\mathcal{S}$ of pairwise disjoint, smoothly embedded, setwise $\tau$-fixed spheres with negative self-intersection numbers.  By construction, $\mathfrak{s}_1 - \overline{\mathfrak{s}}_1$ is the spherical Wu class of $W_{\Gamma_{p,q}}$. It then follows from the discussions above that the given condition is satisfied. The lemma follows.
\end{proof}

\begin{rem}
    It follows directly, at this stage, from the $\tau$-invariance of cycles $x(i)$ for each $i\ge 0$ that the action of $\tau$ on the Heegaard Floer chain complex $CF^-(\partial W_\Gamma)$ is homotopic to the identity. This is stronger than the observations made in \cite{alfieri2020connected} regarding the deck transformation action on $\Sigma_2(T_{p,q})\cong \Sigma(2,p,q)$, and thus might be of independent interest.
\end{rem}

\subsection{The real Fr\o yshov invariants of $T_{2n,1-20n}$}
From the observations we made in the previous subsection, we can prove the following theorem regarding real Fr\o yshov invariants of even torus knots.
\begin{thm}\label{thm:realsphere}
    If $K=T_{p,q}$ be a torus knot, where $p,q>0$ and $p$ is even, then we have
    \[
    \delta_R(K)=\underline{\delta}_R(K)=  \overline{\delta}_R(K)= -\frac{1}{2}\bar{\mu}(\Sigma_2(K)),
    \]
    where $\bar{\mu}$ is the Neumann-Siebenmann invariant for the unique spin structure. 
\end{thm}
\begin{proof}
    It follows from the construction in the previous section that the $I$-invariant locus of the $O(2)$-equivariant lattice homotopy type of the double-branched cover $\Sigma_2(K)$ of $K$ is the fixed locus of the ``central sphere'' under the complex conjugation action and thus given by $[(S^0,0,  \bar{\mu}(\Sigma_2(K)))]$ as a $\mathbb{Z}_2$-homotopy type. 
    More precisely, we see 
  \[
   \mathcal{H}(\gamma,\mathfrak{s})^I =( \Gamma_0 \cup \Gamma_{\theta})^I = \Gamma_{\theta}^I = \mathbb{E} ( e_{\mathfrak{s}_{-1}, \mathfrak{s}_{+1}} )^I = \left( \left( \C^{ \frac{1}{8}\left(c_1^2(\mathfrak{s}_{-1}) - \sigma (W_\Gamma)\right)} \right)^+ \right)^I \wedge \left\{ \frac{1}{2} \right\}. 
   \] 
    Note that the $\mathrm{spin}^c$ structure corresponding to the definition of $\bar{\mu}$ invariant is $\fraks_1$ in the previous section. We also note that $c_1(\fraks_1)^2 = c_1(\fraks_{-1})^2$, which follows from the definition of an almost $I$-invariant path.
    Thus, 
     \[
   \mathcal{H}(\gamma,\mathfrak{s})^I  = \left(\left(\C^{- \overline{\mu}(\Sigma_2(K)) }  \right)^+ \right)^I  =  \left(\R^{- \overline{\mu}(\Sigma_2(K)) }  \right)^+  = \left[\left(S^0, 0, \overline{\mu}(\Sigma_2(K)) \right)  \right] \in \mathfrak{C}_{\Z_2}. 
   \] 
    Since both lattice homotopy types and Seiberg--Witten homotopy types are finite $\Z_2$-spectra--with $\Z_2$ acting by $I$ on the former and by $I = j\circ \wt{\tau}$ on the latter--it follows from \Cref{main computation} and \Cref{lem:keylemma1} that $SWF^I(K)$ is a $\mathbb{Z}_2$-homology sphere of dimension $-\bar{\mu}(\Sigma_2(K))$. Therefore we deduce from \Cref{lem:keylemma2} that
    \[
    \delta_R(K)=\underline{\delta}_R(K)=  \overline{\delta}_R(K)= -\frac{1}{2}\bar{\mu}(\Sigma_2(K)),
    \]
    as desired.
\end{proof}

Using \Cref{thm:realsphere}, we can compute the $\overline{\delta}_R$ invariant for the torus knot $T_{2n,1-20n}$.

\begin{cor}\label{cor:realinvs}
    Let $n\ge 1$ be an odd integer. Then we have
    \[
    \overline{\delta}_R(T_{2n,1-20n}) = -\frac{9}{8}. 
    \]
\end{cor}
\begin{proof}
    Using \Cref{thm:realsphere} and computations from \Cref{subsec:topinvs}, we see that
    \[
    \underline{\delta}_R(T_{2n,20n-1}) = -\frac{1}{2}\bar{\mu}(\Sigma_2(K)) = \frac{9}{8}.
    \]
    By \cite[Lemma 3.28]{KMT:2023}, we deduce that
    \[
    \overline{\delta}_R(T_{2n,1-20n}) = -\underline{\delta}_R(T_{2n,20n-1}) = -\frac{9}{8}.\qedhere
    \]
\end{proof}

\subsection{Proof of \Cref{thm:main}}
We can now prove the main theorem, using our computations of real Fr\o yshov invariants of $T_{2n,1-20n}$.
\begin{proof}[Proof of \Cref{thm:main}]
In order to make use of \Cref{Theorem B for links}, we have to check that its assumptions are satisfied. Consider the smooth concordance, as described in \Cref{prop:cableof41totorusknot}, $S_n$ from $E_{2n,1}$ to $T_{2n,1-20n}$ in a twice-punctured $2\mathbb{CP}^2$, which is denoted by $X$. This concordance has the homology class $(2n,6n)$. We calculate:
\[
\begin{split}
    b_2^+(\Sigma_2(S_n)) - b_2^+(X) 
    &= b_2^+(X) -\frac{1}{4} [S_n]^2 + \frac{1}{2} \sigma (T_{2n,1-20n})\\
    &= 2-\frac{1}{4}\left((2n)^2+(6n)^2\right)+\frac{1}{2} (20n^2-2) \\
    &= 2-10n^2+(10n^2-1) \\&= 1.
\end{split}
\]
Hence the assumptions are satisfied, and as before we get
\[
 \underline{\delta}_R(E_{2n,1}) -\frac{1}{16}\left(  2 \sigma (X) - \frac{1}{2} [S_n]^2 + \sigma (T_{2n,1-20n})\right) \leq \bar{\delta}_R(T_{2n,1-20n}).
\]
Using
\[
\begin{split}
    -\frac{1}{16}\left(  2 \sigma (X) - \frac{1}{2} [S_n]^2 + \sigma (T_{2n,1-20n})\right) &= -\frac{1}{16} \left( 2\cdot 2-\frac{1}{2}\left((2n)^2+(6n)^2\right) + (20n^2-2)  \right) \\
    &= -\frac{1}{8},
\end{split}
\]
and \Cref{cor:realinvs}, we conclude that $$ \underline{\delta}_R(E_{2n,1}) \leq -1.$$ 
The proof is complete by applying Corollary~\ref{cor:negativeimmersedpoints}.
\end{proof}

\begin{rem}
As it is observed in \cite[Proposition 4.9]{KMT:2023}, the map sends a knot concordance class to the $(G=\Z_4, H=\Z_2)$ local equivalence class of the real Floer homotopy type giving a homomorphism: 
    \[
[K] \mapsto [SWF_R(K)]_{\mathrm{loc}}\colon   \mathcal{C} \to \mathcal{LE}_G.   
\]
We have observed that for any torus knot $T_{p, q}$, we have 
$
[SWF_R(T_{p, q} )]_{{\mathrm{loc}} }$ is equal to some sphere spectrum. It is also true for any two-bridge knot. 
In other words, all torus knots are sent to the subgroup in $\mathcal{LE}_G$ generated by sphere spectrums. Note that the invariants $\delta_R, \underline{\delta}_R, \overline{\delta}_R$ factor through the group homomorphism $SWF_R \colon  \mathcal{C} \to \mathcal{LE}_G $. Moreover, one can use the fact that $E_{2,1}$ bounds nullhomologous disks in both $\mathbb{CP}^2$ and $\overline{\mathbb{CP}}^2$, and apply \cite[Theorem 3.23]{KMT:2023} to conclude that $\delta_R(E_{2n,1}) = 0$. Then we have $\delta_R( E_{2n,1}) > \underline{\delta}_R ( E_{2n,1})$ and  $SWF^R_{\mathrm{loc}} (E_{2n,1})$ is not equal to some sphere spectrum in the group $\mathcal{LE}_G$. 
It is already observed in \cite[Example 1.11]{KMT:2023} that certain Montesinos knots also satisfy this property.
\end{rem}

\subsection{Arborescent knots}\label{sec:Montesino}
Let $\Gamma$ be a negative-definite almost rational plumbing graph. Recall that the associated plumbed 4-manifold $W_\Gamma$ is defined by gluing together the total spaces of disk bundles $p_v\colon E_v \rightarrow S^2$ whose Euler number is equal to the weight of the node $v$, where $v$ runs over all nodes in $\Gamma$. For each disk bundle $p_v$, we consider the complex conjugation action $\tau_v$, which acts by reflection on the base $S^2$ and also on each fiber of $p_v$. Clearly, $\tau_v$ is an orientation-preserving smooth involution on the total space of $p_v$ for each node $v$, and these local involutions can be glued together to obtain an orientation-preserving smooth involution $\tau$ on $W_\Gamma$. It is straightforward to observe that $W_\Gamma / \tau \cong S^3$, and the projection map $W_\Gamma \rightarrow S^3$ is a 2-fold branched covering, whose branching locus is a knot in $S^3$. Knots arising in this way are called \emph{arborescent knots}. We note that, when $\Gamma$ is a plumbing graph obtained via the Seifert algorithm from the Brieskorn sphere $\Sigma(a_1, \ldots, a_n)$, defined by
\[
\Sigma(a_1, \dots, a_n) = \left\{ (z_1, \dots, z_n) \in \mathbb{C}^n \,\middle\vert\, B \begin{pmatrix} z_1^{a_1} \\ \vdots \\ z_n^{a_n} \end{pmatrix} = 0, \quad |z_1|^2 + \cdots + |z_n|^2 = 1 \right\} \subset S^{2n-1}
\]
for any $(n-2)$-by-$n$ complex matrix $B$ whose maximal minors are nonzero,\footnote{This definition is independent of $B$, up to diffeomorphism.} the action of $\tau$ on $\partial W_\Gamma$ coincides with the complex conjugation action, i.e.,
\[
(z_1, \dots, z_n) \mapsto (\bar{z}_1, \dots, \bar{z}_n).
\]

Let $\Gamma$ be a negative-definite AR-graph whose corresponding boundary involution is given by complex conjugation on the Brieskorn sphere $\Sigma(a_1, \dots, a_n)$, which can be realized as the double branched cover of a Montesinos knot. In this case, all paths are strict $I$-invariant in the sense that every $\mathrm{spin}^c$ structure contained in any path is $I$-invariant. In fact, every $\mathrm{spin}^c$ structure on $W_\Gamma$ is $I$-invariant, as first observed in~\cite{alfieri2020connected}. Hence, unlike the case of almost $I$-invariant paths, we may simply take any path $\gamma$ that carries the lattice homology. We will construct an $O(2)$-equivariant map
\[
\mathcal{T} \colon \mathcal{H}(\gamma, \mathfrak{s}_0) \to SWF(\Sigma_2(K)).
\]

In the context of arborescent  knots, we define a class of $O(2)$-actions on the path homotopy type:
\[
\mathcal{H}(\gamma, \mathfrak{s}) = \bigcup_{1 \leq i \leq m} \mathbb{S}(\mathfrak{s}_i) \cup \bigcup_{1 \leq i \leq m - 1} \mathbb{E}(e_{\mathfrak{s}_i, \mathfrak{s}_{i - 1}}).
\]
We define the involution $I$ on the spheres and edges by complex conjugation:
\[
I \colon \mathbb{S}(\mathfrak{s}_i) \to \mathbb{S}(\mathfrak{s}_i), \qquad
I \colon \mathbb{E}(e_{\mathfrak{s}_i, \mathfrak{s}_{i - 1}}) \to \mathbb{E}(e_{\mathfrak{s}_i, \mathfrak{s}_{i - 1}}).
\]
This defines a well-defined $O(2)$-action on $\mathcal{H}(\gamma, \mathfrak{s})$.

Let $\Gamma$ be an almost rational negative-definite plumbing graph, $K$ be the associated arborescent knot, and $\mathfrak{s}$ be the unique spin structure on $\Sigma_2(K)$. As in the case of torus knots, for each vertex $\fraks_i$ of a path $\gamma$ which carries the lattice homology of $(\Gamma,\mathfrak{s})$, we associate the corresponding Bauer--Furuta invariant 
\[
BF_{W_\Gamma, \fraks_i} \colon \mathbb{S}(\fraks_i) \to \Sigma^{\frac{h}{2}\mathbb{C} } SWF(Y, \fraks)
\]
with stabilizations by ${\R}, {\wt{\R}} $ and ${\C}$, which is $O(2)$-equivariant.
For the maps corresponding to edges, we use the $O(2)$-adjunction relation stated below.

The following is the $O(2)$-adjunction relation, which can be regarded as an $O(2)$-equivariant version of  \cite[Proposition 3.15]{DSS2023}: 

\begin{prop}\label{adjunction I ver}
    Let $(X, \fraks)$ be a $\mathrm{spin}^c$
cobordism from $(Y_0, s_0)$ to $(Y_1, \fraks_1)$ with $b_1(X)=
b_1(Y_i)=0$. Suppose there is a smooth involution $\tau$ on $X$ such that 
\[
\tau^* \fraks \cong \overline{\fraks}. 
\]
Suppose that we have an embedded sphere $S$ in $X$ with $S\cdot S < 0$ and with $\tau(S) =S$ so that $\tau|_S \colon S \to S$ is the complex conjugation on $\C\mathbb{P}^1$.   
Let $L$ be
the complex line bundle on $X$ with  $c_1(L)= PD(S)$.
Set 
\[
\fraks' = \fraks\otimes L\qquad\text{ and } \qquad
n:= \frac{\langle c_1(\fraks) ,[S]\rangle +[S]^2}{2}.
\]
We write the $O(2)$-equivariant Bauer--Furuta invariants of $\fraks$ and $\fraks'$ 
as maps  
\begin{align*}
   &  BF_{X,\fraks }   \colon \left(\C^{\frac{c_1(\fraks)^2 - \sigma(X)}{8}} \right)^+ \wedge SWF(Y_0)  \to SWF(Y_1) \\
      & BF_{X,\fraks' }  \colon \left(\C^{\frac{c_1(\fraks')^2 - \sigma(X)}{8}} \right)^+\wedge SWF(Y_0)  \to SWF(Y_1). 
\end{align*}
Then, $BF_{X,\fraks } \text{ and } U^n BF_{X,\fraks' }$ are $O(2)$-stably homotopic up to certain coordinate changes if $n>0$, and the same statement holds for $U^{-n}BF_{X,\fraks } \text{ and }  BF_{X,\fraks' }$ if $n<0$.
Here $U$ denotes the stable homotopy class of a map
\[
X \to \Sigma^{\C } X 
\]
obtained as $x \mapsto (0, x)$.
\end{prop}

\begin{rem}\label{coordinate change}
    The meaning of ``up to certain coordinate changes" in \cref{adjunction I ver} is the following: if we need, after precomposing an odd permutation 
    \[
   (z_1, z_2, z_3,  \dots, z_n) \mapsto (z_2, z_1, z_3,  \dots, z_n) \colon  \C^n \to \C^n,  
    \]
    the maps $BF_{X,\fraks } \text{ and } U^m BF_{X,\fraks' }$ are $O(2)$-equivariantly stably homotopic. 
\end{rem}
The proof is similar to that given in the proof of \cite[Proposition 3.15]{DSS2023}. The only difference is: we need to analyze the Bauer--Furuta invariants for $I$-fixed point parts in our $O(2)$-setting. 

\begin{proof}[Proof of \cref{adjunction I ver}]
    
We first decompose $X$ into 
\[
X = \nu (S) \cup  \left(X \smallsetminus \operatorname{int} \nu (S)\right)
\]
$\tau$ equivariantly, where $\nu(S)$ is the disk normal bundle of $S$ identified with a tubular neighborhood of $S$.
Then, the equivariant version of the gluing theorem implies 
\begin{align}\label{equ gluing}
BF_{X, \fraks } =BF_{X \smallsetminus \operatorname{int} \nu (S), \fraks |_{X \smallsetminus \operatorname{int} \nu (S)}} \circ BF_{\nu(S), \fraks|_{\nu(S)} } . 
\end{align}
This follows from the gluing theorem proven by Miyazawa in \cite[Theorem 2.12]{Mi23}.

Since the involution $\tau$ preserves the standard positive scalar curvature metric on the lens space $\partial \nu(S)$, one can regard 
$BF_{\nu(S), \fraks|_{\nu(S)} }$ as an $O(2)$-equivariant map 
\[
BF_{\nu(S), \fraks|_{\nu(S)} } \colon V^+ \to W^+
\]
for some $O(2)$-representation spaces. 
Since $\fraks'$ and $\fraks$ are the same on $X \smallsetminus \operatorname{int} \nu (S)$ and we have \eqref{equ gluing}, it is sufficient to give an $O(2)$ homotopy between 
\[
BF_{\nu(S), \fraks|_{\nu(S)} } \colon V^+ \to W^+ \qquad\text{ and }\qquad U^m \circ BF_{\nu(S), \fraks'|_{\nu(S)} }\colon V^+ \to W^+ 
\]
which are maps between spheres when $m\geq 0$. The case $m <0$ follows from completely the same argument. 

We will prove these maps $BF_{\nu(S), \fraks|_{\nu(S)} }$ and $U^m BF_{\nu(S), \fraks'|_{\nu(S)} }$ are $O(2)$-stably homotopic to  $O(2)$-equivariant maps obtained from the inclusions
\[
\iota\colon \C^n \hookrightarrow  \C^{n+l}
\]
when $l>0$ and  
\[
\pm \id \colon \C^n \to \C^n
\]
when $l=0$ up to certain coordinate changes.  Here we are using $\nu(S)$ is negative-definite. 

We shall use the equivariant version of Hopf's classification result stated in \cref{TD} to make an $O(2)$-homotopy between $BF_{\nu(S), \fraks|_{\nu(S)} }$ and $BF_{\nu(S), \fraks'|_{\nu(S)} }$. 
We have two cases: 
\begin{itemize}
    \item The case of $\dim V^I <  \dim W^I$ and 
    \item The case of $\dim V^I = \dim W^I$.
\end{itemize}
In the first case, we only need to see 
\[
\deg BF_{\nu(S), \fraks|_{\nu(S)} }^{S^1} = \deg  BF_{\nu(S), \fraks'|_{\nu(S)} }^{S^1}. 
\]
This is obvious since $S^1$-invariant part of the Bauer--Furuta invariant does not depend on the choices of $\mathrm{spin}^c$ structures. In the second case, we will prove 
\[
\deg BF_{\nu(S), \fraks|_{\nu(S)} }^{I} =\pm 1,  
\]
which is a non-trivial computation. Note that, in the second case, the corresponding $\mathrm{spin}^c$ structure $\mathfrak{s}$ on $\nu(S)$ satisfies
\[
c_1^2(\mathfrak{s}) - \sigma(\nu(S)) = 4  d\big(-\partial \nu(S) = -L(p, 1),\, \mathfrak{s}|_{-L(p, 1)}\big)
\]
for some integer $p$, where $d$ denotes the Heegaard Floer $d$-invariant~\cite[Section~4]{OzSz03b}. This condition is equivalent to having a sharp Fr{\o}yshov inequality. In this case, the $O(2)$-Bauer--Furuta invariant can be written as
\[
BF_{\nu(S), \mathfrak{s}|_{\nu(S)}} \colon  
\left(\C^{\frac{c_1^2(\mathfrak{s}) - \sigma(\nu(S))}{8}}\right)^+ 
\to 
\left(\C^{\frac{d\left(L(p, 1),\, \mathfrak{s}|_{L(p, 1)}\right)}{2}}\right)^+.
\]
In order to compute the degrees, we use the following key lemma:

\begin{lem}\label{equivariant emb}
Let $p$ be a negative integer, and let $O(p)$ denote the total space of the disk bundle over $S^2$ with Euler number~$p$. Define an involution $\tau \colon O(p) \to O(p)$ as complex conjugation on both the base and fiber directions. Let $\mathfrak{s}$ be a $\mathrm{spin}^c$ structure on $O(p)$ satisfying $\tau^* \mathfrak{s} \cong \overline{\mathfrak{s}}$, and suppose that
\[
c_1^2(\mathfrak{s}) - \sigma(O(p)) = -4  d\left(\partial O(p) = L(p, 1),\, \mathfrak{s}|_{L(p, 1)}\right).
\]
Then, there exists an equivariant embedding
\[
O(p) \hookrightarrow \#_{-p} \overline{\mathbb{CP}}^2,
\]
extending the $\mathrm{spin}^c$ structure $\mathfrak{s}$, where the action on $\#_{-p} \overline{\mathbb{CP}}^2$ is the connected sum of the complex conjugations and
\[
c_1(\mathfrak{s}) = (\pm 1, \dots, \pm 1).
\]
Here, the $\pm$ signs need not be synchronized.
\end{lem}

\begin{proof}
    Let $U$ be the unknot, so that attaching a $2$-handle along $U$ to $B^4$ with framing $p$ yields $O(p)$. Choose a strong inversion of $U$ and denote its rotation axis by $\ell$. Let $m_1, \ldots, m_{-p-1}$ be $0$-framed parallel copies of the meridian of $U$, arranged so that the rotation along $\ell$ induces a strong inversion on each $m_i$.

   We then attach $(-1)$-framed $2$-handles to $O(p)$ along each $m_i$, and cap off the resulting manifold with a $4$-handle. Let $W_p$ denote the resulting closed $4$-manifold. Since we are attaching $2$-handles along each component of a strongly invertible link, the involution $\tau$ extends smoothly to an involution on $W_p$. Furthermore, by performing equivariant blowdowns, we obtain
\[
W_p \cong \#_{-p} \overline{\mathbb{CP}}^2,
\]
where the diffeomorphism is $\tau$-equivariant.

    To prove the statement about extensions of $\mathrm{spin}^c$ structures, we recall that $O(p)$, considered as a cobordism from $L(-p, 1)$ to $S^3$, is negative-definite. It then follows from~\cite[Section~9]{OzSz03b} that the Heegaard Floer cobordism map
\[
F^-_{O(p), \mathfrak{s}} \colon HF^-\big(L(-p, 1),\, \mathfrak{s}|_{L(-p, 1)}\big) \to HF^-(S^3) = \mathbb{Z}_2[U]
\]
becomes a homotopy equivalence after localizing $U^{-1}$. 
Using the degree shift formula in Heegaard Floer homology~\cite[Section~2]{OzSz03b}, one sees that the degree shift  is
\[
\deg F^-_{O(p), \mathfrak{s}} = \frac{c_1(\mathfrak{s})^2 + 1}{4} = \frac{c_1(\mathfrak{s})^2 - \sigma(O(p))}{4},
\]
which, by assumption, is equal to $-d\big(L(p, 1),\, \mathfrak{s}|_{L(-p, 1)}\big)$. 
Since $d(S^3) = 0$ and $L(-p, 1)$ is an L-space, we deduce that $F^-_{O(p), \mathfrak{s}}$ is an isomorphism. Consequently, the hat-flavored cobordism map
\[
\widehat{F}_{O(p), \mathfrak{s}} \colon \widehat{HF}\big(L(-p, 1),\, \mathfrak{s}|_{L(-p, 1)}\big) \to \widehat{HF}(S^3)
\]
is also an isomorphism.

    It is easy to see, via explicit holomorphic triangle counts on Heegaard triple diagrams, that for any $n > 1$ and any $\mathrm{spin}^c$ structure $\mathfrak{s}_0$ on $L(n,1)$, the canonical negative-definite cobordism $X_{n-1}$ from $L(n-1,1)$ to $L(n,1)$ (given by attaching a $(-1)$-framed $2$-handle to a meridian of an $(n-1)$-surgered unknot) admits a $\mathrm{spin}^c$ structure $\tilde{\mathfrak{s}}_0$ extending $\mathfrak{s}_0$, such that the hat-flavored cobordism map
\[
\hat{F}_{W_{n-1,1}, \tilde{\mathfrak{s}}_0} \colon 
\widehat{HF}\left(L(n-1,1),\, \tilde{\mathfrak{s}}_0|_{L(n-1,1)}\right) \to 
\widehat{HF}\left(L(n,1),\, \mathfrak{s}_0\right)
\]
is an isomorphism. By induction on $-p$, this implies that there exists a $\mathrm{spin}^c$ structure $\mathfrak{s}_p$ on the cobordism
\[
W_p \smallsetminus \left(O(p) \sqcup \mathring{B}^4\right) \cong 
X_{-p-1} \cup_{L(-p-1,1)} X_{-p-2} \cup_{L(-p-2,1)} \cdots \cup_{L(2,1)} X_1
\]
such that the cobordism map
\[
\hat{F}_{W_p \smallsetminus O(p), \mathfrak{s}_p} \colon 
\widehat{HF}(S^3) \to 
\widehat{HF}\left(L(-p,1),\, \mathfrak{s}|_{L(-p,1)}\right)
\]
is an isomorphism. Composing this with $\hat{F}_{O(p), \mathfrak{s}}$, we obtain
\[
\hat{F}_{W_p, \tilde{\mathfrak{s}}_p} \colon 
\widehat{HF}(S^3) \to \widehat{HF}(S^3),
\]
which is an isomorphism, where $\tilde{\mathfrak{s}}_p = \mathfrak{s} \cup \mathfrak{s}_p$ is the induced $\mathrm{spin}^c$ structure on $W_p$.
Now, since $W_p \cong \#_{-p} \overline{\mathbb{CP}}^2$, the $\mathrm{spin}^c$ structures on $W_p$ are classified by their first Chern classes. If we write
\[
c_1(\tilde{\mathfrak{s}}_p) = (\lambda_1, \dots, \lambda_{-p}),
\]
where we are choosing the generators of $H^2\left(\#_{-p}\overline{\mathbb{CP}}^2;\mathbb{Z}\right)$ to be our choice of basis for $H^2(W_p;\mathbb{Z})$, then
\[
\hat{F}_{W_p, \tilde{\mathfrak{s}}_p} = 
\hat{F}_{\overline{\mathbb{CP}}^2, \mathfrak{s}_{\lambda_1}} \circ \cdots \circ 
\hat{F}_{\overline{\mathbb{CP}}^2, \mathfrak{s}_{\lambda_{-p}}},
\]
where $\mathfrak{s}_{\lambda_i}$ denotes the unique $\mathrm{spin}^c$ structure on $\overline{\mathbb{CP}}^2$  whose $c_1$ is $\lambda_i$.
It is straightforward to verify, again via holomorphic triangle counts, that the map
\[
\hat{F}_{\overline{\mathbb{CP}}^2, \mathfrak{s}_{\lambda_i}} \colon 
\widehat{HF}(S^3) \to \widehat{HF}(S^3)
\]
is an isomorphism if $\lambda_i$ generates 
$H^2\left(\overline{\mathbb{CP}}^2; \mathbb{Z}\right)$, and is zero otherwise. Therefore, in order for $\hat{F}_{W_p, \tilde{\mathfrak{s}}_p}$ to be an isomorphism, we must have
\[
c_1\left(\tilde{\mathfrak{s}}_p\right) = (\pm 1, \dots, \pm 1).
\]
Since $\tilde{\mathfrak{s}}_p$ extends the given $\mathrm{spin}^c$ structure $\mathfrak{s}$ on $O(p)$, the lemma follows.
\end{proof}




Using \cref{equivariant emb}, we have an equivariant embedding $f \colon \nu(S) \to \#_n \overline{\mathbb{CP}}^2$ for some $n>0$.
Again, from $O(2)$-equivariant gluing formula of the Bauer--Furuta invariants, we have 
\[
BF_{\nu(S), \fraks|_{\nu(S)} } \circ BF_{\#_n \overline{\mathbb{CP}}^2 \smallsetminus \operatorname{int}f(\nu(S)) , \fraks_0|_{\#_n \overline{\mathbb{CP}}^2 \smallsetminus \operatorname{int}f(\nu(S))}  } = BF_{\#_n \overline{\mathbb{CP}}^2, \fraks_0} 
\]
up to $O(2)$-equivariant stable homotopy.  Here $\fraks_0$ denotes the $\mathrm{spin}^c$ structure on $\#_n \overline{\mathbb{CP}}^2$ such that $c_1(\fraks_0) = ( \pm 1, \dots, \pm 1)$.
By an equivariant version of the connected sum formula of $O(2)$-equivariant Bauer--Furuta invariant \cite[Theorem 2.12]{Mi23}, we see 
\[
\deg BF_{\#_n \overline{\mathbb{CP}}^2, \fraks_0}  = \left(\deg BF_{\overline{\mathbb{CP}}^2, \fraks_0|_{\overline{\mathbb{CP}}^2} }\right)^n. 
\]

Let $\tau_{\overline{\mathbb{CP}}^2}$ denote the complex conjugation. Then this preserves the standard positive scalar curvature metric on $\overline{\mathbb{CP}}^2$. So, one can see 
\[
\deg BF^I_{\overline{\mathbb{CP}}^2, \fraks_0|_{\overline{\mathbb{CP}}^2}}  =\pm 1,
\]
which is stated in \cite[Theorem 1.9, the third item]{Mi23}.

Thus, we have 
\[
\deg \left(BF_{\nu(S), \fraks|_{\nu(S)} }^I \right) \cdot \deg \left( BF_{\#_n \overline{\mathbb{CP}}^2 \smallsetminus \operatorname{int}f(\nu(S)) , \fraks_0|_{\#_n \overline{\mathbb{CP}}^2 \smallsetminus \operatorname{int}f(\nu(S))}  } ^I\right) = \deg \left(BF_{\#_n \overline{\mathbb{CP}}^2, \fraks_0} ^I \right)=\pm 1 . 
\]
Thus, one can see 
\[
\deg \left(BF_{\nu(S), \fraks|_{\nu(S)} }^I \right) =\pm 1. 
\]

Note that the base change 
    \[
   (z_1, z_2, z_3,  \dots, z_n) \mapsto (z_2, z_1, z_3,  \dots, z_n) \colon  \C^n \to \C^n
    \]
    changes the sign of the mapping degree $BF_{\nu(S), \fraks|_{\nu(S)} }^I$. Thus, if necessary, after composing it, one can confirm that 
    \[
\deg \left(BF_{\nu(S), \fraks|_{\nu(S)} }^I \right) =1. 
\]
Therefore, up to sign, from \cref{TD}, we see $BF_{\nu(S), \fraks|_{\nu(S)} }$ and $U^m BF_{\nu(S), \fraks'|_{\nu(S)} }$ are $O(2)$-stably homotopic. This completes the proof. 
\end{proof}

\begin{proof}[Proof of \Cref{thm:Montesinos}]
Let $Y = \Sigma_2(K)$. For a fixed, strictly $I$-invariant path in the arborescent knot cases, and for each vertex $\fraks_i$,  we associate the corresponding Bauer--Furuta invariant 
\[
BF_{W_\Gamma, \fraks_i}\colon \mathbb{S}(\fraks_i) \to \Sigma^{\frac{h}{2}\mathbb{C} } SWF(Y, \fraks)
\]
with stabilizations by ${\R}, {\wt{\R}} $ and ${\C}$, which is $O(2)$-equivariant.

Note that $\fraks_i - \fraks_{i-1}$ can be represented by $PD(S)$, where $S$ is the connected sum of certain $2$-handle cores having negative self-intersections. Moreover, from the construction of involution on the graph $4$-manifold, we see the $2$-handle cores are preserved by the involution and it reverses an orientation of each $2$-handle core. 
Therefore, we can apply \cref{adjunction I ver} to $\fraks_i$ and $\fraks_{i-1}$ to obtain an $O(2)$-equivariant homotopy $H$ after composing a base change if necessary. This gives an $O(2)$-equivariant map 
\[
H\colon \mathbb{E}(e_{\fraks_i, \fraks_{i-1}}) \to \Sigma^{\frac{h}{2} \C}  SWF(Y)
\]
which gives a well-defined $O(2)$ equivariant map 
\[
\mathcal{T}^{O(2)}\colon \mathcal{H}(\gamma,\mathfrak{s}_0) \to SWF(Y). 
\]
From the construction, if we forget $I$ action, it is nothing but the construction of the original $S^1$-equivariant map given in \cite{DSS2023}, which is $S^1$-homotopy equivalence. 
\end{proof}

We now prove \Cref{cor:Montesinos}. Note that its proof relies on \Cref{deg comp}, which will be proven in \Cref{subsec: miyazawa}.

\begin{proof}[Proof of \Cref{cor:Montesinos}]
    We recall the process of drawing a graded root (up to overall grading shift, for simplicity) $R$ from a (finite) path $\gamma$ carrying the lattice homology of $(\Gamma,\mathfrak{s})$. Write $\gamma = \{ \mathfrak{s}_1,\ldots,\mathfrak{s}_n\}$, where every $\mathfrak{s}_i$ restricts to $\mathfrak{s}$ on $Y_\Gamma$, and choose characteristic vectors $k_i$ that represent $\mathfrak{s}_i$. Then we consider the sequence
    \[
    k_1^2,\ldots,k_n^2.
    \]
    Let $i_1,\ldots,i_m$ be the indices where the sequence achieves a local maximum. Then, for each $s=1,\ldots,m-1$, we consider the subsequence
    \[
    k_{i_s}^2,k_{i_s+1}^2,\ldots,k_{i_{s+1}}^2;
    \]
    this sequence admits a global minimum, at an index which we denote as $j_s$, such that $k_t^2 \ge k_{j_s}^2$ for any $t=i_s,\dots,i_{s+1}$. Then $R$ is a graded root which consists of leaves $v_1,\dots,v_m$, with $\mathrm{gr}(v_t)=k_{i_t}^2$, and angles $w_1,\dots,w_{m-1}$ between the leaves (where $w_t$ lies between $v_t$ and $v_{t+1}$), with $\mathrm{gr}(w_t)=k_{j_t}^2$.

    Now we calculate the Euler characteristic of the fixed point locus. Since the Euler characteristic can be computed using $\mathbb{Z}_2$-coefficient homology, it follows from \Cref{lem:keylemma1} that
\[
\chi\left(SWF(\Sigma_2(K), \mathfrak{s})^I\right) = \chi\left(\mathcal{H}(\gamma, \mathfrak{s})^I\right).
\]
Recall that $\mathcal{H}(\gamma, \mathfrak{s})^I$ can be constructed combinatorially as follows. For each $t = 1, \dots, m$, we consider the sphere $\mathcal{S}_t$ of dimension $\frac{k_{i_t}^2}{2}$; we then form their bouquet
\[
\mathcal{S} = \mathcal{S}_1 \vee \cdots \vee \mathcal{S}_m.
\]
Next, for each $t = 1, \dots, m-1$, we consider the cylinder $\mathcal{C}_t = S^{\frac{k_{j_t}^2}{2}} \wedge [0,1]$. We attach its ``boundary" $$\mathcal{C}_j^\partial = S^{\frac{k_{j_t}^2}{2}} \vee S^{\frac{k_{j_t}^2}{2}}$$ to $\mathcal{S}$; we do not need to know exactly how it is attached. The resulting pointed space, considered as a spectrum via $\Sigma^\infty$, is homotopy equivalent to $\mathcal{H}(\gamma, \mathfrak{s})^I$.

Clearly, we have
\[
\chi(\mathcal{H}(\gamma, \mathfrak{s})^I) = \tilde\chi(\mathcal{S}) + \sum_{t=1}^{m-1} \tilde\chi(\mathcal{C}_j) - \sum_{t=1}^{m-1} \tilde\chi(\mathcal{C}_j^\partial) = \sum_{t=1}^m \tilde\chi(\mathcal{S}_i) + \sum_{t=1}^{m-1} \left( \tilde\chi(\mathcal{C}_j) - \tilde\chi(\mathcal{C}_j^\partial) \right).
\]
Here, $\tilde\chi$ denotes the reduced Euler characteristic, i.e., $\tilde\chi(X) = \sum_{n \geq 0} \dim_{\mathbb{Z}_2} \tilde{H}_n(X; \mathbb{Z}_2)$ for spaces $X$ of finite type. Since $\tilde\chi(S^n) = (-1)^n$, we deduce that
\[
\chi(\mathcal{H}(\gamma, \mathfrak{s})^I) = \sum_{t=1}^m (-1)^{\frac{k_{i_t}^2}{2}} - \sum_{t=1}^{m-1} (-1)^{\frac{k_{j_t}^2}{2}}.
\]
It then follows from the choice of indices $i_1, \ldots, i_m$ and $j_1, \ldots, j_{m-1}$ and \Cref{deg comp} that
\[
\left\vert \deg(K) \right\vert = \left\vert \chi\left(\mathcal{H}(\gamma, \mathfrak{s})^I\right) \right\vert = \left\vert \sum_{t=1}^m (-1)^{\frac{k_{i_t}^2}{2}} - \sum_{t=1}^{m-1} (-1)^{\frac{k_{j_t}^2}{2}} \right\vert = \left\vert \sum_{v \in L(R)} (-1)^{\frac{\mathbf{gr}(v)}{2}} - \sum_{v \in A(R)} (-1)^{\frac{\mathbf{gr}(v)}{2}} \right\vert,
\]
as desired.
\end{proof}

\subsection{Examples of $| \chi ( SWF_R(K)) |$}
We give several concrete examples of the computation of $| \chi ( SWF_R(K)) |$ from \Cref{cor:Montesinos}. It is observed in \cite[Proof of Lemma 3.28]{KMT:2023} that $SWF_R(K)$ and $SWF_R(-K)$ are $V$-dual, where $-K$ denotes the mirror of $K$ and $V$ is some vector space. Therefore, we have 
\[
| \chi ( SWF_R(K)) |  =| \chi ( SWF_R(-K)) |. 
\]
Thus, we do not need to care about the convention of knots about the mirrors here. 

\begin{ex}
    Consider the plumbing graph\vspace{.3cm}
    \[\Gamma:=
\begin{tikzpicture}[xscale=1.5, yscale=1, baseline={(0,-0.1)}]
    \node at (-.3, 0) {$-1$};
    \node at (1.3, 0) {$-3$};
    \node at (1.3, 1) {$-2$};
    \node at (1.3, -1) {$-7$};
    \node at (0, 0) (A0) {$\bullet$};
    \node at (1, 0) (A1) {$\bullet$};

    \node at (1, 1) (B1) {$\bullet$};
    \node at (1, -1) (C1) {$\bullet$};
    
    \draw (A0) -- (B1);
    \draw (A0) -- (C1);
    \draw (A0) -- (A1);
    
\end{tikzpicture}
\]
 \vspace{.2cm}

\noindent Then $Y_\Gamma$ is the double-branched cover of the pretzel knot $K=P(2,-3,-7)$. We will present a path of $\mathrm{spin}^c$ structures on $W_\Gamma$, presented in terms of homology classes in $H_2(W_\Gamma;\mathbb{Z})$, which carries the lattice homology of $(Y_\Gamma,\mathfrak{s})$, where $\mathfrak{s}$ denotes the unique $\mathrm{spin}^c$ structure on $Y_\Gamma$.



    We will use the following notation: classes in $H_2(W_\Gamma;\mathbb{Z})$ are represented as quadruples $x=(a,b,c,d)$. This would mean that $x$ is the sum
    \[
    x=a[S_{-1}]+b[S_{-2}]+c[S_{-3}]+d[S_{-7}],
    \]
    where $S_{-n}$ denotes the node of $\Gamma$ whose self-intersection is $-n$. This setting is a bit different from the one that we used in the proof of \Cref{cor:Montesinos}, and thus the weight functions are defined differently. In fact, in this setting, the weight function is defined as
    \[
    w(x)=x^2 + k\cdot x,
    \]
    where $k=(0,1,1,1)$ is the spherical Wu class. 

    Now we consider the path 
    \[
    \gamma = \{ (-1,-1,-1,-1),(0,-1,-1,-1),(0,0,-1,-1),(0,0,0,-1),(0,0,0,0),(1,0,0,0) \}.
    \]
    The sequence of weights are then given by
    \[
    w(\gamma) = \{ 2,0,0,0,0,2 \}.
    \]
    it is then easy to see that $\gamma$ carries the lattice homology of $( Y_\Gamma,\mathfrak{s})$. In fact, a careful reader can observe that $\gamma$ is actually an almost $J$-invariant path in the sense of \cite[Definition 6.2]{DSS2023}. From this data, we see that the $S^1$-equivariant lattice Floer homotopy type $\mathcal{H}(\gamma)$ (which is the same as the $S^1$-equivariant Seiberg--Witten homotopy type of $\Sigma(2,3,7)$) is given by $S^2 \cup_{S^0} S^2$. Note that, since $\Sigma(2,3,7)$ is a homology sphere, it has only one $\mathrm{spin}^c$ structure, and thus we are dropping $\mathrm{spin}^c$ structures from our notations.

    To see the $O(2)$-action on this homotopy type, we observe that $S^2$ and $S^0$ are actually given in terms of compactifications of $S^1$-representations as follows:
    \[
    S^2 = \left(\mathbb{C}^1\right)^+\qquad\text{ and }\qquad S^0 = \left(\mathbb{C}^0\right)^+.
    \]
    The $I$-action on complex representations are given by the complex conjugation, so we see that 
    \[
    SWF_R(P(-2,3,7))=\mathcal{H}(\gamma)^I \simeq S^1 \cup_{S^0} S^1 \simeq S^1 \vee S^1 \vee S^1,
    \]
    and thus $\vert SWF_R(P(-2,3,7)) \vert = 3\cdot \vert \chi(S^1) \vert = 3$. Note here that we take $\chi(S^1)=1$, as we are considering $S^1$ as a graded spectrum $\Sigma^\infty S^1$ and thus we are computing the Euler characteristic of its reduced homology.
    
    For a sanity check, we will also use \Cref{thm:Montesinos} and check that we get the same result. From the sequence of weights of lattice points on the given path $\gamma$, we see that the associated graded root is given as follows.
    \vspace{.3cm}
    \[
    	\begin{tikzpicture}[scale=.8]
	\draw [fill=black] (0,0) circle (1.5pt);
	\draw [fill=black] (-.5,1) circle (1.5pt);
        \draw [fill=black] (.5,1) circle (1.5pt);
	\draw [fill=black] (0,-1) circle (1.5pt);
        \draw[fill=black] (0,-1.3) node{$\vdots$};
        \draw[thick] (-.5,1)--(0,0)--(.5,1);
        \draw[thick] (0,0)--(0,-1);
	\end{tikzpicture}
    \]
    This graded root has three vertices, among which two of them are leaves. The leaves lie in degree 2, while the non-leaf vertex, which has only one angle, lies in degree 0. Hence we see that \Cref{thm:Montesinos} also gives the same result:
    \[
    \left\vert \chi\left( SWF_R(P(-2,3,7)) \right) \right\vert = \vert (-1) + (-1) - 1 \vert = 3.
    \]
\end{ex}

\begin{ex}
    Instead of the pretzel knot $P(-2,3,7)$, we now consider the Montesinos knots $K_n$ given by negative-definite AR plumbing graphs of $\Sigma(2,3,n)$, where $n\ge 7$ and $n$ is relatively prime to $6$. In this case, one can use the computation of the $S^1$-equivariant Seiberg-Witten Floer homology of their double-branched covers $\Sigma_2(K_n) = \Sigma(2,3,n)$, which was already done in \cite[Section 7.2]{Ma07} to determine the graded root, and then use it to compute the value of $\left\vert \chi\left(SWF_R(K_n)\right)\right\vert$. 

    For simplicity, we will only present two cases: $n=12k-5$ and $n=12k+1$ for $k>0$. In the case $n=12k-5$, which also covers the case of $P(2,-3,-7)$, the graded root is given as follows.
    \vspace{.3cm}
    \[
    	\begin{tikzpicture}[scale=.8]
	\draw [fill=black] (0,0) circle (1.5pt);
	\draw [fill=black] (-.5,1) circle (1.5pt);
        \draw [fill=black] (.5,1) circle (1.5pt);
	\draw [fill=black] (-1.5,1) node{$\cdots$};
	\draw [fill=black] (1.5,1) node{$\cdots$};
	\draw [fill=black] (-2.5,1) circle (1.5pt);
	\draw [fill=black] (2.5,1) circle (1.5pt);
        
	\draw [fill=black] (0,-1) circle (1.5pt);
        \draw[fill=black] (0,-1.3) node{$\vdots$};
        \draw[thick] (-.5,1)--(0,0)--(.5,1);
        \draw[thick] (0,0)--(0,-1);
        \draw[thick] (0,0)--(-2.5,1);
        \draw[thick] (0,0)--(2.5,1);
	\end{tikzpicture}
    \]
    It has $2k$ leaves in some even degree, which we consider to be at degree 2 after a suitable degree shift, and $2k-1$ angles in degree 0. Hence we have
    \[
    \left\vert \chi\left(SWF_R(K_n)\right) \right\vert = 4k-1.
    \]
    On the other hand, if $n=12k+1$, then the graded root looks like the following.

    \[
    	\begin{tikzpicture}[scale=.8]
	\draw [fill=black] (0,0) circle (1.5pt);
	\draw [fill=black] (0,1) circle (1.5pt);
        \draw [fill=black] (-1,1) circle (1.5pt);
        \draw [fill=black] (1,1) circle (1.5pt);
	\draw [fill=black] (-2,1) node{$\cdots$};
	\draw [fill=black] (2,1) node{$\cdots$};
	\draw [fill=black] (-3,1) circle (1.5pt);
	\draw [fill=black] (3,1) circle (1.5pt);
        
	\draw [fill=black] (0,-1) circle (1.5pt);
        \draw[fill=black] (0,-1.3) node{$\vdots$};
        \draw[thick] (0,1)--(0,0)--(0,-1);
        \draw[thick] (-1,1)--(0,0)--(1,1);
        \draw[thick] (-3,1)--(0,0)--(3,1);
	\end{tikzpicture}
    \]
    It has $2k+1$ leaves in degree 2 (after a degree shift) and $2k$ angles in degree 0. Hence we get
    \[
    \left\vert \chi\left(SWF_R(K_n)\right) \right\vert = 4k+1.
    \]
   The remaining cases can be dealt with similarly, and so we omit them.

\end{ex}

\begin{ex}
    Let $\Gamma$ be a negative-definite AR plumbing graph, $W_\Gamma$ be the associated smooth 4-manifold, and $K$ be the associated arborescent knot. Since we know from \Cref{thm:Montesinos} that the computation of the Euler characteristic $\left\vert \chi(SWF\left(\Sigma_2(K),\mathfrak{s})^I\right)\right\vert$ depends only on the graded root of $(\Gamma,\mathfrak{s})$ for any $\mathrm{spin}^c$ structure $\mathfrak{s}$ on $Y_\Gamma=\Sigma_2(K)$, their computations are now easy even in much more complicated cases.

    We will give model computations for four additional cases, $Y_1,Y_2,Z_1,Z_2$, defined as follows.
    \[
    Y_1 = \Sigma(3,5,7),\quad Y_2 = \Sigma(5,8,13),\quad Z_1 = \Sigma(3,4,11),\quad Z_2=\Sigma(5,7,17).
    \]
    They are Seifert manifolds and thus admit canonical plumbing graphs which we denote as $\Gamma_{Y_1},\Gamma_{Y_2},\Gamma_{Z_1},\Gamma_{Z_2}$. We will denote their associated Montesinos knots as $K_{Y_1},K_{Y_2},K_{Z_1},K_{Z_2}$. Also, $Y_1,Y_2,Z_1,Z_2$ are all homology spheres, so they only have one $\mathrm{Spin}^c$-structures; hence we will drop them from our notations. The computation of their graded roots are given in \cite[Figures 8 and 9]{karakurt2022almost}. Applying \Cref{thm:Montesinos} then tells us the following.
    \[
        \left\vert \chi\left(SWF_R(K_{Y_1})\right) \right\vert =  
        \left\vert \chi\left(SWF_R(K_{Y_2})\right) \right\vert =  
        \left\vert \chi\left(SWF_R(K_{Z_1})\right) \right\vert = 
        \left\vert \chi\left(SWF_R(K_{Z_2})\right) \right\vert = 1.
    \]
\end{ex}

\section{Concluding remarks}
\subsection{Calculations on Miyazawa's invariant} \label{subsec: miyazawa}

Miyazawa considered the mapping degree of the $\{\pm 1\}$-framed real Bauer--Furuta invariants  
\[
| \deg(S) |\in \Z/\{\pm 1\} =\Z_{\geq 0 } \qquad \text{ and } \qquad 
 |\deg(P) |\in \Z/\{\pm 1\} =\Z_{\geq 0 }
\]
for given a $2$-knot $S$ in $S^4$ and a given $\mathbb{RP}^2$-knot $P$ in $S^4$. We recall the following theorem, proven in \cite{Mi23}: 
\begin{thm}\label{Miya}
    Let $K$ be a knot in $S^3$ with determinant one and $k, l$ be integers. We denote by $\tau_{k,\alpha}(K)$ the $k$-twisted $\alpha$-roll twisted spun knot in $S^4$. 
    If $\frac{k}{2}+ \alpha$ is odd, then we have 
    \begin{align}\label{surgery formula}
    |\deg (\tau_{k,\alpha}(K))| = |\deg(K)|, 
      \end{align}
    where $\deg(K)$ denotes the absolute value of the sign counting of the $(\pm 1)$-framed real Seiberg--Witten moduli space for $\Sigma_2(K)$ with the unique spin structure. 
\end{thm}
For the definition of twisted roll spun 2-knots, see \cite[Section 1]{Pl84}. We shall rewrite the left-hand side of \eqref{surgery formula} in terms of real Seiberg--Witten Floer homotopy type of knots $SWF_R(K)$. 

\begin{prop}\label{deg comp}
For a knot $K$ in $S^3$, we have $|\deg (K) | = \left|\chi \left(SWF_R(K)\right)\right|$. 
\end{prop}
\begin{proof}
The proof needs a comparison between the critical point set of infinite-dimensional Morse functional and that of a finite-dimensional approximation of the functional. Basically, the analysis we need to do is similar to the arguments done in \cite[Section 7 and 9]{LM18}, although we only need to focus on critical point sets, not trajectories. Also, we do not need to consider the blow up of the configuration space since we forcus on the counting of framed moduli spaces.  Such a comparison needs a careful analysis and the discussions rely on the compactness of the Seiberg--Witten equation. In our situation, we are just taking a fixed point part with respect to $I \in O(2)$, so such a compactness is still true. Thus, we will not repeat their argument here, instead, we write a sketch of the proof.

We first see the precise definition of the degree invariant $\deg (K)$. With respect to the unique spin structure on the double-branched cover $\Sigma_2(K)$ with a $\Z_2$-invariant Riemannian metric on $\Sigma_2(K)$, we have the $O(2)$-invariant Chern--Simons Dirac functional on a global slice; 
 \[
CSD \colon \mathcal{C}_K := \left(i\ker d^ *\subset  i\Om^1_{\Sigma_2(K)} \right)  \oplus \Gamma (\mathbb{S}) \to \R. 
\]
Then, we consider the induced function on the fixed point set: 
\[
CSD^I \colon \mathcal{C}_K^I  := (i\ker d^ *)^I   \oplus \Gamma (\mathbb{S})^I  \to \R.
\]
We have an action of constant gauge transformations $\{ \pm 1\}$.  
Now, we take a perturbation that comes from cylinder functions
\[
f : \mathcal{C}_K^I \to \R 
\]
such that all critical points of $CSD^I + f$ are non-degenerate, i.e. the Hessians on the critical point sets are invertible. 
The existence of such a perturbation is proven in \cite[7.4. Proof of transversality]{Li23}. 
After the perturbation, we can assume there are the unique reducible critical point $[(a_0, 0)]$ has stabilizer $\pm 1$, and the set of the other finite irreducible critical points have a free $\Z_2$ action comes from $[(a, \phi )] \to [(a, -\phi )]$.
 We also fix an orientation of a fiber of the determinant line bundle $\operatorname{det} ( \operatorname{Ker} d (CSD^I + f)_{[(a_0, 0)]} )$ corresponding to the reducible $[(a_0, 0)]$. 
Induced from this orientation, we can define the absolute value of the signed counting of all critical points of $CSD^I + f$, which is denoted by $\deg (K)$. Since $\deg (K)$ is a counting of the $\{ \pm 1 \}$-framed moduli space with respect to a fixed $\Z_2$-Riemannian metric and a perturbation, $\deg (K)$ is independent of the choices of a $\Z_2$-invariant metric and a non-degenerate perturbation. Also, since $\deg(K)$ denotes the absolute value, $\deg (K)$ does not depend on the choices of an orientation of the determinant line bundle.

Now, we relate $\deg (K)$ with $\left|\chi \left(SWF(\Sigma_2(K))^I\right)\right|$. The spectrum $SWF(\Sigma_2(K))^I $ was defined by taking the $\langle I\rangle$-fixed point part of the Seiberg--Witten Floer homotopy type $SWF(\Sigma_2(K), \fraks_0)$, again $\fraks_0$ denotes the unique spin structure on $\Sigma_2(K)$. Alternatively, we can describe $SWF(\Sigma_2(K))^I$ as the Conley index of a finite-dimensional approximation of the flow with respect to the vector field $ \operatorname{grad} CSD^I$. Let us say this construction briefly.
Define $V^{\lambda}_{-\lambda}(K)\oplus W^{\lambda}_{-\lambda}(K) \subset \mathfrak{C}_K^I $ to be the direct sums of the eigenspaces of the linear part of $\operatorname{grad} (CSD^I + f) $ whose eigenvalues are in $(-\lambda, \lambda]$, where $V^{\lambda}_{-\lambda}(K)$ is the eigenspace corresponding to the space of 1-forms and $W^{\lambda}_{-\lambda}(K)$ is the eigenspace corresponding to spinors. Then we restrict the perturbed Chern--Simons Dirac functional $CSD^I + f$ to $V^{\lambda}_{-\lambda}(K) \oplus W^{\lambda}_{-\lambda}(K)$. If we take $\lambda$ sufficiently large, the set of critical points of $CSD^I + f$ in $\mathfrak{C}_K^I$ is contained in $V^{\lambda}_{-\lambda}(K) \oplus W^{\lambda}_{-\lambda}(K)$ and the Hessians of the restricted function 
\[
(CSD^I + f)|_{V^{\lambda}_{-\lambda}(K) \oplus W^{\lambda}_{-\lambda}(K)} \colon V^{\lambda}_{-\lambda}(K) \oplus W^{\lambda}_{-\lambda}(K) \to  \R
\]
on each critical point is invertible. 
Note that a comparison between the critical point sets of the infinite-dimensional setting and a finite-dimensional Morse setting is given in \cite[Corollary 7.1.5, 
Corollary 7.2]{LM18} in $S^1$-monopole Floer setting. \footnote{Since they treat blown-up of a finite-dimensional approximation and comparison between Morse chain complexes. In our situation, we are just counting $\{\pm1\}$-framed critical points, we do not need to consider the blow-up configuration space.   }
A similar analysis enables us to see there is no other critical point of $(CSD^I + f)|_{V^{\lambda}_{-\lambda}(K) \oplus W^{\lambda}_{-\lambda}(K)}$ if we take $\lambda$ sufficiently large. 

Now, we consider the gradient flow with respect to $\rho \operatorname{grad}(CSD^I + f)$ on $V^{\lambda}_{-\lambda}(K) \oplus W^{\lambda}_{-\lambda}(K)$, where $\rho$ is a cut-off function appeared as in the case of the construction of the usual Seiberg--Witten Floer homotopy type. Then, one can prove this flow has an isolated invariant neighborhood, which is a big ball in $V^{\lambda}_{-\lambda}(K) \oplus W^{\lambda}_{-\lambda}(K)$, again it is assumed to contain all critical points of $CSD^I + f$.
Then, the Conley index of the vector field $\rho \operatorname{grad}(CSD^I + f)$ is described by a CW complex which has a handle decomposition coming from the Morse handle decomposition with respect to $CSD^I + f$. 
Therefore, it is not hard to see the Euler number of the Conley index is equal to the signed counting of the critical point set of $CSD^I + f$ restricted to $V^{\lambda}_{-\lambda}(K) \oplus W^{\lambda}_{-\lambda}(K)$. (See \cite[Theorem 2.4.3]{LM18}.)
Thus, it is sufficient to see the sign coming from an orientation of the determinant line bundle and the sign comes from the Morse index with respect to $CSD^I + f$ are the same. This sign is equivalent to whether relative grading is odd or even with respect to the relative $\Z$-grading.
Therefore, it is a comparison between the Morse index in the infinite-dimensional setting and the Morse index in a finite-dimensional approximation. 
In \cite[Corollary 9.1.3]{LM18}, such comparisons between the two degrees are given in the usual $ S^1$-monopole Floer setting. 
A similar argument without essential change enables us to see the relative gradings in the infinite-dimensional setting and a finite-dimensional setting are the same. This completes the sketch of a proof.  \end{proof}

Now, from \cref{deg comp}, our result gives combinatorial computations of $|\deg (\tau_{k,\alpha}(K)) |$, described in \cref{Miya_torus} and  \cref{cor:Montesinos}.  


On the other hand, for the standard $P_0=\mathbb{RP}^2$ whose double cover is $\overline{\mathbb{CP}}^2$, we have 
\[
|\deg (P_0)|=1
\]
and the connected sum formula
\[
|\deg (P \# S )| = |\deg (P)| \cdot | \deg (S)| \qquad \text{ and }\qquad |\deg (S \# S' )| = |\deg (S)| \cdot |\deg (S')|
\]
for general $\mathbb{RP}^2$ knot whose double-branched cover has $b_2^+=0$ and $2$-knots $S$ and $S'$, which are again proven in \cite{Mi23}. 

\begin{cor}\label{monte general deg}
 Let $\Gamma$ be a negative-definite AR-graph, $W_\Gamma$ be the associated plumbed 4-manifold with boundary $Y_\Gamma$, and consider the corresponding arborescent knot $K$. Let $\gamma$ be a path which carries the lattice homology of $(\Gamma,\mathfrak{s})$ for any $\mathrm{spin}^c$ structure $\mathfrak{s}$ on $Y_\Gamma$. Suppose that the lattice homology of $(\Gamma,\mathfrak{s})$ is expressed as a graded root $R$, and the determinant of $K$ is one. 
 Denote the sets of leaves and non-leaf vertices of $R$ by $L(R)$ and $NL(R)$, respectively, and shift the grading (if necessary) so that all vertices of $R$ lie on even degrees. Furthermore, we suppose 
    \[
 \left\vert \sum_{v\in L(R)} (-1)^{\frac{\mathbf{gr}(v)}{2}} - \sum_{v\in NL(R)} (-1)^{\frac{\mathbf{gr}(v)}{2}} \right\vert \neq 1. 
    \]
    Then for integers $k, \alpha$ such that $\frac{k}{2}+ \alpha$ is odd, the $k$-twisted $\alpha$-roll twisted spun knot $\tau_{k , \alpha }(K) \# P_0  $ and $P_0$ are not smoothly isotopic.
\end{cor}

\begin{rem}
As observed in \cite[Theorem 4.47]{Mi23}, $\tau_{k , \alpha }(K) \# P_0$ and $P_0$ have non-diffeomorphic complements for $k = 0$, $\alpha = 1$, and $K = P(-2, 3, 7)$. Note that the same proof works in more general situations once we can ensure 
\[
\deg (\tau_{k , \alpha }(K)) > 1.
\]
Under the same assumptions in \cref{monte general deg}, we see that the complements of $\tau_{k , \alpha }(K) \# P_0$ and $P_0$ in $S^4$ are not diffeomorphic.
\end{rem}




\subsection{Structual theorem of an $O(2)$-equivariant Bauer--Furuta invariant}

\begin{proof}[Proof of \cref{structual o2}]

    For a given $2$-knot or $\mathbb{RP}^2$-knot $S$ in $S^4$, we consider its double-branched covering space $\Sigma_2(S)$. We assume $b_2^+(\Sigma_2(S) ) =0$ for the $\mathbb{RP}^2$-knot case. Then, we take the unique spin structure on $\Sigma_2(S)$ when $S$ is 2-knot and the $\mathrm{spin}^c$ structure $\fraks$ such that $c_1(\fraks)^2 =-1$ when $S$ is $\mathbb{RP}^2$-knot. Associated to it, we have an $O(2)$-equivariant map 
    \[
    BF_{\Sigma_2(S), \fraks} \colon W^+  \to V^+. 
    \]
    with respect to the above spin or $\mathrm{spin}^c$ structure $\fraks$. 
One can easily check that $W$ is isomorphic to $V$ as $O(2)$-representation spaces and the $O(2)$-equivariant stable homotopy class of $ BF_{W, \fraks}$ is an invariant of smooth isotopy classes of $2$-knots or such $\mathbb{RP}^2$ knots. If we take $\langle I \rangle \subset O(2)$-invariant part of $ BF_{\Sigma_2(S), \fraks}$, we recover the Miyazawa's invariant $\deg (S)$ as the mapping degree of $ BF_{\Sigma_2(S), \fraks}^I$. 
Such a homotopy class is determined by two quantities 
\[
\deg \left( BF_{\Sigma_2(S), \fraks}^I\right) \qquad \text{ and }\qquad  \deg \left( BF_{\Sigma_2(S), \fraks}^{S^1}\right)
\]
by \cref{TD}. The latter one is $+1$ if we take a standard homology orientation. The first one is nothing but Miyazawa's invariant. The sign ambiguity corresponds to composing the permutation 
  \[
   (z_1,z_2,z_3 \dots, z_n) \mapsto (z_2, z_1,z_3, \dots, z_n) \colon \C^n \to \C^n.   
    \]
    This completes the proof. 
\end{proof}

\appendix

\section{$O(2)$-representations and $O(2)$-equivariant maps}

\subsection{$O(2)$ representations in our setting}

We first see which representations of $O(2)$ appear in our situation. 

\begin{lem}\label{O2rep}
Consider the Lie group $O(2)$, and identify its identity component with $U(1)$. Choose an order two element $I \in O(2)$ such that $O(2)$ is generated by $U(1)$ and $I$. Let $\rho \colon O(2) \to GL_\mathbb{R}(V)$ be a representation of $O(2)$, where $V = \mathbb{C}^n$ and $U(1)$ acts on $V$ via $\rho$ by complex multiplication. Then the action of $\rho(I)$ is complex conjugation up to base change.
\end{lem}

\begin{proof}
    Since $O(2)$ is compact, its finite-dimensional representations over $\mathbb{R}$ decompose into a direct sum of irreducible representations up to base change. The list of all irreducible representations of $O(2)$ is described below (the proof is straightforward and thus omitted).
    \begin{itemize}
        \item 1-dimensional trivial representation $\mathbb{R}$;
        \item 1-dimensional flip representation $\tilde{\mathbb{R}}$, where $O(2)$ acts through $\pi_0(O(2))$, which then acts on $\mathbb{R}$ by $\pm 1$;
        \item 2-dimensional reprsentations $\mathbb{C}_q$, indexed by positive integers $q$, where $I$ acts on $\mathbb{C}\cong \mathbb{R}^2$ by complex conjugation and $U(1)$ acts by the $q$-fold rotation. (When $q=1$, we denote $\mathbb{C}_1$ by $\mathbb{C}$, as $I$ acts on it by complex conjugation)
    \end{itemize}
    Hence $V$ decomposes into direct sums of several copies of $\mathbb{R}$, $\tilde{\mathbb{R}}$, and $\mathbb{C}_q$ for $q>0$. Observe that, among the irreducible representations of $O(2)$, the only one which induces a free action (outside the origin) of $U(1)$ is $\mathbb{C}$. Since $U(1)$ acts freely on $V\smallsetminus \{0\}$ via $\rho$ by assumption, we deduce that $V = \mathbb{C}^n$ as $O(2)$-representations, and thus the action of $\rho(I)$ is the complex conjugation.
\end{proof}

\begin{lem}\label{auto}
Let $\rho \colon O(2) \to GL_\mathbb{R}(V)$ be a representation of $O(2)$, where $V = \mathbb{C}^n$ and $U(1)$ acts on $V$ via $\rho$ by complex multiplication. Then, the set of automorphisms
\[
\{ f \in GL_\mathbb{R}(V) \mid f \rho = \rho f \}
\]
is identified with $GL(n, \mathbb{R})$.
\end{lem}

\begin{proof}
One can assume $\rho$ is the standard $O(2)$-action by a base change. First of all, $\mathbb{R}$-linearity and $U(1)$-commutativity give $\mathbb{C}$-linearity, and thus the given space is a subspace of $GL_n(\mathbb{C})$. An element $A$ of $GL_n(\mathbb{C})$ is $O(2)$-commutative if and only if it commutes with complex conjugation, i.e., $\overline{A z} = A \overline{z}$ for all complex vectors $z$. This is equivalent to saying that $\overline{A} = A$. Conversely, it is obvious that matrices in $GL_n(\mathbb{R})$ give $O(2)$-commutative automorphisms of $\mathbb{C}^n$.
\end{proof}

\subsection{Equivariant version of Hopf's classification theorem}
We review the equivariant version of Hopf's classification theorem written in \cite[Page 125]{DT87}, which was used to prove the existence of $O(2)$-equivariant map between the lattice homotopy type and the Seiberg--Witten Floer homotopy type. 

Let $V$ and $W$ be $O(2)$-representations. 
We denote by $ V^+$ and $W^+$ the one-point compactifications of $V$ and $W$. 
Suppose the possible isotropy groups of $V^+$ and $W^+$ are 
\[
\{ e\},\, S^1,\, \langle I \rangle,\, O(2) \subset O(2). 
\]
For each isotopy group $G \subset O(2)$, we have the fixed point spheres $(V^+)^G$ and $(W^+)^G$, whose dimensions are written by 
$n_V (G)$ and $n_W (G)$. 
Let us define the set $\Phi (V, W, O(2))$ of conjugacy classes of isotropy groups $G$ satisfying  
\[
n_V(G) = n_W(G)\qquad \text{ and } \qquad |WG| < \infty,
\]  
where $WG$ is the Weyl group given as $NG / G$. Here, $NG$ denotes the normalizer of $G$ in $O(2)$. Thus, in our situation (assuming that $V$ and $W$ are in our universe $\R^\infty \oplus \wt{\R}^\infty \oplus \C^\infty$), we have  
\[
\Phi (V, W, O(2)) \subset \{ S^1, \langle I \rangle , O(2) \},
\]  
where the notations $S^1$, $\langle I \rangle$, and $O(2)$ denote their conjugacy classes.  

We suppose the following conditions: 
\begin{itemize}
    \item[(I)] For any isotropy group $G \subset O(2)$, we get 
\[
n_V(G) \leq n_W(G). 
\]
\item[(II)] For any $G \in \Phi (V, W, O(2))$, the groups $\wt{H}^{n_V(G)} \left( (V^G)^+ \right)$ and $\wt{H}^{n_W(G)} \left( (W^G)^+ \right)$ are isomorphic as $WG$-modules.

\item[(III)] For any $(K) \in \Phi (V, W, O(2))$, 
we have 
\[
1 + \dim \left((V^+ )^{>K}\right) < n_V(K), 
\]
where $(V^+ )^{>K}$ denotes the set of $K$-fixed points in $(V^+ )$ whose isotropy groups are strictly larger than $K$.  
\end{itemize}
These assumptions (I) and (II) correspond to (ii) and (iv) in \cite[page 125, (4.10)]{DT87} respectively. The other assumptions (i) and (iii) \cite[page 125, (4.10)]{DT87} are obviously satisfied in our situation. 
The condition (III) corresponds to \cite[The assumption in (iv) of Theorem 4.11 in page 126]{DT87}.  Also, for (II), the possibilities of Weyl groups are the trivial or $O(2)/U(1) \cong \Z_2$, and the condition (II) is also obvious in our situation. 
The condition (III) is expressed as  
\[
1 < n_V(O(2)), \qquad 1 < n_V(S^1), \qquad \text{ and } \qquad 1 < n_V(\langle I \rangle), 
\] 
which can also be achieved in the $O(2)$-equivariant stable homotopy category.
Under these assumptions, the following is proven in \cite[page 126, Theorem 4.11]{DT87}:  
\begin{thm}\label{TD}
Under the assumptions above, two $O(2)$-equivariant continuous maps  
\[
f_0, f_1 \colon V^+ \to W^+
\]
are $O(2)$-equivariantly homotopic if $\deg \left(f_0^G \right) = \deg \left(f_1^G\right)$ for any $G \in \Phi (V, W, O(2))$.
\end{thm}

\bibliographystyle{alpha}
\bibliography{knotbib}

\end{document}